\documentclass[a4paper,10pt]{amsart}

\usepackage{amsthm,amssymb,amsmath,amsfonts,mathrsfs,amscd,graphicx}
\usepackage{extarrows}
\usepackage{mathdots}
\usepackage[utf8]{inputenc}
\usepackage[all]{xy}
\usepackage{latexsym}
\usepackage{longtable}
\usepackage{xcolor}
\usepackage{enumerate}
\usepackage{extarrows}
\usepackage[new]{old-arrows}
\usepackage{mathtools}
\usepackage{bbm}
\usepackage{tikz}
\usetikzlibrary{arrows}

\theoremstyle{plain}
\newtheorem{theorem}{Theorem}[section]
\newtheorem{corollary}[theorem]{Corollary}
\newtheorem{lemma}[theorem]{Lemma}
\newtheorem{proposition}[theorem]{Proposition}
\newtheorem{conjecture}[theorem]{Conjecture}

\theoremstyle{definition}
\newtheorem{definition}[theorem]{Definition}

\newtheorem{examplewr}[theorem]{Example}

\theoremstyle{remark}
\newtheorem{obswr}[theorem]{Observation}
\newtheorem{remarkwr}[theorem]{Remark}

\newenvironment{remark}{\begin{remarkwr}\begin{upshape}}{\end{upshape}\end{remarkwr}}
\newenvironment{example}{\begin{examplewr}\begin{upshape}}{\end{upshape}\end{examplewr}}

\DeclareMathOperator{\Span}{span}
\newcommand{\Ker}{{\mbox{Ker}}}

\newcommand{\res}{\mathrm{res}}

\DeclareMathOperator{\Hom}{Hom} 
\DeclareMathOperator{\End}{End}

\DeclareMathOperator{\Gal}{Gal}
\DeclareMathOperator{\rk}{rk}

\newcommand{\absgr}{\mathcal{G}}
\newcommand{\absel}{g}
\newcommand{\latt}{L}
\newcommand{\iso}{\mathrm{iso}}
\newcommand{\geom}{\mathrm{geom}}

\newcommand{\cgroup}{\mathscr{G}}

\newcommand{\legendre}[2]{\genfrac{(}{)}{}{}{#1}{#2}}

\newcommand{\sign}{\mathrm{sign}}
\newcommand{\rec}{\mathrm{rec}}
\newcommand{\Ev}{\mathscr{E}}
\newcommand{\lra}{\longrightarrow}
\newcommand{\Div}{\mathrm{Div^{\dagger}_{rq}}}
\newcommand{\rM}{\mathcal{M}_{\mathrm{rq}}}
\newcommand{\divmap}{\mathrm{div}}
\newcommand{\adm}{\mathrm{adm}}

\newcommand{\ord}{{\mathrm{ord}}}

\newcommand{\Gammao}{{\Gamma_{\!\circ}}}
\newcommand{\sD}{{\mathscr D}}

\newcommand{\cA}{\mathcal A}

\newcommand{\cH}{\mathcal H}

\newcommand{\cZ}{\mathcal Z}
\newcommand{\cO}{\mathcal O}

\newcommand{\Z}{\mathbb{Z}}
\newcommand{\F}{\mathbb{F}}
\newcommand{\Q}{\mathbb{Q}}
\newcommand{\R}{\mathbb{R}}
\newcommand{\C}{\mathbb{C}}
\newcommand{\A}{\mathbb{A}}
\newcommand{\I}{\mathbb{I}}
\newcommand{\G}{\mathbb{G}}
\newcommand{\D}{\mathbb D}

\newcommand{\GL}{\mathrm{GL}}
\newcommand{\Mat}{\mathrm{M}}
\newcommand{\SL}{{\mathrm{SL}}}
\newcommand{\Mp}{{\mathrm{Mp}}}
\newcommand{\PGL}{{\mathrm{PGL}}}
\newcommand{\PSL}{{\mathrm{PSL}}}
\newcommand{\SO}{{\mathrm{SO}}}
\newcommand{\Ch}{\mathrm{Ch}}

\newcommand{\sbullet}{{\mbox{\tiny\textbullet}}}

\newcommand{\mat}[4]{\left(\begin{array}{cc}#1&#2\\#3&#4\end{array}\right)}
\newcommand{\smallmat}[4]{\bigl(\begin{smallmatrix}#1&#2\\#3&#4\end{smallmatrix}\bigr)}

\newcommand{\Setminus}{\!-\!}
\newcommand{\Smallsetminus}{\!-\!}

\include{thebibliography}

\begin{document}

\title[Rigid meromorphic cocycles for orthogonal groups]
{Rigid meromorphic cocycles \\ for orthogonal groups}

\author{Henri Darmon, Lennart Gehrmann, and Michael Lipnowski}

\address{H.~D.: McGill University, Montreal, Canada}
\email{darmon@math.mcgill.ca}
\address{L.~G.: Universitat de Barcelona, Barcelona, Spain}
\email{gehrmann.math@gmail.com}
\address{M.~L.: Ohio State University, Columbus, USA}
\email{lipnowski.1@osu.edu}
\dedicatory{}

\begin{abstract}
  {\em Rigid meromorphic cocycles} are defined 
in the setting of orthogonal groups of arbitrary real signature 
and constructed in some instances
via a $p$-adic analogue of Borcherds' singular theta lift.
The {\em values} of rigid meromorphic cocycles
at special points of an associated $p$-adic symmetric space are
then conjectured to belong to class fields of suitable global \emph{reflex fields},
suggesting an eventual framework for explicit class field theory
beyond the setting of CM fields explored in the treatise of Shimura and Taniyama. 
\end{abstract}

\maketitle

\tableofcontents

\section*{Introduction}
Let $(V,q)$ be a non-degenerate quadratic 
 space over $\Q$ of dimension $n\geq 3$ with associated bilinear form $\langle\cdot,\cdot\rangle$, satisfying $\langle v,v\rangle = 2 q(v)$, and  
let   $G=\SO_{V}$  denote the special orthogonal 
  group of $V$, viewed as a reductive algebraic group over $\Q$.
The real Lie group $G(\R)$  is determined up to isomorphism 
by the signature $(r,s)$ of the real quadratic space $V_\R:= V\otimes \R$.  
 In the special case of signature $(r,2)$, arithmetic quotients attached to $G$ give  rise to a collection of Shimura varieties  endowed with a systematic supply of algebraic cycles of all possible codimensions which have been the object of extensive study in recent decades. These cycles conjecturally arise, notably, as Fourier coefficients of modular generating series with coefficients in the arithmetic Chow groups of orthogonal Shimura varieties, following a general program initiated by Kudla \cite{kudla-msri}.
 
 \medskip
 The present work proposes a conjectural framework for extending some aspects of the arithmetic theory to the setting of arbitrary real signature, in which case the Archimedean symmetric spaces associated with $G$ is not endowed with any complex structure unless $r = 2$ or $s = 2.$ This framework rests on the notion of {\em rigid meromorphic cocycles} for $p$-arithmetic subgroups of $G$, whose study is initiated in this paper.

\bigskip Fix a prime $p$ for which the $p$-adic quadratic space $V_{\Q_p}:= V\otimes \Q_p$ admits a \textit{self-dual} $\Z_p$-lattice $\Lambda$, i.e., the bilinear pairing $\langle\cdot,\cdot\rangle$ identifies $\Lambda$ with its $\Z_p$-dual, or,  equivalently, 
the induced $\F_p:=\Z/p\Z$-valued pairing on $\Lambda/p\Lambda$ is non-degenerate.
Furthermore, fix a $\Z[1/p]$-lattice  $\latt\subseteq V$ on which the quadratic form $q$ is $\Z[1/p]$-valued with associated
dual lattice
\[
\latt^{\#}=\{v\in V \ \vert \    \langle v,w\rangle\in \Z[1/p] \  \forall w\in \latt\} \supseteq \latt
\]
and \emph{discriminant module}
\[
\D_{\latt} := \latt^\#/ \latt.
\]

 \bigskip \noindent
For any field $K$ of  characteristic different from $2$,
 denote by $G(K)^{+}$
the kernel of the spinor norm
\begin{equation}
\label{eqn:spin-gen}
\mathrm{sn}_K\colon G(K)  \rightarrow  K^\times/ (K^\times)^2
\end{equation}
and let 
\[ \Gamma \subseteq \SO(\latt) \]
be a $p$-arithmetic congruence subgroup
that lies in the kernel of the real spinor norm.
It is a discrete subgroup of $G(\R)^{+}\times G(\Q_p)$.

\bigskip
Chapter  \ref{sec:symmetric-space}
attaches to $V$ certain Archimedean and  $p$-adic symmetric spaces,
denoted $X_\infty$ and  $X_p$ respectively.
The former is a real analytic Riemannian manifold of dimension $rs$
 equipped with an action of $G(\R)$, and the 
latter  has the structure of a rigid analytic space of dimension $n-2$
 on which $G(\Q_p)$ acts naturally.  The $p$-arithmetic group $\Gamma$   acts on both $X_\infty$ and $X_p$, and acts discretely on the product $X_\infty \times X_p$.

\bigskip
Chapter \ref{sec:divisors} describes special cycles on $X_\infty,$ special divisors on 
$X_p$, and the attendant structures that arise from them.
More precisely,
the symmetric space  $X_\infty$ is equipped with a systematic collection of special cycles having real codimension $s$, which were studied by Kudla and Millson (see \cite{KM}). 
These cycles admit $p$-adic avatars having codimension one in $X_p$ 
 and are referred to as {\em rational quadratic divisors}.
 Combining the real and $p$-adic objects leads to the study of so-called
 {\em Kudla--Millson divisors}, 
namely, special classes in the $s$-th cohomology of $\Gamma$
with values in a module $\Div(X_p)$ of rational quadratic divisors
  satisfying a suitable 
 ``local finiteness" condition which is spelled out in Chapter
 \ref{sec:divisors}.
Under the assumption that $\Gamma$ acts trivially on the discriminant module $\D_{\latt}$ there is---in analogy with the theory of Heegner divisors on orthogonal Shimura varieties---a Kudla--Millson divisor
\begin{equation}
\label{eqn:hkmdd} 
\sD_{m,\beta} \in H^s(\Gamma,\Div (X_p))\  \mbox{for each}\  m\in\Q_{>0}\ \mbox{and}\  
 \beta\in  {\mathbb D}_{\latt},
\end{equation}
admitting a representative 
whose values are supported on the rational quadratic 
divisors attached to vectors
$v\in \beta$ satisfying $q(v)=m$.
 
 \bigskip
 Chapter \ref{sec:rmc} 
introduces  the {\em rigid meromorphic cocycles} themselves. To this end, let
   $\rM^\times$ be the multiplicative group of 
rigid meromorphic functions on $X_p$ with divisor in $\Div(X_p)$,
 equipped with the  natural  $\Gamma$-module structure via left translation on the arguments:
\[
\gamma f(x) := f(\gamma^{-1} x), \qquad \mbox{ for } \gamma \in \Gamma, \ \ f \in \rM^\times, \ \ x \in X_p.
\]
A class in   $H^s(\Gamma, \rM^\times)$ is called a 
{\em rigid meromorphic cocycle}\footnote{While ``rigid meromorphic cohomology class" would have been a more accurate if slightly less euphonious name, he first author's original sin of misleading terminology \cite{DV1} has been perpetrated again for the sake of consistency.} for $\Gamma$ if its divisor is a Kudla--Millson divisor.
Theorem \ref{maindefinite2} and Theorem \ref{mainhyperbolic2} describe the construction of a $p$-adic variant of Borcherds' singular theta lift in certain special signatures - namely, in signature $(n,0)$ for $n$ arbitrary, in signature $(3,1),$ and in signature $(4,1)$ - having rigid meromorphic cocycles whose divisors are prescribed Kudla-Millson divisors.  This implies that rigid meromorphic cocycles exist in abundance, at least in the latter special signatures.

\bigskip
Chapter \ref{sec:special-points} defines a notion of \emph{special points} on the symmetric space $X_p$ (relative to the action of $G(\Q)$).
These special points are meant to generalise CM points on orthogonal Shimura varieties, but the tori that arise from their stabilisers are not compact at infinity unless $r =0$ or $s = 0.$ 
If $r\geq s$, one can make sense of the {\em value} of a 
 rigid meromorphic cocycle $J$ at a special point $x$, leading to an analytically defined, and a priori transcendental, numerical
quantity  $J[x]\in\C_p$.

\bigskip
The main  conjecture of this work, formulated in
Conjecture \ref{conj:main} of 
 Section \ref{sec:conjecture}, asserts that the quantities $J[x]$ are in fact
 {\em algebraic}.
 More precisely, $J[x]$ is predicted to  lie in a  class field  of an appropriate \emph{reflex field} attached to the special point $x$.  This conjecture  suggests a program to  extend  the scope of the   Shimura--Taniyama--Weil theory of complex multiplication to the setting of arbitrary, not necessarily CM, number fields.
More precisely, it suggests that the algebraicity properties of  CM points on orthogonal Shimura varieties attached to quadratic spaces of signature $(n,2)$ can be meaningfully extended to special points on $p$-adic symmetric spaces for orthogonal groups attached to quadratic spaces of {\em arbitrary} real signature.

\bigskip
Chapter \ref{sec:examples}
concludes with a discussion of rigid meromorphic cocycles in a few of the simplest concrete settings. 
When $V$ is positive definite, i.e., when $s=0$,
the group $\Gamma$ acts discretely on $X_p$, and,
when
$r=3$ or $4$, the quotient 
$\Gamma\backslash X_p$
can be identified with the $\C_p$-points of a Shimura curve or a quaternionic Hilbert--Blumenthal surface respectively. Rigid meromorphic cocycles, which belong to  $H^0(\Gamma, \rM^\times)$, define rational functions on these varieties, and special points $\Gamma\backslash X_p$  correspond to CM points.  
These considerations lead to a proof of Conjecture \ref{conj:main} when $s=0$ and $r\leq 5$. (See the upcoming work \cite{GeBa} for details.)  For this reason, Chapter \ref{sec:examples} focusses on the more subtle indefinite setting.

Section \ref{sec:example-21} discusses the first case  going  truly beyond the setting of complex multiplication,
where $V$ is a quadratic space of signature $(2,1)$ and  special points correspond to real quadratic elements of the Drinfeld $p$-adic upper half-plane. The study of this setting, which serves as the basic prototype for the general conjecture of this paper,
   was initiated in 
  \cite{DV1} and a  great deal of evidence, both experimental and theoretical, has been  amassed in its support  (cf.~for instance~\cite{dpv2}, \cite{Ge-quaternionic}, \cite{GMX}, \cite{DV2}).

Section \ref{sec:example-31} ends with a discussion of the first non-trivial setting  not covered by these
prior works, that of a quadratic space of signature $(3,1)$  where the associated $p$-arithmetic group is a congruence subgroup of the Bianchi group 
$\Gamma := \SL_2(\cO_K[1/p])$, with $K$ an imaginary quadratic field.
The resulting ``Bianchi cocycles" play the role of meromorphic Hilbert modular forms arising as Borcherds lifts of weakly holomorphic modular functions.
Their ``values" at special points of  $\Gamma\backslash (\cH_p \times \cH_p)$ belong  conjecturally to abelian extensions of quadratic extensions of real quadratic fields with a single complex place and can be envisaged as the counterpart of  CM values of Hilbert modular functions.

\bigskip
\noindent\textbf{Acknowledgements.}
The work on this article began while Lennart Gehrmann was visiting McGill University, supported by Deutsche Forschungsgemeinschaft, and he thanks both of these institutions. The other two authors were supported by NSERC Discovery grants.
The authors are grateful to the Centre de Recherches Math\'ematiques in  Montreal, which continued to host in-person activities   during  the 2020 thematic semester on number theory, at the height of the covid-19 pandemic.
Notable among these was a   ``$p$-adic Kudla seminar''  whose participants  provided a congenial 
and stimulating atmosphere for the  elaboration of this article. 
The authors would like to thank Tonghai Yang and Sören Sprehe for useful comments on an earlier draft of this manuscript.
More recently, Lennart Gehrmann has received funding from the Maria Zambrano Grant for the attraction of international talent in Spain.


\section{Symmetric spaces}
\label{sec:symmetric-space}

\subsection{Quadratic lattices}
Some notations and basic facts regarding lattices in quadratic spaces are collected in this section.

Fix a principal ideal domain $R$ with $\mathrm{char}(R)\neq 2$ and let $K$ be its field of fractions.
Let $(W,q_W)$ be a finite-dimensional non-degenerate quadratic space over $K$ with associated bilinear form $b_W(\cdot,\cdot)$, that is, $b_W(w,w)=2q(w)$ for all $w\in W$.
An $R$-lattice in $W$ is a finitely generated $R$-submodule $\Lambda\subseteq W$ of maximal rank.
The dual of an $R$-lattice $\Lambda\subseteq W$ is
\[
\Lambda^{\#}=\{w\in W\ \vert\ b_W(w,\Lambda)\subseteq R\}.
\]
The dual of an $R$-lattice is also an $R$-lattice. Moreover, the equality
\[
(\Lambda^{\#})^{\#}=\Lambda
\]
holds.
An $R$-lattice $\Lambda$ is contained in its dual if and only if $b_W(\Lambda,\Lambda)\subseteq R$.
In that case the \emph{discriminant module of $\Lambda$} is defined to be the quotient
\[
\D_\Lambda=\Lambda^{\#}/\Lambda.
\]
It is a finitely generated $R$-torsion module.
A lattice $\Lambda\subseteq W$ is called self-dual if $\Lambda=\Lambda^{\#}$.
Let $\Lambda\subseteq W$ be a self-dual lattice and let $\mathfrak{q}\subseteq R$ be a maximal ideal with $\mathrm{char}(R/\mathfrak{q})\neq 2$.
Then the reduction of the quadratic form $q_W$ defines a non-degnerate quadratic form on the quotient $\Lambda/\mathfrak{q}\Lambda$.

The following basic lemma describes the behaviour of self-dual lattices under orthogonal projection.
\begin{lemma}\label{lem-orthdecomp}
Let $(W,q_W)$ be a finite-dimensional non-degenerate quadratic space over $K$ with an orthogonal direct sum decomposition: $W=W_1 \bigoplus W_2$ with $W_1 \perp W_2.$
Denote the orthogonal projections by $\pi_i\colon W \rightarrow W_i$, $i=1,2$.
Let $\Lambda\subseteq W$ be self-dual lattice and put $\Lambda_i=\Lambda \cap W_i$, $i=1,2$.
Then the following holds:
\begin{enumerate}[(a)]
\item $\Lambda_i=\pi(\Lambda)^{\#}$ for $i=1,2$
\item $\pi_1(\Lambda)/\Lambda_1\cong \pi_2(\Lambda)/\Lambda_2$ as $R$-modules
\end{enumerate}
\end{lemma} 
\begin{proof}
Fix $w\in W_1$.
Since $\Lambda$ is self-dual, it follows that $w\in \Lambda_1$ if and only if $b_W(w,\lambda)\subseteq R$ for all $\lambda\in\Lambda$.
Write $\lambda=\pi_1(\lambda)+\pi_2(\lambda)$.
Then, by definition of the orthogonal projection we have
\[
b_W(w,\lambda)=b_W(w,\pi_1(\lambda)).
\]
Thus, it follows that $w\in \Lambda_1$ if and only if $w\in \pi_1(\Lambda)^{\#}$.
This proves the first claim.

For the second claim let $w_1$ be an element of $\pi_1(\Lambda)$.
By definition, there exists $w_1'\in \pi_2(\Lambda)$ such that $w_1+w_1'\in \Lambda$.
Moreover, the coset $w_1'+\Lambda_2$ is independent of the choice of $w_1'$.
In case $w_1\in \Lambda_1$ one may take $w_1'=0$.
Hence, the map
\begin{align*}
\pi_1(\Lambda)/\Lambda_1 &\longrightarrow \pi_2(\Lambda)/\Lambda_2 \\
w_1+\Lambda_1 &\longmapsto w_1'+\Lambda_2
\end{align*}
is well-defined.
It is easy to check that this defines an isomorphism of $R$-modules.
\end{proof}

\subsection{Archimedean symmetric spaces} \label{Archimedeansymmetricspace}
\subsubsection{Definition}
The real symmetric space attached to $V_\R$ 
is the manifold, denoted $X_\infty$, parametrizing maximal negative-definite subspaces of 
$V_\R$.
Note that $X_\infty$ is also in natural bijection with 
the set of maximal positive definite subspaces of $V_\R$, under the map $Z\mapsto Z^\perp$.

The real Lie group $G(\R)$ acts transitively on  $X_\infty$.
The stabiliser of a point  $Z\in X_\infty$ in $\mathrm{O}_V(\R)$ 
is  identified with the maximal compact subgroup
\[
\mathrm{O}_{Z^\perp}(\R) \times \mathrm{O}_Z(\R) \simeq \mathrm{O}(r)\times \mathrm{O}(s) \subset \mathrm{O}_V(\R).
\]
Hence the inverse of the map $ \gamma \mapsto \gamma Z$ gives an identification
\begin{equation}
\label{eqn:Xinfty-G}
 X_\infty  = \mathrm{O}_V(\R)/ (\mathrm{O}_{Z^\perp}(\R) \times \mathrm{O}_Z(\R)).
 \end{equation}
Since the orthogonal group of a non-degenerate quadratic form of rank $n$ has dimension $n(n-1)/2$, it follows 
from 
\eqref{eqn:Xinfty-G}
 that the dimension of the real manifold
$X_\infty$  is  given by
\[
\dim X_\infty = rs.
\]
Since the inclusion $O_{Z^\perp}(\R)\times O_Z(\R) \hookrightarrow O_V(\R)$ induces a surjection on the groups of connected components, it follows that the space $X_\infty$ is connected.
As $X_\infty$ is an open subset of the Grassmannian of $s$-dimensional subspaces in $V_{\R}$,
the tangent space to $X_\infty$ at $Z$ is canonically identified with
\begin{align}\label{eqn:tangent-space}
\begin{split}
 T_Z(X_\infty) &= {\rm Lie}(\GL(V_\R))/{\rm Lie}({\mathrm{Stab}_{\GL(V_\R)}}(Z))  \\ 
&\cong \Hom_\R(Z,V_\R/Z) \\
&\cong   \Hom_\R(Z,Z^\perp).
\end{split}
\end{align}
The penultimate isomorphism is obtained by restricting an endomorphism of $V_{\R}$ to $Z$ and composing it with the natural surjection $V_\R\rightarrow V_\R/Z$.

Every arithmetic subgroup $\Gammao\subseteq G(\Q)$ acts discretely on $X_\infty$ and the quotient space $\Gammao\backslash X_\infty$ has finite volume, and is compact if and only if $V$ is anisotropic.

\subsubsection{Functoriality of $X_{\infty}$}
More generally, given a finite-dimensional real non-degenerate quadratic space $(W,q_W)$ one defines $X_\infty(W,q_W)$ to be the space of maximal negative-definite subspaces of $W$.
Sending a maximal negative-definite subspace of $W$ subspace to its orthogonal complement yields an isomorphism between $X_\infty(W,q_W)$ and $X_\infty(W,-q_W)$.
Let $(W',q_W')$ be another real non-degenerate quadratic spaces such that the dimension of a maximal negative subspace of $(W,q_W)$ and $(W',q_W')$ agree. 
For any isometric embedding $f\colon (W,q_W) \hookrightarrow (W',q_W')$ the induced map
\begin{align*}
X_{\infty}(W,q_W) &\hookrightarrow X_{\infty}(W',q_W') \\
Z &\mapsto f(Z)
\end{align*} 
is an isometric embedding of Riemannian manifolds.
In particular, if $f$ is an isomorphism of quadratic spaces, then it induces an isometry $f\colon X_{\infty}(W,q_W) \simeq X_{\infty} (W',q_W')$ of symmetric spaces.\footnote{This explains why the orthogonal group $O(W,q_W)$ acts by isometries on the symmetric space $X_{\infty}(W,q_W).$} The latter observation implies that the symmetric space $X_{\infty}$ is isometric to $X_{\infty}(V_0,q_0)$ for any particular quadratic space $(V_0,q_0)$ over $\R$ having signature $(r,s)$.
In particular examples, it is sometimes convenient to work with a particular model quadratic space $(V_0,q_0)$.
In the below examples for small values of $s$, specializing to some particular models allows us to recover classical descriptions of certain well-known symmetric spaces.

\subsubsection{Examples of $X_{\infty}$ for small $s$} 
\text{}

\smallskip

\noindent
{\em Signature $(r,0)$}.
In case $V_\R$ is positive or negative definite, the group  $G(\R)$ is compact and the space $X_\infty$ consists of a single point.
 
 \medskip
\noindent
{\em Signature $(r,1)$}.
When $s=1$, the space $X_\infty$ is isometric to the set of points $(x_1:\cdots: x_r: x_{r+1})\in \mathbb{P}^r(\R)$  satisfying the inequality
\[ x_1^2 + \cdots + x_r^2 - x_{r+1}^2 <0,\]
and can be identified via stereographic projection with the open unit ball in $r$-dimensional euclidean space, endowed with its natural hyperbolic structure.

\medskip
\noindent
{\em Signature $(r,2)$}.
The case $s=2$ plays a special role in the theory. 
Firstly, it represents essentially the only setting where $X_\infty$ 
has the structure of a Hermitian symmetric domain, possessing a complex structure. Secondly,  the corresponding $X_\infty$ admits a convenient description which  motivates the definition of the $p$-adic symmetric space $X_p$ taken up in \S \ref{sec:def-Xp}.
 
Let $Q$ be the smooth quadric  over $\Q$ whose points over a field $K$ of characteristic zero are 
 identified with
the set of   isotropic lines in $V_K:=V\otimes_\Q K$:
\[
Q(K):=  \ \left\{ v\in V_K \Smallsetminus\{0\}\ \vert\ q(v) =0 \right\}/K^\times \  \subseteq  \ \mathbb{P}_V(K).
\] 
Letting $[v]$ denote the  class in $Q(K)$
of the non-zero isotropic vector $v\in V_K$,
we define
\begin{equation}
\label{eqn:Xinfty-r2}
	\widetilde{X}_\infty:= \{[v]\in Q(\C)\ \vert\ \langle v, \overline v \rangle < 0 \},
\end{equation}
	where $v\mapsto \overline v$ denotes the complex 
	conjugation on $V_\C$.
	The Hermitian symmetric space $\widetilde{X}_\infty$ has two connected components
	which are interchanged by  
	$[v]\mapsto [\overline{v}]$.
	The map sending the line $[v]\in Q(\C)$ spanned by $v = v_1 + i v_2$, with $v_1, v_2 \in V_\R$, to the negative definite two-dimensional subspace spanned by $v_1$ and $v_2$ gives a two-to-one map $\widetilde{X}_\infty \rightarrow X_\infty$ 
	and identifies each of the connected components with $X_\infty$.
	Because $s = 2,$ the orthogonal complement of $\{v_1,v_2\}$ in $V_{\R}$ is positive definite. In particular, it does not contain any isotropic vectors.
	It follows that
	\begin{equation}\label{eqn:Xinfty-r3}
	\widetilde{X}_\infty \subseteq \widetilde{X}_\infty':=\{[v]\in Q(\C)\ \vert\  \langle v, w \rangle \ne 0, \mbox{ for all } [w]\in Q(\R) \}.
	\end{equation}
This is in fact an equality except if $r=2$.
In signature $(2,2)$, the space $\widetilde{X}_\infty'$ has four connected components: two of them are given by $\widetilde{X}_\infty$, while the other two correspond to positive definite planes in $V_{\R}$.

\bigskip

We now specialise further to discuss indefinite quadratic spaces of dimensions $3$ and $4$, in which   $V_{\R}$ and $X_\infty$ admit useful concrete descriptions:

\medskip \noindent
{\em Signature $(1,2)$}.
When $V_\R$ is of signature $(1,2)$ (resp.~signature $(2,1)$),
it can be identified with the subspace of trace zero elements
$\mathrm{M}_2^0(\R)$  of the space of real $2\times 2$ matrices, with quadratic form given by $q=\det$ (resp.~$q=-\det$).
The special orthogonal group $G(\R)$ is identified with 
$\PGL_2(\R)$ and  $g\in G(\R)$ 
acts on $V_\R$ by conjugation:
\[
g \cdot v := gv g^{-1}.
\]
Isotropic vectors in $V_\C$ correspond to non-zero linear endomorphisms of $\C^2$ of trace 0 and determinant 0, that is, non-zero nilpotent endomorphisms.  Such an endomorphism of $\C^2$ is uniquely determined by its kernel, which equals its image.  The nilpotent endomorphism of $\C^2$ with kernel (and image) spanned by $\left(\begin{array}{c} \tau \\ 1 \end{array}\right)$ equals    
\[
 M_\tau := \left(\begin{array}{c} \tau \\ 1 \end{array}\right) ( 1,-\tau) = \left(\begin{array}{cc} \tau & -\tau^2 \\ 1 & -\tau\end{array}\right),
\qquad \tau \in \C \cup \{ \infty\},
\]
up to scaling.  The line spanned by $M_\tau$ belongs to
$\widetilde{X}_\infty$ if and only if this vector
is not defined over $\R$, i.e., $\tau$ belongs to 
the union of the complex upper and lower half planes.
The map $\tau \mapsto M_\tau$ therefore identifies 
the complex upper half space $\cH$ with a connected component of $\widetilde{X}_\infty$:
\begin{equation}
\label{eqn:Xinfty-21}
X_\infty \cong \cH := \{\tau\in \C\ \vert\ \mbox{Im}(\tau)>0 \}.
\end{equation}
The action of $G(\R)^{+}$ on $X_\infty$ corresponds to the usual action of $\PSL_2(\R)$ on $\cH$ by
M\"obius transformations via this identification.

\medskip\noindent
{\em Signature $(2,2)$}.
 When $V_{\R}$ is of signature $(2,2)$,
it can be identified with the space  $\mathrm{M}_2(\R)$ 
of $2\times 2$ real matrices endowied with the determinant quadratic form. 
The group $G(\R)^{+}$ is identified with 
$(\SL_2(\R) \times \SL_2(\R))/\{\pm 1\}$ and acts on $V_\R$ by left and right multiplication:
\begin{equation}
\label{eqn:leftrightmultact}
 (\gamma_1, \gamma_2).v := \gamma_1 v \gamma_2^{-1}.
\end{equation}
Isotropic vectors in $V_\C$ correspond to linear endomorphisms of $\C^2$ of rank one.  Such an endomorphism is uniquely determined, up to scaling, by its kernel and image.  For $\tau_1, \tau_2 \in \C,$ the endomorphism of $\C^2$ with image and kernel respective spanned by the vectors $\left(\begin{array}{c} \tau_1 \\ 1 \end{array} \right)$ and $\left( \begin{array}{c} \tau_2 \\ 1 \end{array} \right)$ equals
\[
M_{\tau_1,\tau_2} := \left(\begin{array}{c} \tau_1 \\ 1 \end{array}\right) \left( 1,  -\tau_2 \right) = \left(\begin{array}{cc} \tau_1 & -\tau_1\tau_2 \\ 1 & -\tau_2\end{array}\right),
\qquad \tau_1,\tau_2 \in \C.
\]
up to scaling.
Furthermore, the line spanned by $M_{\tau_1,\tau_2}$ belongs to
the space $\widetilde{X}_\infty'$ defined in \eqref{eqn:Xinfty-r3} if and only if
$\tau_1,\tau_2\notin \R$, i.e., if both $\tau_1$ and $\tau_2$ belong  to 
the union of the complex upper and lower half planes.
The map $(\tau_1,\tau_2) \mapsto [M_{\tau_1,\tau_2}]$ thus identifies 
$\cH\times \cH$ with one connected component of $\widetilde{X}'_\infty$ and, therefore, one may identify: 
\[
X_\infty \cong \cH \times \cH.
\]
Under this identification, the action of $G(\R)^+$ on $X_\infty$ corresponds to the natural action of $\PSL_2(\R)\times \PSL_2(\R)$ on $\cH\times \cH$ by M\"obius transformations on each coordinate.

\medskip\noindent
{\em Signature $(3,1)$}.
 When $V_{\R}$ is of signature $(3,1)$,
it can be identified with the space
\[
\{ M\in \mathrm{M}_2(\C) \ \vert\ M^\dagger = -\overline{M} \}, \qquad  q(M) = -{\det}(M),
\]
where $M^\dagger$ denotes the principal involution on $M_2(\C)$ satisfying $MM^\dagger = \det(M)$, and 
$\overline{M}$ is the complex conjugate of $M$. 
The action 
\[
 \gamma.M := \gamma \cdot M \cdot \overline{\gamma}^{-1}
\]
of $\PSL_2(\C)$ on $V_\R$ by
 twisted conjugation identifies this group with the connected component $G(\R)^{+}$ of the special orthogonal group of $V_\R$.
The isomorphism
 $$ X_\infty \cong \cH^{(3)} :=  \C\times \R^{>0},$$
of the symmetric space with hyperbolic three-space is given via the assignment
 $$\left(\begin{array}{cc} z & t \\ 1 & -\overline{z} \end{array}\right) \mapsto (z, -t-z\overline{z}).$$
The case of signature $(3,1)$ is noteworthy for  representing  the first  scenario in  ranks $3$ or $4$
 where $X_\infty$ does not admit a complex structure.
 It will be discussed at greater length in \S  
  \ref{sec:example-31}.

\subsection{$p$-adic symmetric spaces}
\label{sec:def-Xp}
\subsubsection{Definition}
Recall from the introduction that $p$ is assumed to be an odd  prime for
which $V_{\Q_p}$ admits a self-dual lattice $\Lambda$. 
When $n\geq 3$, 
the non-degenerate quadratic $\F_p$-space
$\Lambda/p\Lambda$ is  isotropic, i.e., contains a non-zero isotropic vector, and the same is thus true for
$V_{\Z_p}$ by Hensel's Lemma. It follows that
  $V_{\Q_p}$ is isotropic as well:  in fact, the {\em Witt index} of $V_{\Q_p}$ 
(the dimension of a maximal isotropic subspace) is essentially maximal: equal to $(n-1)/2$ if $n$ is odd, and    greater or equal to $(n-2)/2$ if $n$ is even.

 The description 
 of $X_\infty$ (or rather, of its double cover $\widetilde{X}_{\infty}$) 
  given in 
  \eqref{eqn:Xinfty-r3}
  in the special case of signature 
   $(r,2)$ 
  motivates the following definition of the 
 $p$-adic symmetric space 
 $X_p$, which shall be adopted {\em for all signatures}:
  \begin{equation}
\label{eqn:Xp-def}
	X_p := \{[v]\in Q(\C_p) \mbox{ for which  } \langle v, w \rangle \ne 0, \mbox{ for all } [w]\in Q(\Q_p) \}.
	\end{equation}
	The running  assumption  on $V_{\Q_p}$ implies that 
	the boundary $Q(\Q_p)$ of $X_p$ is non-empty.

\subsubsection{Rigid analytic structure}\label{sec-rigidstructure}
Like the real symmetric  space $X_\infty$ of \eqref{eqn:Xinfty-r2} in signature $(r,2)$, the space
$X_p$ admits a natural analytic structure: in this instance, 
  a
rigid analytic structure obtained by
 expressing $X_p$ 
 as an increasing union of {\em affinoid subsets}, as will now be described.
 
The choice of a self-dual $\Z_p$-lattice  $\Lambda\subseteq V_{\Q_p}$ extends the projective space $\mathbb{P}_{V_{\Q_p}}$
over $\Q_p$ to a smooth proper model  $\mathbb{P}_\Lambda$ over
${\rm Spec}(\Z_p)$. 
Since $\Lambda$ is self-dual, the zero locus of $q$ in $\mathbb{P}_\Lambda$ extends the quadric $Q_{\Q_p}$
to  a smooth  integral model over  ${\rm Spec}(\Z_p)$.
By the valuative criterion for properness,
\[
Q(\Q_p)=Q_{\Lambda}(\Z_p)\quad \mbox{ and } 
\quad Q(\C_p)=Q_{\Lambda}(\cO_{\C_p}).
\]
In more concrete terms, put $\Lambda_{\cO_{\C_p}}=\Lambda\otimes_{\Z_p} \cO_{\C_p}$
and let
$$\Lambda':=  \Lambda \Setminus p \Lambda, \qquad \Lambda_{\cO_{\C_p}}' := \Lambda_{\cO_{\C_p}} \Setminus {\mathfrak{m}}_{\cO_{\C_p}} \Lambda_{\cO_{\C_p}} $$ 
denote the sets of primitive vectors in $\Lambda$ and $\Lambda_{\cO_{\C_p}}$ respectively.
Each $\xi\in Q(\Q_p)$ 
(resp.~$Q(\C_p)$)  can be  represented by an
isotropic vector $v_\xi$  in $\Lambda'$ (resp.~in
$\Lambda_{\cO_{\C_p}}'$)  which is well-defined up to multiplication by 
$\Z_p^\times$ (resp.~$\cO_{\C_p}^\times$) and extends $\xi$ to a point of $Q_{\Lambda}$ over $\Z_p$ or $\cO_{\C_p}$ respectively.

For each integer $k\geq 0$ and every self-dual lattice $\Lambda\subseteq V_{\Q_p}$, the attached {\em basic affinoid subset} $X_{p,\Lambda}^{\leq k}\subseteq X_p$ is defined 
 to be
\[
X_{p,\Lambda}^{\leq k} = \left\{\xi \in Q(\C_p), \mbox{ with }  
 \ord_p(\langle v_\xi,w\rangle) \leq k  \mbox{ for all }    w\in (\Lambda')_0\right\},
\]
where the subscript $0$ denotes the set of isotropic vectors.
 \begin{lemma}\label{affinoidlemma}
 The space $X_p$ is the increasing union of the   
 $X_{p,\Lambda}^{\leq k}$.
  \end{lemma}
  \begin{proof}
  It suffices to prove the inclusion 
      $$ \left(\bigcup_{k=0}^{\infty}
   X_{p,\Lambda}^{\leq k}\right)^c \subseteq X_p^c,$$
   where $A^c$ denotes the  complement of a subset $A$
	inside the projective space $\mathbb{P}_V(\C_p).$
If $\xi$ belongs to $ \bigcap (X_{p,\Lambda}^{\leq k})^c$, 
  then  for each $k \geq 1$ there is a vector
  $w_k \in (\Lambda')_0$ satisfying
  $$ \langle v_\xi, w_k\rangle \equiv 0 \bmod{p^k \cO_{\C_p}}.$$
  By the compactness of $(\Lambda')_0$, the sequence $(w_k)_{k\geq 1}$ contains a convergent subsequence, whose limit is an element $w\in (\Lambda')_0$ satisfying
  $$ \langle v_\xi,w\rangle = 0.$$
  It follows that $\xi$ belongs to $X_p^c$, as was to be shown.
  \end{proof}
	
	\begin{remark}
	It is easily checked that the covering $X_{p,\Lambda}^{\leq 0}\subseteq X_{p,\Lambda}^{\leq 1} \subseteq X_{p,\Lambda}^{\leq 2}\subseteq\ldots$ fulfils the conditions of \cite[Definition 2.3]{Kiehl}.
	Hence, the $p$-adic symmetric space $X_p$ is a rigid analytic Stein space.
	\end{remark}

\medskip\noindent
\subsubsection{Low-dimensional examples of $X_p$}

Assume for this section that $V_{\Q_p}$ is \emph{split}, i.e., there is a basis such that the corresponding Gram matrix is of the form 
\[\begin{pmatrix} & & 1\\ &\iddots & \\ 1 & & \end{pmatrix}.\]
Let $Q^{\adm}$ be the connected rigid analytic space over $\Q_p$ for which 
\[
Q^{\adm}(\C_p)\ =\ \{\xi \in Q(\C_p)\ \vert\ \xi \cap D_{\C_p}=0 \mbox{ for all totally isotropic } D\subseteq V_{\Q_p} \}.
\]
By definition, the $p$-adic symmetric space $X_p$ is an admissible open subspace of $Q^{\adm}.$  In dimensions $n\leq 4,$  the two spaces agree and admit familiar descriptions, as will be explained below. On the other hand, for $n\geq 5$ the inclusion $X_p\subseteq Q^{\adm}$ is strict.

\begin{remark}
	The rigid analytic space $Q^{\adm}$ naturally appears in the theory of local Shimura varieties as a $p$-adic period domain.
	To any triple $(H,\mu,b)$ consisting of
	\begin{itemize}
	\item a connected reductive group $H$ over $\Q_p$,
	\item a miniscule cocharacter $\mu\colon \G_{m,\overline{\Q}_p}\to H_{\overline{\Q}_p}$ and
	\item an element $b\in H(\breve{\Q}_p)$, where $\breve{\Q}_p$ denotes the completion of the maximal unramified extension of $\Q_p$,
	\end{itemize}
	one can attach the weakly admissible period domain $\mathcal{F}^{\mathrm{wa}}(H,\mu,b)$ (\cite[\S 2]{CFS})
	and the admissible period domain $\mathcal{F}^{\mathrm{a}}(H,\mu,b)\subseteq \mathcal{F}^{\mathrm{wa}}(H,\mu,b)$ (see Definition 3.1 of \cite{CFS}).
	They are open rigid analytic subvarieties of the Flag variety $\mathcal{F}(H,\mu)$ attached to the pair $(H,\mu)$ defined via $p$-adic Hodge theoretic means.
	Moreover, they are stable under the action of the Frobenius stabilizer $J_b$ of $b$ in $H$ (see \cite[p.~229]{CFS}).
	Now let $H$ be the special orthogonal group of the split $n$-dimensional quadratic form, $\mu$ the standard cocharacter given by $z \mapsto \mathrm{diag}(z,1\ldots,1,z^{-1})$ and $b=1$.
	Then $J_b$ is equal to $H$, the admissible and weakly admissible period spaces agree and are equal to $Q^{\adm}$ (cf.~\cite[Appendix A]{Shen}).
	In particular, the local reductive group $G_{\Q_p}$ should be viewed as the Frobenius stabilizer $J_b$ of some different orthogonal group.

\end{remark}

\begin{example}[Three-dimensional quadratic spaces]
 The dimension of a totally isotropic subspace of a three-dimensional quadratic space is at most one.  Hence,  
 \[
 X_p=Q^{\adm}= Q(\C_p) \Smallsetminus Q(\Q_p).
 \]
 An explicit model for the split quadratic space of dimension 3 over $\Q_p$ is given by the space of $2 \times 2$-matrices with trace zero.
Just as   in the discussion of  real quadratic spaces of
signature $(2,1)$, one obtains an identification 
of $G(\Q_p)$ with $\PGL_2(\Q_p)$.
 The same reasoning that led to 
\eqref{eqn:Xinfty-21}
leads to the analogous identification
\begin{equation}
\label{eqn:XpHp}
 X_p\cong \mathbb{P}^1(\C_p) \Smallsetminus \mathbb{P}^1(\Q_p) =: \mathcal{H}_p
 \end{equation}
of $X_p$ with the Drinfeld 
 $p$-adic upper half plane $\mathcal{H}_p$,
 on which $\G(\Q_p) = \PGL_2(\Q_p)$ 
acts via
M\"obius transformations.
\end{example}

\begin{example}[Four-dimensional quadratic spaces]
 Let $\xi=[v_\xi]$ be an element of $Q(\C_p)$ that is not contained in $X_p$, i.e.~for which there is an element
 $[w]\in Q(\Q_p)$ with $\langle w,v_\xi\rangle=0.$
 The two-dimensional regular quadratic space $V'=w^\perp/\Span(w)$ is split over $\Q_p$.
Hence, the image of $v_\xi$ in $V'_{\C_p}$ is a multiple of a $\Q_p$-rational isotropic vector.
 Let $w'\in V_{\Q_p}$ be a preimage of this isotropic vector.
 By construction, $w$ and $w'$ span a totally isotropic subspace $D\subseteq V_{\Q_p}$ and $v_\xi$ is an element of $D_{\C_p}.$
 Thus, $v$ is not an element of $Q^{\adm},$
and it follows that
 \[
 X_p=Q^{\adm}.
 \]

 An explicit model for the split quadratic space of dimension 4 over $\Q_p$
 is given by setting
\begin{equation}
 \label{eqn:example4}
 V_{\Q_p}=\Mat_2(\Q_p)
 \quad \mbox{with} \quad 
 \langle A,A \rangle=\det(A).
 \end{equation}
 The group $\SL_2(\Q_p) \times \SL_2(\Q_p)$ acts on $V_p$ via
 \eqref{eqn:leftrightmultact}, just as in the real case, inducing an isomorphism 
\begin{equation}
 \label{excisom2}
 G(\Q_p)^{+} =  (\SL_{2}(\Q_p) \times \SL_{2}(\Q_p))/\{\pm 1\},
\end{equation} 
as well as an identification
$$
 X_p = \mathcal{H}_p \times \mathcal{H}_p
$$
 of rigid analytic varieties. The action
 of $G(\Q_p)^{+}$ on $X_p$ is by componentwise M\"obius transformations under these
 identifications.
\end{example}

\subsubsection{GIT characterization of $X_p$.}

We now turn to a description of the points in 
$X_p$ in the spirit of ``geometric invariant theory",  
characterizing them in terms of their stabilisers and their orbits under the action of split tori in $G(\Q_p)$. 

In dimension $3$ and $4$, the symmetric space $X_p$ consists of exactly the points in $Q(\C_p)$ that are not stabilized by a $\Q_p$-split torus in $G_{\Q_p}.$
The following illustrative example shows: in dimension $> 4$, yet more points of $Q(\C_p)$ beyond those stabilized by a $\Q_p$-split torus in $G_{\Q_p},$ must be removed in order to obtain $X_p$ from $Q(\C_p).$ 

\begin{example}\label{ex-5dim}
 Let $V_{\Q_p}$ be the split quadratic space of dimension $5$.
 Write $V_{\Q_p}$ as the orthogonal direct sum
 \[
 V_{\Q_p}=H \oplus U
 \]
 of a hyperbolic plane $H$ and a three-dimensional split quadratic space $U$.
 Let $0\neq w_H\in H$ be an isotropic vector, $w_U\in U_{\C_p}$ an isotropic vector such that
 \[
 \langle w_U,u \rangle\notin \Q_p\quad \forall u\in U\Smallsetminus\{0\},
 \]
 and put $v:=w_H + w_U$.
 It is easy to check that $v^\perp\cap V_{\Q_p}=\Q_p w.$
 In particular,  the image $\xi=[v]$ of $v$ in $ Q(\C_p)$ does not belong to $X_p$.  We argue next that $\xi$ is not stabilized by any $\Q_p$-split maximal torus.
	
 Suppose that $T\subseteq G_{\Q_p}$ is a $\Q_p$-split torus that stabilizes the line $\xi=[v]$, and let 
 \[
 V_{\Q_p}=\bigoplus_{\chi \in X^\ast(T)} V_{\Q_p,\chi}
 \]
be the  eigenspace decomposition
of the vector space $V_{\Q_p}$
 with respect to  the action of $T$.
 Since $T$ acts via orthogonal transformations for a non-degenerate quadratic space, it follows  that 
  \begin{itemize}
  \item $\dim_{\Q_p}V_{\Q_p,\chi}=\dim_{\Q_p} V_{\Q_p,\chi^{-1}}$ for every character $\chi\in X^\ast(T)$;
  \item $V_{\Q_p,\chi_1} \perp V_{\Q_p,\chi_2}$ for all characters $\chi_1,\chi_2\in X^\ast(T)$ with $\chi_1\neq \chi_2^{-1}$.
  \end{itemize}
Thus, $\dim V_{\C_p,0}$ must be odd, hence $V_{\Q_p,0} \neq \{0 \}.$  On the other hand, since $G_{\Q_p},$ and hence $T,$ acts faithfully on $V_{\Q_p}$ we must have that $V_{\Q_p, \chi^{\pm 1}} \neq \{0\}$ for some non-trivial character $\chi\in X^{\ast}(T)$.
It follows that $\dim v^\perp \cap V_{\Q_p} \geq 2$.
Indeed:
\begin{itemize}
\item
Since $T$ stabilizes $\xi=[v]$, the vector $v$ is an eigenvector for the action of $T.$  If $v \in V_{\C_p,0},$ then it is perpendicular to $V_{\Q_p, \chi} \oplus V_{\Q_p,\chi^{-1}},$ which has dimension at least 2.  If $v \in V_{\C_p,\chi^{\pm 1}}$ on the other hand, then $v$ is perpendicular to $V_{\Q_p,0} \oplus V_{\Q_p, \chi^{\mp 1}},$ which has dimension at least 2 
\end{itemize}
Since $v^{\perp} \cap V_{\Q_p} = \Q_p w$ is only one-dimensional, the above gives a contradiction.
Thus, the stabilizer of $\xi$ contains no $\Q_p$-split torus.
\end{example}

The point $\xi$ in the example above is not stabilized by a $\Q_p$-split subtorus of $G_{\Q_p}$.
Nevertheless, it is still unstable with respect to such a subtorus.
Recall the definition of stability with respect to one-parameter subgroups: 
fix a one-parameter $\Q_p$-subgroup
$$\mu\colon\G_{m,\Q_p}\rightarrow G_{\Q_p},$$
i.e, a non-trivial homomorphism of algebraic groups over $\Q_p$, 
and a point $\xi \in Q(\C_p)$.
As before, we fix a lift $v_{\xi}$ of $\xi$ to $V_{\C_p}$.
Let us recall that $\xi$ is called \emph{stable with respect to $\mu$} if
\begin{itemize}
\item the stabilizer of $v_{\xi}$ in $\mathbb{G}_{m,\C_p}$ is finite and
\item the $\G_{m,\C_p}$-orbit of $v_{\xi}$ is Zariski-closed in $V_{\C_p}$.
\end{itemize}
The action of $\G_{m,\Q_p}$ on $V_{\Q_p}$ induces an eigenspace decomposition
	\begin{align}\label{eigenspace}
	V_{\Q_p}=\bigoplus_{n\in \Z} V_{\Q_p,n}.
	\end{align}
Here we used the canonical identification $X^\ast(\G_m)=\Z$.
As $\G_{m,\Q_p}$ acts via orthogonal transformations we see that $V_{\Q_p,n}$ and $V_{\Q_p,m}$ are perpendicular unless $n=-m$.
In particular, $V_{\Q_p,n}$ is totally isotropic for all $n\neq 0$.
It is easy to see that $\xi$ is stable with respect to $\mu$ if and only if $v_\xi$ is neither contained in
\begin{align*}
	V_{\C_p, \geq 0}&=\bigoplus_{n\geq 0} V_{\C_p,n}\\
\intertext{nor in}
	V_{\C_p, \leq 0}&=\bigoplus_{n\geq 0} V_{\C_p,n}.
\end{align*}
In particular, if $x$ is fixed by $\G_{m,\C_p}$, then $x$ is not stable.

 \begin{proposition} \label{GITXp}
 Let $\xi$ be an element of $Q(\C_p)$.
 The following conditions are equivalent:
  \begin{itemize} 
	\item $\xi\in X_p$ and
  \item $\xi$ is stable with respect to every one-parameter $\Q_p$-subgroup of $G_{\Q_p}$.
  \end{itemize}
 \end{proposition}
  \begin{proof}
  Suppose there exists an one-parameter $\Q_p$-subgroup $\mu\colon\mathbb{G}_{m,\Q_p}\hookrightarrow G_{\Q_p}$
	such that $\xi$ is not stable with respect to $\mu$.
	After possibly replacing $\mu$ by $-\mu$ we may assume that $v_{\xi}$ is an element
	of the non-negative part $V_{\C_p, \geq 0}$ of the eigenspace decomposition \eqref{eigenspace} induced by $\mu$.
	Let $n_0 > 0$ be an integer with $V_{\C_p,n_0}\neq 0$, which exists as $\mu$ is non-trivial.
	Then $\langle w,v_{\xi}\rangle=0$ for all $w\in V_{\Q_p,n_0}$ and, hence, $\xi$ is not an element of $X_p$.
	
	Conversely, suppose that there exists an isotropic vector $w\in V_{\Q_p}$ such that $\langle w,v_{\xi}\rangle=0.$
	There exists a isotropic vector $w'\in V_{\Q_p}$ with 
	\[
	\langle w,w'\rangle=1.
	\]
	Let $U$ be the orthogonal complement of $\Span\{w,w'\}$.
	This yields a one-parameter $\Q_p$-subgroup $\mu\colon\mathbb{G}_{m,\Q_p}\rightarrow G_{\Q_p}$ 
	as follows:
  \[
	\mu(t)(y) = \begin{cases}
                 t\cdot y & \text{ if } y \in \Q_p w \\
	       t^{-1}\cdot y &\text{ if } y \in \Q_p w' \\ 
	          y & \text{ if } y \in U.
			\end{cases}
	\]
	In other words, there is an eigenspace decomposition of the form
	\[
	V_{\Q_p}=V_{\Q_p,1}\oplus V_{\Q_p,-1}\oplus V_{\Q_p,0}\ \mbox{with}\ V_{\Q_p,1}=\Q_p w.
	\]
	Because $w$ and $v_{\xi}$ are perpendicular we see that $v_{\xi}\in V_{\C_p,1}\oplus V_{\C_p,0}$
	and, therefore, $\xi$ is not stable with respect to $\mu$.
	\end{proof}
	
 \begin{remark}
In \cite{Totaro} Totaro shows that the weakly admissible period domain attached to a triple $(H,\mu,b)$ is given by the subspace of the Flag variety $\mathcal{F}(H,\mu)$ consisting of those points that are \textit{stable} with respect to all one-parameter $\Q_p$-subgroups of $J_b$. 
Let $\mathcal{F}$ be a generalized flag variety for some split reductive group $J$ over $\Q_p$.
Voskuil and van der Put analysed when the subspace of $\mathcal{F}$ given by those points that are stable with respect to all one-parameter $\Q_p$-subgroups of $J$ agrees with the space of points which are semistable with respect to all one-parameter $\Q_p$-subgroups of $J$ (cf.~\cite{VoskuilvanderPut}), which explains the phenomenon observed in Example \ref{ex-5dim}.
 \end{remark}


\section{Kudla--Millson divisors}
\label{sec:divisors}

This chapter introduces the key notion
 of \emph{Kudla--Millson divisors}.
In the theory of rigid meromorphic cocycles, they 
  play the role 
  of Heegner divisors on orthogonal Shimura varieties.

\subsection{Homological algebra notation}
Some notations and conventions regarding homological algebra  are collected in this  section.

Let $R$ be a (not necessarily commutative) ring.
A {\em chain complex}  of $R$-modules is  a chain complex of left $R$-modules concentrated in non-negative degree, i.e., a complex of the form
\[\ldots \xlongrightarrow{d_3} A_2 \xlongrightarrow{d_2} A_1 \xlongrightarrow{d_1} A_0 \xlongrightarrow {d_0} 0 \xlongrightarrow {d_{-1}} 0 \xlongrightarrow{d_{-2}} \ldots \]
The category of such chain complexes is denoted by $\Ch_{\geq 0}(R)$. 
For a chain complex $A_{\sbullet}\in \Ch_{\geq 0}(R)$ and an integer $t\geq 0$, 
write $A_{\sbullet}[-t]$ for the chain complex given by
\[A_q[-t]=A_{q-t}.\]
The canonical $t$-truncation of $\tau_{\geq t} A_{\sbullet}$ of $A_{\sbullet}$ is the subcomplex given by
\[
(\tau_{\geq t} A)_q  =\begin{cases}
0 & \mbox{if}\ q< t,\\
\Ker(d_q)& \mbox{if}\ q=t,\\
A_q & \mbox{if}\ q > t.
\end{cases}
\]
By construction, the homology of the $t$-truncation is given by
\[
H_q \left(\tau_{\geq t} A_{\sbullet} \right)=\begin{cases}
0 & \mbox{if}\ q< t,\\
H_q(A_{\sbullet}) & \mbox{if}\ q \geq t.
\end{cases}
\]
A left $R$-module $M$ will often be viewed as a  chain
 complex concentrated in degree $0$.
A {\em resolution} of $M$ is a chain complex $A_{\sbullet}\in \Ch_{\geq 0}$ together with a quasi-isomorphism $f\colon A_{\sbullet} \rightarrow M$.
It is called a {\em projective resolution} if each $A_q$ is a projective $R$-module.

Given a group $\absgr$ and a commutative ring $R$, write $R[\absgr]$ for the group ring of $\absgr$ over $R$.
Given a  $\absgr$-set $I$,  let 
$\underline{I}$  denote the category
whose objects are the elements of $I$, 
with morphisms given by
$$ \Hom(i,j) := \{ \absel\in\absgr \mbox{ with } gi =j \}.$$
An $R$-linear representation of $I$ is a functor from $\underline{I}$ to the category of $R$-modules.
More concretely, an $R$-linear representation of $I$ consists of
\begin{itemize}
\item a collection of $R$-modules $M_{i}$ indexed by  $i\in I$,
\item for each $\absel\in\absgr$ and $i\in I$, an $R$-linear homomorphism
$$\absel \colon M_{i} \longrightarrow M_{\absel i}$$
satisfying the properties implicit in functoriality: the neutral element of $\absgr$ represents
the identity transformation from $M_i$ to $M_i$  for each $i$, and the following diagram commutes, for all 
$i\in I$ and $\absel, \absel' \in \absgr$:
\begin{equation}
 \label{compatibility}
 \xymatrix{ M_i \ar[r]^{\absel}   \ar@/^1.5pc/[rr]^{\absel'\absel} &  M_{\absel i}  \ar[r]^{\absel'}  
  & M_{\absel'\absel i}. }  
\end{equation}

\end{itemize}
Given an $R$-linear representation $M = ( M_i)_{i\in I}$ of $I$, both the direct sum $\oplus_{i\in I} M_{i}$ and the product $\prod_{i\in I} M_{i}$ are equipped with a $\absgr$-module structure via
$$\absel.(m_i)_{i\in I} = (\absel m_{\absel^{-1} i})_{i\in I}.$$
Likewise, a chain complex $M_{\sbullet}$ of $R$-linear representations of $I$ is a functor from
${\underline I}$ to the category of complexes of 
$R$-modules, i.e.,  a collection of  chain complexes $M_{{\sbullet},i}$ of $R$-modules indexed by $i\in I$, together with homomorphisms of  complexes
$$\absel\colon M_{{\sbullet}, i} \longrightarrow M_{{\sbullet},\absel i}$$
for all $\absel \in \absgr$ and $i\in I$,  for which 
the analogue of \eqref{compatibility} commutes.

\subsection{Special cycles on $X_\infty$}
\label{Archimedean}
\subsubsection{Definitions}
A vector $v\in V_{\R}$ is said to have \emph{positive length} if $q(v)> 0$.
Denote by $V_{\R,+}$ (resp.~$V_{+}$) the subset of $V_\R$ (resp.~$V$) of vectors of positive length.
Given a vector $v\in V_{\R,+}$,  let $\Delta_{v,\infty}\subseteq X_\infty$ be the real cycle defined by
\[
\Delta_{v,\infty}=\left\{Z\in X_\infty \ \vert\ v \perp Z \right\}.
\]
The cycle $\Delta_{v,\infty}$ is identified with the Archimedean symmetric space attached to $(\R v)^\perp$, a quadratic space
of signature $(r-1,s)$. Hence $\Delta_{v,\infty}$ is of dimension $(r-1)s$ and defines
 a cycle of real codimension 
$s$ in $X_\infty$.
Alternately,    let 
$$ \tau := \{ (w, Z)\ \vert\ w \in Z \} \  \subseteq \ V_{\R} \times Z$$
be the tautological  rank $s$ vector bundle over $X_\infty$
whose fibre over $Z$  (relative to the second coordinate projection) is the vector space $Z$ itself.
A positive vector $v$ defines a section $s_v\colon X_\infty \rightarrow \tau$ sending $Z$ to $(v_Z,Z)$, where $v_Z$ is the orthogonal projection of $v$ onto  $Z$.
The  cycle $\Delta_{v,\infty}$ is the zero locus of 
this  section.

\subsubsection{Orientations.}\label{sec-orientations}
We recall some of the generalities concerning orientation on special cycles $\Delta_{v,\infty}$ as described in \cite[pp. 130-131]{KM}.

An orientation on a real vector space $W$ of dimension $d$ is a choice of a connected component $\mathfrak{o}\subset \wedge^d W -\{0\}\simeq \R-\{0\}$. 
A basis $(w_1,\ldots,w_d)$ of $W$ is then said to be {\em positively oriented}, or simply {\em positive}, if $w_1 \wedge\cdots \wedge w_d$ belongs to $\mathfrak{o}$.
An orientation on $W$ induces one on its dual $W^\ast$,
 by declaring that the dual basis of any positive basis for $W$  is positive.
Given oriented  real vector spaces $W_1$ and $W_2$ with positive bases $\{w_{1,1},\ldots, w_{1,k}\}$ and $\{w_{2,1},\ldots, w_{1,\ell}\}$, the basis $\{w_{1,i}\otimes w_{2,j}\ \vert\ 1\leq i \leq k, 1\leq j \leq \ell\}$ ordered  lexicographically from right to left is a positive  basis for $W_1\otimes W_2$.
The space $\Hom_{\R}(W_1\otimes_\R W_2)\cong W_1^\ast \otimes_\R W_2$ is oriented accordingly.
Note that if $\dim_\R W_2=1$, changing the orientation on $W_2$ changes the orientation on $W_1\otimes  W_2$ if and only if $\dim W_1$ is odd.

Fix  an orientation on $V_\R$ once and for all.
Because $X_{\infty}$ is contractible, the  real vector bundle $\tau \rightarrow X_{\infty}$  is orientable. An orientation on one of the fibres of $\tau$ thus equips each maximal negative-definite subspace of $V_\R$ with a consistent choice of orientation, compatible with the natural action of $G(\R)^+$. 
With these two choices of orientation in hand, a basis for a maximal positive-definite subspace $W\subseteq V_\R$ is said to be positively oriented if completing it  by a positive  basis of $W^\perp$ yields a positive  basis of $V_\R$.

Let $\Delta_{v,\infty}$ be the  special cycle attached to $v\in V$,
 and let $Z\in \Delta_{v,\infty}$.
The subspace $U\subseteq V_\R$ generated by $v$ is oriented  by saying that $\{v\}$ is a positive basis.
Following the reasoning in 
\eqref{eqn:tangent-space}, 
there are canonical isomorphisms
$$
T_Z(X_\infty) \cong \Hom_\R(Z,Z^\perp), \qquad \nu_Z(X_\infty) \cong \Hom_\R(Z,U),
$$
where $\nu_Z(X_\infty)$ denotes the normal space of $\Delta_{v,\infty}$ at $Z$. 
These isomorphisms lead to 
 orientations on both the tangent and normal spaces of $X_\infty$ at $Z$.
In other words,   the tangent space $T_Z(\Delta_{v,\infty})$ is oriented  by the convention  that a positive basis of $T_Z(\Delta_{v,\infty})$ completed by a positive basis of $\nu_Z(\Delta_{v,\infty})$ is a positive basis of $T_Z(X_\infty)$.
The cycle  $\Delta_{v,\infty}$ is equipped with the induced orientation. 
The spaces
$\Delta_{v,\infty}$ and $\Delta_{-v,\infty}$ carry the same orientation if $s$ is even, while they carry opposite orientations if $s$ is odd.
One immediately deduces that the action of 
 $g \in G(\R)^{+}$ preserves the chosen orientation on the cycles $\Delta_{v,\infty}$, i.e., there is an equality of \emph{oriented} cycles
\begin{align*}
g\Delta_{v,\infty}=\Delta_{g\cdot v,\infty}.
\end{align*}

\subsubsection{Homology of complements of special cycles}
Given a topological space $X$, 
 denote by $C_{\sbullet}(X)$ the singular chain complex of $X$ with integer coefficients, i.e., $C_q(X)$ is the free abelian group generated by all continuous maps from the standard $q$-simplex to $X$.
If $A$ is a subspace of $X$,   write
$$C_{\sbullet}(X,A)=C_{\sbullet}(X)/C_{{\sbullet}}(A)$$
for the relative singular chain complex with 
associated relative homology groups $H_q(X,A)$ for $q\geq 0$.
\begin{proposition}\label{relativehomology}
Let $v$ be an element of $V_\R$ of positive length.
Then
$$H_q(X_\infty,X_\infty\Setminus \Delta_{v,\infty})
\cong\begin{cases}
\Z & \mbox{if}\ q=s \\
0 & \mbox{ otherwise.}
\end{cases}
$$
\end{proposition}
\begin{proof}
We may assume that $s\geq 1$, the case $s=0$ being trivial.

The assertion is a consequence of the long exact sequence for relative homology and the following claim:
$X_\infty\Setminus \Delta_{v,\infty}$ is homotopy equivalent to an $(s-1)$-sphere.
Indeed, since $\Delta_{v,\infty}$ is a totally geodesic submanifold of the symmetric space $X_\infty,$ the exponential map induces a diffeomorphism
$$\exp\colon T_Z(X_\infty)\Setminus T_Z(\Delta_{v,\infty})\xlongrightarrow{\cong} X_\infty \Setminus \Delta_{v,\infty}$$
for every $Z\in \Delta_{v,\infty}$.
But for any finite dimensional $\R$-vector space $V$ the complement of a proper subspace $W\subsetneq V$ of codimension $s$ is homotopy equivalent to $(V/W) \Setminus \{0\}$, which in turn is homotopy equivalent to an $(s-1)$-sphere. 
\end{proof}

The chosen orientation of $\Delta_{v,\infty}$ determines a generator of $H_s(X_\infty,X_\infty\Setminus \Delta_{v,\infty})$, giving rise to an identification
\begin{align}\label{orientation}
H_s(X_\infty,X_\infty\Setminus \Delta_{v,\infty})=\Z.
\end{align}

\begin{lemma}
If $\gamma \in G(\R)^{+}$ fixes $v$, then the induced map
$$\gamma_\ast\colon H_q(X_\infty,X_\infty\Setminus \Delta_{v,\infty}) \longrightarrow H_q(X_\infty,X_\infty\Setminus \Delta_{v,\infty}) $$
is the identity.
\end{lemma}

\begin{proof}
Let $W$ be the orthogonal complement of $v$ in $V_\R$.
Then the stabilizer of $v$ in $G(\R)^{+}$ is the group of elements in $\SO(W)$ that have trivial spinor norm.
Since this group is connected, the claim follows.
\end{proof}

\subsubsection{Kudla-Millson cycles intersecting a compact region in $X_{\infty}$}

 \begin{lemma}
 \label{Archimedeanlemma}
 Let $C\subseteq {X}_\infty$ be a compact subset and $m>0$ a fixed real number.
 Then the set
 \[
K:=  \left\{v\in V_\R\ \middle\vert  \  q(v)=m,\ \Delta_{v,\infty} \cap C \neq \emptyset  \ \right\}
 \]
 is compact.
 \end{lemma}
  \begin{proof} 
Consider the set 
$$ \Omega := \left\{ (Z,v) \in  C\times V_\R   \  \vert \ Z^\perp \ni v, \ \langle v,v\rangle = m  \right\},$$
and let $p_1\colon \Omega \rightarrow C$ and $p_2\colon \Omega\rightarrow V_{\R}$ be
the two coordinate projections. The map $p_1$, whose fibres  are spheres of radius $m$ in $r$-dimensional Euclidean space, is proper.  Hence, $\Omega$ is compact, and  the same follows for $K = p_2(\Omega)$. 
  \end{proof}

\subsection{Quadratic divisors on $X_p$}

\subsubsection{Definition}
A vector $v\in V_{\Q_p}$ is said to be {\em anisotropic} if $q(v)\neq 0$.
The {\em quadratic divisor} on $X_p$ attached 
to such a $v$ is the subset
 $\Delta_{v,p}\subseteq X_p$ given by
   \begin{align*}
\Delta_{v,p}=\left\{\xi\in X_p\ \vert\ v \perp \xi \right\}.
\end{align*}
In case $v\in V_{+}$, the divisor $\Delta_{v,p}$ is called \emph{rational quadratic}.
 By definition, 
\begin{equation}
\label{equivatp}\gamma \cdot \Delta_{v,p}=\Delta_{\gamma v,p}
\end{equation}
for all $\gamma \in G(\Q_p)$ and all anisotropic vectors $v\in V_{\Q_p}$.

\begin{remark}
Suppose $n=3$.
Then $\Delta_{v,p}=\emptyset$ if and only if the orthogonal complement of $v$ in $V_{\Q_p}$ is a hyperbolic plane.
If $V_{\Q_p}$ is the split $3$-dimensional quadratic space, this is equivalent to $q(v)$ being a square in $\Q_p^\times$.
For $n\geq 4$ the divisor $\Delta_{v,p}$ is always non-empty. (See Lemma \ref{padiclemmaconverse} below.)
 \end{remark}

\subsubsection{Quadratic divisors intersecting affinoids}
Let $\Lambda\subseteq V_{\Q_p}$ be a self-dual $\Z_p$-lattice.
Remember that $\Lambda'$ denotes the set of primitive vectors of $\Lambda$.
Given a non-zero vector $v\in V_{\Q_p}$ there exists an integer $k$ and a primitive vector $v_0\in \Lambda'$ such that $v=p^k\cdot v_0$.
The \emph{order} of $v$ with respect to $\Lambda$ is
\[
\ord_{\Lambda}(v)=k \in \Z
\]
and the \emph{isotropy level} of $v$ with respect to $\Lambda$ is given by
\[
\iso_{\Lambda}(v):=\ord_p(q(v_0)) \in \Z_{\geq 0} \cup \{\infty\}.
\]
Since $v$ determines $v_0$ up to multiplication by $\Z_p^\times,$ both $\ord_{\Lambda}(v)$ and $\iso_{\Lambda}(v)$ are well-defined.  

As explained in Section \ref{sec-rigidstructure} the $\Z_p$-lattice $\Lambda\subseteq V_{\Q_p}$ allows us to express $X_p$ as an increasing union of affinoids:  
$$ X_p = \bigcup_{k=0}^\infty X_{p,\Lambda}^{\le k}.$$
The following $p$-adic counterpart of
Lemma \ref{Archimedeanlemma} relates this filtration to isotropy level.
\begin{lemma}
 \label{padiclemma}
Let $\Lambda\subseteq V_{\Q_p}$ be a self-dual $\Z_p$-lattice and $v\in V_{\Q_p}$ an anisotropic vector.
Then:
\[
\iso_\Lambda(v)>k \quad \Rightarrow \quad \Delta_{v,p}\cap X_{p,\Lambda}^{\le k}=\emptyset.
\]
In particular, for every $m\in \Q_p^\times$ and every integer $k\in\Z$ the set
\[
K:=\{ v\in V_{\Q_p}\ \vert\ q(v)=m,\  \Delta_{v,p}\cap X_{p,\Lambda}^{\le k}\neq \emptyset  \ \}
\]
is contained in $p^{-\ell}\Lambda$ for $\ell \leq \frac{k-\ord_p(m)}{2}$.
 \end{lemma}
 \begin{proof} We may assume that $v\in \Lambda'$ and thus that $d:= \iso_\Lambda(v)=\ord_p(q(v))$.
  We note that $q(v)= 0$ modulo $p^{d}$, and hence the image of the vector $v$ in $\mathbb{P}({\Lambda}/p^{d}\Lambda)$ lies on the quadric $Q_\Lambda(\Z/p^{d}\Z)$. 
  Since this quadric is smooth, Hensel's Lemma
   implies that there is an isotropic
    vector $\widetilde v\in \Lambda'$ satisfying 
  $$\widetilde v\equiv v \bmod{p^{d}\Lambda}.$$
  This implies that if $z\in \Lambda'_{\cO_{\C_p}}$
   is orthogonal to $v$, then 
  $$\ord_p(\langle z, \widetilde v\rangle ) \geq d > k,$$
  and therefore that $[z]\notin X_{p,\Lambda}^{\leq k}$.
 \end{proof}

The following is a partial converse to the previous lemma:
 \begin{lemma}\label{padiclemmaconverse}
Let $\Lambda\subseteq V_{\Q_p}$ be a self-dual $\Z_p$-lattice and $v \in V_{\Q_p}$ with $q(v)\neq 0$ such that the orthogonal complement of $v$ in $V_{\Q_p}$ is not a hyperbolic plane. Then:
\[
\iso_{\Lambda}(v) \leq k \quad \Rightarrow \quad \Delta_{v,p}\cap X_{p,\Lambda}^{\leq \lceil 3k/2\rceil }\neq \emptyset.
\]
 \end{lemma}
  \begin{proof}
As before, we assume that $v\in \Lambda'$ and hence   $d:= \iso_\Lambda(v)=\ord_p(q(v))$.
Let $V_1$ be the orthogonal complement of $v$ in $V_{\Q_p}$ and write	$\pi\colon V_{\Q_p} \rightarrow V_1$	for the orthogonal projection onto $V_1$.
By Lemma \ref{lem-orthdecomp} the projection $\pi(\Lambda)$ is the dual lattice of the intersection $\Lambda_1=\Lambda\cap V_1$ and the quotient $\pi(\Lambda)/\Lambda_1$ is cyclic of order $p^d$.
By assumption $v$ is not isotropic modulo $p^{d+1}$ and, therefore, $v\not\equiv w \bmod p^{d+1}$ for every isotropic vector $w\in (\Lambda')_0$.
The reduction modulo $p^{d+1}$ of $\langle\cdot,\cdot\rangle$ defines a perfect bilinear form on $\Lambda/p^{d+1}\Lambda$.
It follows that for every $w\in (\Lambda')_0$ there exists an element $\bar{u}\in \Lambda/p^{d+1}\Lambda$ with
\[
\langle u, v\rangle= 0 \bmod{p^{d+1}\Z_p}\quad\mbox{and}\quad\langle u, w\rangle\neq 0 \bmod{p^{d+1}\Z_p}.
\]
By Hensel's Lemma one can lift $\bar{u}$ to an element $u\in\Lambda_1$.
In particular, we see that
\begin{align}\label{bound}
\pi(w)\notin p^{d+1}\pi(\Lambda).
\end{align}
	We may find a self-dual $\cO_{\C_p}$-lattice $\widetilde{\Lambda}_1\subseteq (V_1)_{\C_p}$ such that
	$$(\Lambda_1)_{\cO_{\C_p}}\subseteq \widetilde{\Lambda}_1 \subseteq \pi(\Lambda)_{\cO_{\C_p}}.$$
	Moreover, each quotient of subsequent lattices in the chain above is annihilated by $p^{\lceil d/2 \rceil}.$
	By \eqref{bound} we may write $$\pi(w)=p^{m_w} \widetilde{w}$$ with $\widetilde{w}$  a primitive vector of $\widetilde{\Lambda}_1$ and $m_w\leq d$.	Since $\Lambda$ is compact, it follows that the set	$\{\widetilde{w} \bmod{\mathfrak{m}_{\cO_{\C_p}} \widetilde{\Lambda}_1}\ \vert\ w\in (\Lambda')_0\}	\subseteq \widetilde{\Lambda}_1 /\mathfrak{m}_{\cO_{\C_p}} \widetilde{\Lambda}_1$	is finite.
	Let $Q_{\widetilde{\Lambda}_1}\subseteq \mathbb{P}(\widetilde{\Lambda}_1)$ be the smooth quadric cut out by the equation $q=0$.
	The $\overline{\F}_p$-valued points of $Q$ are not contained in any finite union of hyperplanes.
	Thus, by Hensel's Lemma there exists a primitive vector $\widetilde{z}\in \widetilde{\Lambda}_1$ such that	$$q(\widetilde{z})=0\qquad\mbox{and}\qquad \ord_p(\langle \widetilde{z},\widetilde{w}\rangle)=0.$$	for every $w\in (\Lambda')_0.$	There exists a rational number $m_z\leq d/2$ such that $z=p^{m_z}\widetilde{z}$ is primitive in $\Lambda_{\cO_{\C_p}}$.
	Thus,  	$$\ord_p(\langle z,w\rangle) = m_z+ m_w \leq 3d/2$$	for all $w\in (\Lambda')_0$,	i.e., the class $[z]$ lies in the intersection $\Delta_{v,p}\cap X_{p,\Lambda}^{\leq \lceil 3d/2\rceil }.$
	\end{proof}
	
\begin{remark}\label{divisorsequality}
Similar arguments as in the proof of Lemma \ref{padiclemmaconverse} show the following:
Let $v_1,v_2\in V_{\Q_p}$ with $q(v_i)\neq 0$  and $\Delta_{v_i,p}\neq \emptyset$ for $i=1,2$.
Then $\Delta_{v_1,p}=\Delta_{v_2,p}$ if and only if there exists $a\in\Q_p^\times$ with $v_{1}=a\cdot v_{2}$.
\end{remark}

\subsubsection{Locally finite, rational quadratic divisors}

\begin{definition}
A  formal sum
$$ \Delta := \sum_{v\in V_{+}} a_v \cdot \Delta_{v,p}$$
of rational quadratic divisors on $X_p$ is said to be  {\em locally finite} if it satisfies the following equivalent conditions:
\begin{enumerate}
\item\label{lfrq1} For each affinoid  subset $\cA\subseteq X_p$, the formal sum
$$ \Delta \cap \cA := \sum_{\Delta_{v,p}\cap \cA\ne \emptyset} 
a_v \cdot \Delta_v$$
is a divisor, that is a finite linear combination.
\item\label{lfrq2}
For one (and thus for each) self-dual lattice $\Lambda\subseteq V_{\Q_p}$ and each $k\geq 0$, the formal sum $\Delta\cap X_{p,\Lambda}^{\leq k}$ is a divisor.
\end{enumerate}
\end{definition}
Write $\Div(X_p)$ for the module of locally finite, rational quadratic divisors.
The group $G(\Q)$ naturally acts on $\Div(X_p)$.
The equivalence of \eqref{lfrq1} and \eqref{lfrq2} follows from Lemma \ref{affinoidlemma}.

Let $\latt\subseteq V$ be the $\Z[1/p]$-lattice fixed in the introduction.
The most common construction of  locally finite divisors 
 rests on
the following lemma:
\begin{lemma}
\label{lemma:arch-padic}
Let $C$ be a compact subset of $X_\infty$, $m$ a positive rational number, and $\cA$ an affinoid subset of $X_p$.
Then the set
$$ X \  = \   \{ v\in \latt\ \vert\ q(v)=m, \ \Delta_{v,\infty}\cap C \ne \emptyset,\ \mbox{and} \
\Delta_{v,p} \cap  \cA \ne \emptyset \} $$ 
is finite.
\end{lemma}
\begin{proof}
Let $\Lambda\subseteq V_{\Q_p}$ a self-dual $\Z_p$-lattice.
Since any affinoid subset is contained in one of the form $X_{p,\Lambda}^{\le k}$,  we may assume that $\cA = X_{p,\Lambda}^{\le k}$  without loss of generality.
By Lemma \ref{Archimedeanlemma},  $X$ is contained in a compact subset of $V_\R$, and by Lemma \ref{padiclemma}, it is contained in  
$p^{-l} \Lambda \cap \latt$ for some $l\in \Z$, a discrete subset of $V_\R$. 
The finiteness of $X$ follows.
\end{proof}

\subsubsection{Compactly supported products}
Let $\cO\subseteq V_{+}$ be a finite union of $\Gamma$-orbits, such as 
the set of all $v\in \latt$ satisfying $q(v)=m$  for a given $m$.
Lemma \ref{lemma:arch-padic}
motivates the following definition:
\begin{definition}
Let $M = (M_v)_{v\in \cO}$ be a $\Z$-linear representation of $\cO$.
The \emph{compactly supported product} of the abelian groups $M_v$ is the $\Gamma$-submodule of
$ \prod_{v\in \cO} M_{v}$
given by
\begin{equation*}
  \sideset{}{'}\prod_{v\in \cO} M_{v} =
      \left\{
			(m_v)\in \prod_{v\in \cO} M_{v}\ \middle\vert
				\begin{array}{l}\mbox{there exists}\ C\subseteq X_\infty \text{ compact} \\
       \mbox{such that}\ \Delta_{v,\infty}\cap C =\emptyset \Rightarrow m_v=0 \end{array}
    \right\}.
\end{equation*}
\end{definition}

Let $(\Z)_{v\in \cO}$ be the representation of $\cO$, whose objects are all equal to $\Z$ and whose arrows attached to $\gamma\in \Gamma$ are the identity maps.
After multiplying $\cO$ with a positive rational number we may assume that $\cO$ is contained in $\latt$.
Lemma \ref{lemma:arch-padic} implies that the natural $\Gamma$-equivariant
homomorphism
$$\cZ\colon \sideset{}{'}\prod_{v\in \cO} \Z \longrightarrow \Div(X_p),\quad (a_v)_{v\in\cO}\longmapsto \sum_{v\in \cO} a_v  \cdot\Delta_{v,p}$$
is well-defined.
 
For later purposes let us discuss an important feature of compactly supported products.
The \emph{compactly supported product} of a chain complex $(M_{\sbullet,v})_{v\in\cO}$ of $\cO$-re\-pre\-sen\-tations is the subcomplex
$$\sideset{}{'}\prod_{v\in \cO} M_{{\sbullet},v}\subseteq \prod_{v\in \cO} M_{{\sbullet},v}$$
defined by taking the compactly supported product in each degree.
One easily deduces that taking compactly supported products commutes with taking homology:
\begin{lemma}\label{homologyofrestricted}
For all $q\geq 0$,
$$H_q \left(\sideset{}{'}\prod_{v\in \cO} M_{{\sbullet},v} \right)=\sideset{}{'}\prod_{v\in \cO} H_q(M_{{\sbullet},v}).$$
\end{lemma}

\subsection{Kudla--Millson divisors}
 \label{kudlamillsondivisors}
Suppose for the moment that $s=0$ and  that $\Gamma$ is torsion-free.
Then $\Gamma$ acts discretely on the $p$-adic locally symmetric space $X_p$ and the quotient $\Gamma\backslash X_p$ is a smooth rigid analytic variety.
We may identify $\Div (X_p)^{\Gamma}$ with the space of those divisors on $\Gamma\backslash X_p$ whose pullback to $X_p$ is a (locally finite) rational quadratic divisor.

For general $s$ the ``Archimedean components'' of rational quadratic divisors define  real codimension $s$ cycles on the Archimedean locally symmetric space $X_\infty$.
This suggests studying classes in $H^s(\Gamma, \Div(X_p))$.
We will attach to every subset $\cO\subseteq V_+$ that is a finite union of $\Gamma$-orbits a class $\sD_\cO\in H^s(\Gamma, \Div(X_p))$. 
Roughly speaking, $\sD_\cO$ should arise from the assignment
\begin{align}\label{hybridformula}
C_s(X_\infty)\longrightarrow \Div(X_p),\qquad c\longmapsto \sum_{v\in\cO}(c\cap \Delta_{v,\infty}) \cdot \Delta_{v,p}
\end{align}
where $c$ is a $s$-chain on $X_\infty$ and $(\Delta_{v,\infty}\cap c)$ is the signed count of intersection points of $c$ and $\Delta_{v,\infty}$.
The main obstacle to making this rigorous is arises from 
the possibility that   $c$ and $\Delta_{v,\infty}$ might
 not intersect transversally. In what follows, this difficulty is overcome  by constructing a $G(\Q)$-invariant subcomplex $\mathfrak{C}_{\sbullet}\subseteq C_{{\sbullet}}(X_\infty)$ for which
\begin{itemize}
\item the inclusion $\mathfrak{C}_{\sbullet}\hookrightarrow C_{{\sbullet}}(X_\infty)$ is a quasi-isomorphism and
\item every $c\in \mathfrak{C}_s$ intersects every Archimedean cycle $\Delta_{v,\infty}$ nicely.
\end{itemize}

\subsubsection{A subcomplex of the singular chain complex}\label{transversal}
Every $\gamma\in G(\Q)$ induces homomorphisms 
\begin{align*}
\gamma\colon C_{\sbullet}(X_\infty) &\longrightarrow C_{\sbullet}(X_\infty),\\
\intertext{as well as homomorphisms} 
\gamma\colon C_{\sbullet}(X_\infty, X_\infty\Setminus\Delta_{v,\infty}) &\longrightarrow C_{\sbullet}(X_\infty, X_\infty\Setminus\Delta_{\gamma v,\infty})
\end{align*}
 for each
$v\in V$ of positive length,
which fulfill relation \eqref{compatibility}. 
Hence $(C_{\sbullet}(X_\infty, X_\infty\Setminus\Delta_{\gamma v,\infty}))_{v\in V_+}$ is a representation of  the $\Gamma$-module $V_+$.
The product of the quotient maps
$$C_{\sbullet}(X_\infty)\longrightarrow C_{\sbullet}(X_\infty,X_\infty\Setminus \Delta_{v,\infty} )$$
defines a $G(\Q)$-equivariant homomorphism
$$r_{\sbullet}\colon C_{\sbullet}(X_\infty)\longrightarrow \prod_{v\in V_{+}} C_{\sbullet}(X_\infty,X_\infty\Setminus \Delta_{v,\infty} )$$
of chain complexes.

By Proposition \ref{relativehomology} the homology of the chain complex $$C_{\sbullet}(X_\infty,X_\infty\Setminus \Delta_{v,\infty} )$$
is concentrated in degree $s$.
Therefore, the inclusion
$$\tau_{\geq s} C_{\sbullet}(X_\infty,X_\infty\Setminus \Delta_{v,\infty} )\longrightarrow C_{\sbullet}(X_\infty,X_\infty\Setminus \Delta_{v,\infty} )$$
of the canonical $s$-truncation is a quasi-isomorphism.
Since taking homology commutes with products, the inclusion
$$\iota_{\sbullet}\colon \prod_{v\in V_{+}}\tau_{\geq s}C_{\sbullet}(X_\infty,X_\infty\Setminus \Delta_{v,\infty} )\longrightarrow \prod_{v\in V_{+}}C_{\sbullet}(X_\infty,X_\infty\Setminus \Delta_{v,\infty} )$$
is a quasi-isomorphism as well.

Define
$$\mathfrak{C}_{\sbullet} \subseteq C_{\sbullet}(X_\infty)$$
to be the pullback of $\prod_{v\in V_{+}}\tau_{\geq s}C_{\sbullet}(X_\infty,X_\infty\Setminus \Delta_{v,\infty} )$ along $r_{\sbullet}$.
More concretely,  
\begin{align*}
\mathfrak{C}_q=
\begin{cases}C_{q}(X_\infty) & \mbox{ for } q>s,\\
C_{q}(X_\infty \Setminus \bigcup_{v\in V_{+}}\Delta_{v,\infty})& \mbox{ for } q< s,
\end{cases}
\end{align*}
and
\[
\mathfrak{C}_s=\left\{c\in C_{s}(X_\infty)\ \middle\vert\  d_s(c)\in C_{s-1}(X_\infty \Setminus \textstyle\bigcup_{v\in V_{+}}\Delta_{v,\infty})\right\}.
\]

\begin{proposition}\label{qisom}
The inclusion
$$\mathfrak{C}_{\sbullet}\longrightarrow C_{\sbullet}(X_\infty)$$
is a quasi-isomorphism.
In particular, $\mathfrak{C}_{\sbullet}$ is a resolution of the constant $\Z[\Gamma]$-module $\Z$.
\end{proposition}

\begin{proof}
Since countable unions of proper closed submanifolds have measure zero, it follows that there exists a point
$$Z\in X_\infty \Setminus \bigcup_{v\in V, q(v)>0} \Delta_{v,\infty}.$$
Thus, surjectivity on homology follows from the fact that $X_\infty$ is contractible.
Injectivity on homology follows from the next general lemma on chain complexes.
\end{proof}

\begin{lemma}
Let $f_{\sbullet}\colon A_{\sbullet} \rightarrow B_{\sbullet}$ be a quasi-isomorphism of chain complexes and $g_{\sbullet}\colon C_{\sbullet} \rightarrow B_{\sbullet}$ a homomorphism of chain complexes.
Then the pullback
$$\widetilde{f}\colon A_{\sbullet}\times_{B_{\sbullet}} C_{\sbullet} \longrightarrow C_{\sbullet}$$
of $f$ along $g$ induces injective maps on homology.
\end{lemma}

\begin{proof}
Let $(x,y)\in A_n \times_{B_n} C_n$ be a cycle, 
i.e., assume that 
$$dx=0, \qquad dy=0\qquad f(x)=g(y).$$
Suppose that the image of $(x,y)$ in the homology of $C_n$ is equal to zero.
That means that there exists $z\in C_{n+1}$ such that $\widetilde{f}(x,y)=y=dz$ for some $z\in C_{n+1}$. 
Since $f(x)=g(y)=g(dz)=d(g(z))$ is a boundary and $f$ is a quasi-isomorphism, we see that $x$ is a boundary as well.
Thus there exists an element $u\in A_{n+1}$ such that $x=du$.
It follows that $d(u,z)=(x,y)$ is a boundary, which proves the claim.
\end{proof}

\subsubsection{Divisor valued cohomology classes attached to $\Gamma$-orbits}\label{KMdivisorsfromorbits}
Let $\cO\subseteq V_{+}$ be a finite union of $\Gamma$-orbits.
By definition the image of the $\Gamma$-equivariant map
$$r_{\sbullet}^{\cO}\colon C_{\sbullet}(X_\infty)\longrightarrow \prod_{v\in\cO} C_{\sbullet}(X_\infty,X_\infty\Setminus \Delta_{v,\infty})$$
is contained in the compactly supported product.
Thus, restricting this map to $\mathfrak{C}_{\sbullet}$ gives
 a homomorphism
$$r_{\sbullet}^{\cO}\colon \mathfrak{C}_{\sbullet} \longrightarrow \sideset{}{'}\prod_{v\in\cO} \tau_{\geq s} C_{\sbullet}(X_\infty,X_\infty\Setminus \Delta_{v,\infty}).$$
The natural quotient map to homology induces a $\Gamma$-equivariant quasi-isomorphism
$$h_{\sbullet}\colon \sideset{}{'}\prod_{v\in \cO}\tau_{\geq s}C_{\sbullet}(X_\infty,X_\infty\Setminus \Delta_{v,\infty}) \longrightarrow \sideset{}{'}\prod_{v\in \cO}H_s(X_\infty,X_\infty\Setminus \Delta_{v,\infty})[-s].$$
Moreover, the identification \eqref{orientation} induces a $\Gamma$-equivariant isomorphism
$$t\colon \sideset{}{'}\prod_{v\in \cO}H_s(X_\infty,X_\infty\Setminus \Delta_{v,\infty})\xlongrightarrow{\cong}\sideset{}{'}\prod_{v\in \cO}\Z.$$
The composition
$$\cZ\circ t \circ h_{\sbullet} \circ r_{\sbullet}^{\cO}\colon \mathfrak{C}_{\sbullet} \longrightarrow \Div(X_p)[-s]$$
defines a morphism from $\Z$ to $\Div(X_p)[-s]$ in the derived category of $\Z[\Gamma]$-modules
or, in simpler terms, a class
$$\sD_\cO^{\Gamma}\in H^s(\Gamma, \Div(X_p)).$$

Let us unravel this construction. 
It is convenient to introduce the following signed intersection number for $c\in \mathfrak{C}_s$ and $\Delta_v$, $v\in V_{\R,+}$:
by definition  
$$d_s(c)\in \Ker(d_{s-1}\colon C_{s-1}(X_\infty\Setminus \Delta_{v,\infty})\rightarrow C_{s-2}(X_\infty\Setminus \Delta_{v,\infty})).$$
The intersection number $c \cap \Delta_{w,\infty}$ is the image of $d_s(c)$ in the reduced homology group $\tilde{H}_{s-1}(X_\infty\Setminus \Delta_{v,\infty})$ that we identify with $\Z$ by \eqref{orientation}.
If $c$ is a smooth and does intersect $\Delta_{v,\infty}$ transversely, this is simply the signed count of intersection points of $c$ with $\Delta_{v,\infty}$.
Now the class $\sD_\cO^{\Gamma}$ is in fact given by the $s$-cocycle
\[
\mathfrak{C}_s\longrightarrow \Div(X_p),\qquad c\longmapsto \sum_{v\in\cO}(c \cap \Delta_{v,\infty}) \cdot \Delta_{v,p}.
\]

\subsubsection{Restriction to a subgroup}
Let $\Gamma'\subseteq \Gamma$ be a finite index subgroup and
\[\res\colon H^s(\Gamma, \Div(X_p))\longrightarrow H^s(\Gamma', \Div(X_p))\]
the restriction map on cohomology.
The group $\Gamma'$ fulfills all the conditions that were
 imposed on $\Gamma$.
Moreover, the set $\cO$ is also a finite union of $\Gamma'$-orbits.
Thus, we can define the class
$$\sD_\cO^{\Gamma'}\in H^s(\Gamma', \Div(X_p)).$$
The following proposition follows directly from the construction.
\begin{proposition}
Let $\Gamma'\subseteq \Gamma$ be a finite index subgroup.
Then  
\[
\res(\sD_\cO^{\Gamma})=\sD_\cO^{\Gamma'}.
\]
\end{proposition}
Since the class $\sD_\cO^{\Gamma}$ does not depend on the group $\Gamma$ preserving $\cO$, we adopt the abbreviation
$$\sD_\cO:=\sD_\cO^{\Gamma}.$$

\begin{remark}
If $\cO$ decomposes into the union of $\Gamma$-orbits $\cO_1,\ldots,\cO_h$, then  
\[
\sD_\cO =\sum_{i=1}^{h}\sD_{\cO_i}.
\]
Moreover,   $\sD_{a\cO}=\sign(a)^{s}\cdot \sD_{\cO}$ for every $a\in \Q^\times$. 
If $V$ is three-dimensional and $\cO$ consists of a single $\Gamma$-orbit, the class $\sD_\cO$ is obviously zero if $\Delta_v=\emptyset$ for one (and thus any) $v\in\cO$. 
It is natural to wonder whether the latter relations generate all relations between the divisors attached to $\Gamma$-orbits in $V_{+}$..
This is known in signature $(3,0)$ and $(2,1)$. 
The case that $V$ is the split three-dimensional quadratic space and $\Gamma=\SL_2(\Z[1/p])$ is implicitly treated in \cite{DV1}.
For the general case see \cite{Ge-quaternionic}.
\end{remark}

\subsubsection{Kudla--Millson divisors attached to Schwartz functions}
The class $\sD_\cO$ attached to a $\Gamma$-orbit $\cO\subseteq V_+$ should be viewed as an analogue of the connected cycles on orthogonal Shimura varieties introduced by Kudla in \cite[Section 3]{kudla-cycles}.
In the following we introduce the analogue of weighted cycles in the sense of \cite[Section 5]{kudla-cycles}.
Let $\A^{p}_{f}$ denote the ring of finite adeles away from $p$, that is, the restricted product over all completions $\Q_\ell$ for all rational primes $\ell\neq p$.
Put $V_{\A^{p}_{f}}:=V\otimes_\Q \A^{p}_{f}$ and write
\[
S(V_{\A^{p}_{f}}):=\{\Phi\colon V_{\A^{p}_{f}}\rightarrow \Z\ \vert\ f\ \mbox{locally constant with compact support}\}
\]
for the space of $\Z$-valued Schwartz functions on $V_{\A^{p}_{f}}$ with its natural $G(\A^{p}_{f})$-action.
The group $G(\Q)$ acts on $S(V_{\A^{p}_{f}})$ via the diagonal embedding $G(\Q)\hookrightarrow G(\A^{p}_{f})$.
Let $\Phi\in S(V_{\A^{p}_{f}})$ be a $\Gamma$-invariant Schwartz function and $m$ a positive rational number.
For every non-zero integer $r$ the set $\cO(m,\Phi,r):=\{v\in V\ \vert\ q(v)=m,\ \Phi(v)=r\}$ is a finite union of $\Gamma$-orbits. Moreover, it is empty for all but finitely many $r$.
Thus, the cohomology class
\[
\sD_{m,\Phi}:= \sum_{r\in\Z\Smallsetminus\{0\}} r\cdot \sD_{\cO(m,\Phi,r)}
\]
is well-defined.
Its definition is inspired by the formula for the restriction of a weighted cycle to a connected component of an orthogonal Shimura variety (cf.~\cite[Proposition 5.4]{kudla-cycles}).
\begin{definition}
The space of Kudla--Millson divisors of level $\Gamma$ is the subspace
\[\mathcal{KM}(\Gamma)\subseteq H^s(\Gamma, \Div(X_p))\]
generated by the classes $\sD_{m,\Phi}$ with $m\in \Q_{>0}$ and $\Phi\in S(V_{\A^{p}_{f}})^\Gamma$.
\end{definition}

\subsubsection{Kudla--Millson divisors attached to cosets}
The following class of Kudla--Millson divisors features prominently in the construction of $p$-adic Borcherds products in Section \ref{maintheorems}.
Assume for the moment that $\Gamma$ acts trivially on the discriminant module $\D_{\latt}$. 
Let $\widehat{L}$ be the completion of the lattice $L$ inside $V_{\A^{p}_{f}}$.
Then $\beta+\widehat{L}$ is a $\Gamma$-invariant compact open subset of $V_{\A^{p}_{f}}$ and, thus, the characteristic function $\mathbf{1}_{\beta+\widehat{L}}$ is a $\Gamma$-invariant Schwartz function.
\begin{definition}
The Kudla--Millson divisor 
$\sD_{m,\beta}\in \mathcal{KM}(\Gamma)$
is defined as
\[
\sD_{m,\beta}:= \sD_{m,\mathbf{1}_{\beta+\widehat{L}}}.
\]
\end{definition}
Denote by $\cO_\latt(m,\beta)$ the set of vectors $v\in \beta+\latt$ such that $q(v)=m$.
The equality 
\[
\sD_{m,\beta}=\sD_{\cO_L(m,\beta)}.
\]
follows directly from the definition.
If $v\in \beta$, then clearly $-v\in -\beta$.
Since $q(v)=q(-v)$, the equality
\[
\sD_{m,\beta}=(-1)^{s}\sD_{m,-\beta}
\]
follows from the discussion of orientations in Section \ref{sec-orientations}.
Moreover, a simple calculation shows that Kudla--Millson divisors are $p$-ordinary in the following sense:
\begin{proposition}\label{ordinarityprop}
For all $\beta\in {\mathbb D}_{\latt}$ and all positive rational numbers $m$ :
\[
\sD_{p^2m,p\beta}=\sD_{m,\beta}.
\]
\end{proposition}

\subsection{Explicit cocycles for small values of $s$}\label{explicitcocycles}
By unravelling the definition we see that in case $s=0$ the class $\sD_{\cO}$ is simply the $\Gamma$-invariant divisor
$$\sD_{\cO}=\sum_{v\in\cO} \Delta_{v,p}.$$
In particular, it is non-zero whenever $\Delta_{v,p}\neq\emptyset$ for one (and thus for all) $v\in\cO$. 

\subsubsection{Transversal base points}
When $s\geq 1$, the complex $\mathfrak{C}_{\sbullet}$ is rather big.
For both computational and theoretical purposes it is desirable to replace it by something more manageable.
\begin{definition}
A {\em transversal base point}
 is a pair $(P_{\sbullet},p_{\sbullet})$ consisting of a projective resolution $P_{\sbullet}\xrightarrow{\epsilon} \Z$ of the trivial $\Gamma$-module $\Z$ and a $\Gamma$-equivariant quasi-isomorphism
$$p_{\sbullet} \colon P_{\sbullet} \longrightarrow \mathfrak{C}_{\sbullet}$$
that induces the identity on $\Z$.
\end{definition}
Let $P_{\sbullet}$ be a projective resolution of $\Z$.
By Proposition \ref{qisom} the chain complex $\mathfrak{C}_{\sbullet}$ is a (non-projective) resolution of $\Z$.
Thus, by \cite[Theorem 2.2.6]{Weibel},  there exists a $\Gamma$-equivariant quasi-isomorphism
$$p_{\sbullet} \colon P_{\sbullet} \longrightarrow \mathfrak{C}_{\sbullet}$$
inducing the identity on $\Z$.
Moreover, any two such maps are homotopic.

Let $(P_{\sbullet},p_{\sbullet})$ be a transversal base point.
For $c\in P_s$ and $v\in V_{\R,+}$, define the intersection product
\[
c\cap_{p_\sbullet}  \Delta_{v,\infty}= p_s(c)\cap \Delta_{v,\infty}.
\]
Then for $s\geq 1$ a cocycle representing the class $\mathcal{D}_{\cO}$ is given by the map
$$P_s\longrightarrow \Div(X_p),\quad c \longmapsto \sum_{v\in\cO} (c\cap_{p_\sbullet} \Delta_{v,\infty}) \cdot \Delta_{v,p}.$$

\subsubsection{Barycentric simplices}
A natural candidate for $P_{\sbullet}$ is the bar resolution $B(\Gamma)_{\sbullet}$ of $\Z$ given by
$$B(\Gamma)_q=\Z[\Gamma^{q+1}]$$
and
$$d_q([\gamma_0,\ldots,\gamma_q])=\sum_{i=0}^{q}(-1)^{i}[\gamma_0,\ldots,\hat{\gamma_i},\ldots, \gamma_q].$$

Every point $x\in X_\infty$ gives rise to a map of chain complexes of $\Z[\Gamma]$-modules
$$b^x_{\sbullet}\colon B(\Gamma)_{\sbullet} \longrightarrow C_{\sbullet}(X_\infty)$$
sending $(\gamma_0, \ldots , \gamma_q)\in B(\Gamma)_q$ to the {\em barycentric simplex} $[\gamma_0 x,\ldots, \gamma_q x]$ with corners $\gamma_0 x, ... , \gamma_q x$.
Let us briefly recall  the construction of barycentric simplices:
  let 
  $$ \Delta_q:= \{ (t_0,t_1,\ldots, t_q)  \  \mbox{ s.t. }  \ 
  0\leq t_j \leq 1,  \ \  \ t_0+t_1 + \cdots + t_q =1 \} \subseteq \R^{q+1}$$
  denote the standard $q$-dimensional simplex in $\R^{q+1}$. The choice of a $(q+1)$-tuple of points 
  $(x_0, x_1,\ldots, x_q)$ in $X_\infty$ determines a map
  $$ \Phi_{x_0,\ldots,x_q}\colon \Delta_q\longrightarrow X_\infty$$
  sending $(t_0,t_1,\ldots, t_q)$ to the unique minimum\footnote{The function $g$ is proper and strictly convex on $X_\infty$.  Thus, it achieves its minimum value at exactly one point.} of the real-valued function $g\colon X_\infty \rightarrow \R$ defined by 
  $$ 
  g(x) := t_0 d(x,x_0)^2 + t_1 d(x,x_1)^2 + \cdots + t_q d(x,x_q)^2.$$
 The   map $\Phi_{x_0,\ldots,x_s}$  maps the vertices of
 $\Delta_q$  to the points $x_0,\ldots, x_q$, and we set
 $$ [x_0,\ldots, x_q] := \Phi_{x_0,\ldots, x_q}\in C_q(X_\infty).$$
 The boundary of this  $q$-dimensional simplex is given by
 $$ \partial([x_0,\ldots,x_q]) = \sum_{i=0}^q (-1)^i [x_0,\ldots, \hat{x_i},\ldots x_q],$$
 where $[x_0,\ldots, \hat{x_i},\ldots x_q]$ denotes the $(q-1)$-dimensional simplex in $X_\infty$ obtained by removing the $i$-th vertex from $[x_0,\ldots, x_q]$.

\subsubsection{Transversal base points in signature $(r,1)$}
It is natural to ask whether there exists a point $x\in X_\infty$ 
for which the map $b^x_{\sbullet}$ is a transversal base point, i.e., factors over $\mathfrak{C}_{\sbullet}$.
This question can be answered affirmatively if $s=1$.
Indeed, in that case the only condition required
 of $x$ is that it not be
  an element of the union of all $\Delta_{v,\infty}$, $v\in\cO$, and it was already observed in the proof of Proposition \ref{qisom} that such points exist.
So in case $s=1$ we may chose $x$ as above and get the following description of the class attached to the orbit of $v$:
the space $X_\infty\Setminus \Delta_{v,\infty}$ decomposes into two connected components $H_v^+$ and $H_v^-$ for all $v\in \cO$,
where the choice of $H_v^+$ is determined by the chosen orientations.
Then the intersection product is given by
\begin{align}
\label{hyperbolicclasses}
[\gamma_0,\gamma_1]\cap_{b^x_{\sbullet}} \Delta_{v,\infty}= 
\begin{cases}
\phantom{-}0 & \mbox{if }  \  \gamma_0 x \mbox{ and }  \gamma_1 x \mbox{ both belong to }  H_v^+  \mbox{ or }  H_v^-;\\
\phantom{-}1 & \mbox{if }\ \gamma_0 x\in H_v^+ \mbox{ and }  \gamma_1 x \in H_v^-;\\
-1 & \mbox{if }\ \gamma_0 x\in H_v^- \mbox{ and } \gamma_1 x \in H_v^+.
\end{cases}
\end{align}

\subsection{Modular symbols} \label{modularsymbols}
Assume that the signature of $V$ is $(r,1)$.
The locally symmetric space $\Gammao\backslash X_\infty$ is not compact for one (and therefore any) arithmetic subgroup $\Gammao\subseteq G(\Q)$ if and only if $Q(\Q)\neq \emptyset$, i.e., $V$ has an isotropic vector.
In this case all maximal isotropic subspaces of $V$ are one-dimensional, i.e., given two distinct elements $[w_1], [w_2]\in Q(\Q)$ we  have
$$\langle w_1, w_2 \rangle \neq 0.$$
For the remainder of this section the assumption $Q(\Q)\neq \emptyset$ is made.
Note that by Meyer's theorem this is automatic in case $r \geq 4$.

The group $\Gamma$ acts on the set $Q(\Q)$ via the inclusion $\Gamma\subseteq G(\Q)$.
Therefore, it acts naturally on the free abelian group $\Z[Q(\Q)]$ generated by $Q(\Q)$ and its subgroup
$$\Z[Q(\Q)]_0:=\left\{ \sum n_{\ell}\cdot \ell \in \Z[Q(\Q)]\ \middle\vert\ \sum n_{\ell}=0 \right\}.$$

\begin{definition}
Let $M$ be any $\Gamma$-module.
The space of \emph{$M$-valued modular symbols} is the $\Gamma$-module
$$\mathrm{MS}(M):=\Hom_{\Z}(\Z[Q(\Q)]_{0}, M).$$
\end{definition}
Let
$$\delta\colon \mathrm{MS}(\Div (X_p))^{\Gamma} \longrightarrow H^1(\Gamma, \Div (X_p))$$
be the boundary map of the long exact sequence obtained by 
applying the functor $\Hom(-, \Div (X_p))$ to the short exact sequence
\begin{align}\label{symbolexact}
0 \longrightarrow \Z[Q(\Q)]_0\longrightarrow \Z[Q(\Q)]\longrightarrow \Z\longrightarrow 0
\end{align}
and  taking $\Gamma$-invariants.

We now discuss how  certain Kudla--Millson divisors can be lifted to $\Gamma$-invariant modular symbols.
This bridges the gap to the formulation of the theory in terms of modular symbols as formulated in \cite{DV1}.
It also gives a criterion for showing that Kudla--Millson divisors are non-zero.
The following lemma shows that the associated modular 
symbol,  if it exists, is  unique.
 \begin{lemma}
 \label{modsymbinjectivity}
 The map
 $$\delta\colon \mathrm{MS}(\Div (X_p))^{\Gamma} \longrightarrow H^1(\Gamma, \Div (X_p))$$
 is injective.
 \end{lemma}
  \begin{proof}
	Analyzing the long exact sequence induced by the short exact sequence \eqref{symbolexact}
	we see that it is enough to show that
	$$\Hom_\Z(\Z[Q(\Q)],\Div (X_p))^{\Gamma}=0.$$
	The stabilizer of each element $w\in Q(\Q)$ defines a parabolic subgroup $P_w\subseteq G$.
	Let $U_w\subseteq P_w$ be its unipotent radical.
	We put $\Gamma_{w}= P_w(\Q)\cap \Gamma$.
	By \cite[Theorem 7.3]{BoCN},  the group $\Gamma$ acts on $Q(\Q)$ with finitely many orbits $\Gamma w_1,\ldots,
	\Gamma w_h$.
	Thus it is enough to show that
	$$\Hom_\Z(\Z[\Gamma w_i],\Div (X_p))^{\Gamma}=0$$
	for every $i=1,\ldots, h$.
	
	For each $w\in Q(\Q),$ the $\Z[\Gamma]$-module $\Hom_\Z(\Z[\Gamma w],\Div (X_p))$ can be identified with
	the coinduction of $\Div (X_p)$ from $\Gamma_w$ to $\Gamma$.
	Thus, by Shapiro's Lemma, we see that it is enough to prove that
	$$\Div (X_p)^{\Gamma_w}=0.$$
	
	The parabolic $P_w$ naturally acts on the $\Q_p$-variety
	\[Q_{w}:=Q_{\Q_p}\Setminus w^{\perp}.\]
	We claim that $Q_w$ is a principal homogeneous $U_w$-space.
	This would immediately imply that
	\[\Div(X_p)^{U_w(\Q_p)}=0\]
	and, as strong approximation holds for unipotent groups, that
  \[\Div(X_p)^{\Gamma_w\cap U_w(\Q_p)}=0.\]
	The claim can be deduced from the Bruhat decomposition.
	In the following we sketch an elementary proof of the claim.
	
	The parabolic group $P_w$ acts on the space $V'=w^\perp/\Span \{ w \}$.
	It is not hard to see that the unipotent radical $U_w$ consists of all elements in the kernel
	$\ker(P_w\rightarrow \GL(V'))$ that act trivially on $w$.
	Let us fix an isotropic vector $w^{\prime}$ with
	\[
	\langle w,w'\rangle=1
	\]
	and put $W=\Span \{ w,w' \}^{\perp}$.
	Let $K$ be any field extension of $\Q_p$.
	We may write any vector $v\in V_{K}$ uniquely as $v=a\cdot w+ b\cdot w' + u$ with $a,b\in K$ and $u$ in $W_K$.
	The condition that $v$ is not in the orthogonal complement of $w$ is then equivalent to $b$ being non-zero.
	So every $\xi \in Q_{w}(K)$ is represented by a unique element of the form $v_\xi= a\cdot w + w' + u$.
	Moreover, since the line $\xi$ is isotropic we see that $a=-q(u)$.
	Conversely, the line spanned by $-q(u)w + w' + u$ is an element of $Q_w(K)$ for every $u\in W_K$.
	Thus, we may identify $Q_w$ with the affine space associated to $W$.
	What remains is to construct a unique transformation $f_u\in U_w(K)$ such that
	\[f_u(w')=-q(u)w + w' + u.\]
	Remember that such a transformation necessarily has to fulfil $f(w)=w$.
	Therefore, it remains to define the restriction of $f_u$ to $W_K$.
	There are two cases:
	first assume that $q(u)\neq 0$.
	Then we put $f_u(u)=u+q(u)w$ and $f_u(x)=x$ for all $x\in W_R$ perpendicular to $u$.
	Secondly, if $u$ is isotropic, we choose an isotropic vector $u'\in W_R$ such that \[
	\langle u,u'\rangle=1
	\]
	We put $f_u(u)=u$, $f_u(u')=u'+w$ and $f_u(x)=x$ for every $x$ perpendicular to $\Span(w,w',u,u').$
	In both cases $f_u$ is an orthogonal transformation that fixes $w$ and acts trivially on $V'_K$.
	\end{proof}

\subsubsection{Modular symbols and intersections with geodesics}
Given distinct isotropic lines $\ell_{-},\ell_{+}\in Q(\mathbb{Q})$ we let $\Pi(\ell_{-},\ell_{+})_\R\subseteq V_\R$ denote the $\mathbb{R}$-plane spanned by $\ell_{-}$ and $\ell_{+}$.
The set
$$[\ell_{-},\ell_{+}]=\{Z\in X_\infty\ \vert\ Z\subseteq \Pi(\ell_{-},\ell_{+})\}$$
defines a one-dimensional flat subspace of $X_\infty$.
Note that $X_\infty$ consists of the negative lines of $V_\R$.
In particular, there is a natural embedding of $X_\infty$ into $\mathbb{P}(V_\R)$.
We may compactify $X_\infty$ by adding $Q(\R)$ as the boundary.
Let $\overline{X}_\infty$ be the resulting closed subspace of $\mathbb{P}(V_\R)$.
Then the closure of $[\ell_{-},\ell_{+}]$ in $\overline{X}_\infty$ is just the geodesic from $\ell_{-}$ to $\ell_{+}$ and we equip it with the induced orientation.

Let $v\in V$ be a vector of positive length such that its orthogonal complement in $V$ is anisotropic.
This implies that the cycles $\Delta_{v,\infty}$ and $[\ell_{-},\ell_{+}]$ intersect transversally and we write $\left( \Delta_{v,\infty} \cap [\ell_{-},\ell_{+}] \right) $ for its signed intersection number.
By the assumption that $v^\perp$ is anisotropic over $\Q$, the closure $\overline{\Delta}_{v,\infty}$ of $\Delta_{v,\infty}$ in $\overline{X}_\infty$ does not intersect $Q(\Q)$.
The complement of  $\overline{\Delta}_{v,\infty}$ in $\overline{X}_\infty$ decomposes into two connected components $\overline{H}_v^{+}$ and $\overline{H}_v^{-}$.
The intersection number is zero if and only if $\ell_{-}$ and $\ell_{+}$ are on the same connected component.
Choosing generators $w_{-}$ and $w_{+}$ of $\ell_{-}$ and $\ell_{+}$ with $\langle w_{-}, w_{+}\rangle < 0$ this happens if and only if
$\langle v, w_{-} \rangle$ and $\langle v, w_{+} \rangle$ have opposite sign.
 Concretely the intersection number can be computed by
\begin{align}\label{compacthyperbolicclasses}
 \Delta_{v,\infty} \cap [\ell_{-},\ell_{+}]
=\begin{cases}
0 & \mbox{if}\ \ell_{-}, \ell_{+} \in \overline{H}_v^{+}\ \mbox{or}\ \ell_{-}, \ell_{+}\in \overline{H}_v^{-}\\
1 & \mbox{if}\ \ell_{-}\in \overline{H}_v^{+}, \ell_{+} \in \overline{H}_v^{-}\\
-1 & \mbox{if}\ \ell_{-}\in \overline{H}_v^{-}, \ell_{+} \in \overline{H}_v^{+}
\end{cases}
\end{align}

\subsubsection{Finiteness of intersections}

 \begin{lemma}\label{modularsymbolintersection}
 Fix a pair of distinct elements $[w_{-}],[w_{+}]\in Q(\Q)$.
 We may assume that $\langle w_{-}, w_{+}\rangle = -1.$
 Fix a rational number $d > 0$ and a $\Z$-lattice $L_\circ$ in $V$. 
 There are only finitely many vectors $v \in L$ satisfying:
  \begin{itemize}
  \item $q(v)= d,$
  \item $\langle v, w_{-} \rangle, \langle v, w_{+} \rangle$ are both non-zero and have opposite sign.
  \end{itemize}   
 \end{lemma}
  \begin{proof}
  Let $\Pi$ be the subspace generated by $w_{-}$ and $w_{+}$.
	The plane $\Pi$ is hyperbolic and the restriction of $\langle \cdot,\cdot\rangle$ to $\Pi$ is non-degnerate.
	So there is an orthogonal decomposition $V = \Pi \oplus \Pi^{\perp}$ and $\Pi^{\perp}$ is positive definite.
	We may write $v = v^{||} + v^{\perp}$ with respect this decomposition.
	The projection $v^{||}$ equals  
  \[
	v^{||} = - \left(  \langle v,w_{+} \rangle w_{-} + \langle v, w_{-} \rangle w_{+} \right).
	\]

	We calculate:
  \begin{align*}
  d &= q(v) \nonumber \\
  &= q(v^{\perp}) + q \left(  - \left(  \langle v,w_{+} \rangle w_{-} + \langle v, w_{-} \rangle w_{+} \right) \right) \nonumber \\
  &= q(v^{\perp}) +   \left( \langle v,w_{+} \rangle \cdot \langle v, w_{-} \rangle \cdot \langle w_{-},w_{+} \rangle \right) \nonumber \\
  &= q(v^{\perp}) - \langle v,w_{+} \rangle \cdot \langle v, w_{-} \rangle \nonumber \\
  &= q(v^{\perp}) +   \left| \langle v,w_{+} \rangle \right| \cdot \left| \langle v, w_{-} \rangle \right|,
  \end{align*}
  where the last equality follows because $\langle v,w_{-} \rangle$ and $\langle v,w_{+} \rangle$ have opposite signs.
	There is some integer $N$ such that $N w_{-}, N w_{+} \in L.$  Then 
  \[
	-N^2 \cdot v^{\perp} = \langle v, N w_{+} \rangle N w_{-} + \langle v, N w_{-} \rangle N w_{+} \in L.
	\]

  So 
  \begin{align*}
  d N^4 &= q(-N^2 \cdot v^{\perp}) +  N^2  \cdot \left| \langle v,Nw_{+} \rangle \right| \cdot \left| \langle v, Nw_{-} \rangle \right| \\
  \implies d  N^4 &\geq q(-N^2 \cdot v^{\perp}) \text{ and } N^2  \cdot \left| \langle v,Nw_{+} \rangle \right| \cdot \left  \langle v, Nw_{-} \rangle \right|,
  \end{align*}
  since both summands are positive.
	There are only finitely many possibilities for the vector $v^\perp$ since $rN^2 \cdot v^{\perp}$ is a vector in the lattice $L_\circ \cap \Pi^\perp$ equipped with the quadratic form $Q$ which is positive definite on $\Pi^\perp.$  

  Also, $ \left| \langle v,Nw_{+} \rangle \right|$ and $ \left| \langle v, Nw_{-} \rangle \right|$ are positive integers (they are non-zero by assumption), and their product has bounded size.
  Thus there are only finitely many possibilities for both of $\langle v,w_{-} \rangle$ and $\langle v,w_{+}\rangle$ and hence only finitely many possibilities for $v^{||}.$
	Hence, there are only finitely many possibilities for $v = v^{||} + v^{\perp},$ proving the Lemma.
  \end{proof}

\subsubsection{Lifting Kudla--Millson divisors to modular symbols}

 \begin{proposition}\label{modsymbprop}
 Let $\cO$ be a $\Gamma$-orbit of vectors of positive length in $V$ such that the orthogonal complement in $V$ of one 
 (and thus every) $v\in\cO$ is anisotropic.
 For every pair of rational isotropic lines $\ell_{-}, \ell_{+} \in Q(\mathbb{Q})$ the formal sum
 \[
\sum_{v\in\widetilde{\sD}_{\cO}} \left( \Delta_{v,\infty} \cap [\ell_{-},\ell_{+}] \right) \cdot \Delta_{v,p} \in \Div(X_p)
\]
 is locally finite.
 Furthermore, the assignment
 \begin{align*}
 \widetilde{\sD}_{\cO}\colon Q(\mathbb{Q}) \times Q(\mathbb{Q}) &\rightarrow \Div(X_p) \\
 (\ell_{-}, \ell_{+}) &\mapsto \sum_{v \in \cO} \left( \Delta_{v,\infty} \cap [\ell_{-},\ell_{+}] \right) \cdot \Delta_{v,p}
 \end{align*}
 defines a $\Gamma$-invariant $\Div(X_p)$-valued modular symbol that lifts the divisor $\sD_{\cO}$, i.e.,  
 $$\delta(\widetilde{\sD}_{\cO})=\sD_{\cO}.$$
If in addition $\cO\neq -\cO $ and the orthogonal complement in $V$ of one (and thus every) $v\in \cO$ is not a hyperbolic plane, then the divisor valued cohomology class $\sD_{\cO}$ is non-zero.
 \end{proposition}
  \begin{proof}
	Replacing Lemma \ref{Archimedeanlemma} by Lemma \ref{modularsymbolintersection},  local  finiteness of the divisor 
	follows as before.
  That $\widetilde{\sD}_{\cO}$ defines a $\Gamma$-invariant modular symbol is a formal calculation reducing to equivariance of intersection numbers and of Kudla-Millson cycles.
	Comparing the two different intersection products \eqref{hyperbolicclasses}
	and \eqref{compacthyperbolicclasses} yields the equality $\delta(\widetilde{\sD}_{\cO})=\sD_{\cO}$.
	In order to show that $\sD_{\cO}$ is non-zero, it is enough to show that $\widetilde{\sD}_{\cO}$ is non-zero by Lemma \ref{modsymbinjectivity}.
	Remark \ref{divisorsequality} implies that $\Delta_{v_1,p}\neq \Delta_{v_2,p}$ for all vectors $v_1,v_2\in \cO$ with $v_1\neq v_2$.
	Thus, it is enough to prove the existence of $v\in\cO$ and $\ell_{-},\ell_{+}\in Q(\Q)$ such that $\Delta_{v,\infty}\cap[\ell_{-},\ell_{+}]\neq 0$.
	But this is obvious from the description \eqref{compacthyperbolicclasses} of the intersection product.
  \end{proof}

Whether the orthogonal complement of $v$ is anisotropic or not only depends on $q(v)$ by Witt's cancellation theorem.
This leads to the following definition
\begin{definition}\label{def-compactm}
A positive rational number $m\in \Q_{>0}$ is called {\em compact }with respect to $(V,q)$ if the orthogonal complement of one (and thus  all) $v\in V$ with $q(v)=m$ is anisotropic.
\end{definition}

Proposition \ref{modsymbprop} immediately implies that Kudla--Millson divisors attached to compact $m\in \Q_{>0}$ can be lifted to modular symbols:
\begin{corollary}\label{modsymbcor}
Assume that $\Gamma$ acts trivially on $\D_\latt$.
Let $m\in \Q_{>0}$ be compact with respect to $(V,q)$ and $\beta\in\D_\latt$.
Then there exists a $\Gamma$-invariant modular symbol $\widetilde{\sD}_{m,\beta}\in\mathrm{MS}(\Div(X_p))^{\Gamma}$ such that
\[
\delta(\widetilde{\sD}_{m,\beta})=\sD_{m,\beta}.
\]
\end{corollary}

  \begin{remark}
The discussion above is only applicable in small dimensions.
Indeed, by Meyer's theorem there are no vectors $v\in V$ of positive length whose orthogonal complement in $V$ is anisotropic if $r\geq 5$.
On the other hand in dimension $3$ the Kudla--Millson divisor $\sD_{\cO}$ of every $\Gamma$-orbit $\cO$ can be lifted to a modular symbol since, if the orthogonal complement of $v$ in $V$ contains an isotropic vector, then $\Delta_{v,p}$ is empty anyway.
\end{remark}

\section{Rigid meromorphic cocycles and $p$-adic Borcherds products}
\label{sec:rmc}
This chapter introduces the notion of \emph{rigid meromorphic cocycles} and constructs an analogue of Borcherds' singular theta lift under certain restrictive assumptions on the signature of $V$.

\subsection{Rigid meromorphic cocycles}
\subsubsection{Definitions}
Let $\rM^\times$ be the multiplicative group of rigid meromorphic functions on $X_p$ that are defined over $\Q_p$ and whose divisor belongs to $\Div(X_p)$.
The $G(\Q_p)$-equivariant homomorphism
\[
\divmap \colon \rM^\times \longrightarrow\Div(X_p)
\]
induces the divisor map
\[
\divmap_\ast \colon H^s(\Gamma,\rM^\times) \longrightarrow H^s(\Gamma,\Div(X_p)).
\]

\begin{definition}
A {\em rigid meromorphic cocycle of level $\Gamma$} is a  class in 
$J\in H^s(\Gamma, \rM^\times)$ whose image under the divisor map is a Kudla--Millson divisor.
Let
\[
\mathcal{RMC}(\Gamma)=\divmap_\ast^{-1}(\mathcal{KM}(\Gamma))
\]
denote the space of rigid meromorphic cocycles.
\end{definition}

\subsubsection{Lifting obstructions}
By definition, there is an exact sequence of $G(\Q_p)$-modules
\[
0 \longrightarrow \cA^\times \longrightarrow \rM^\times \xlongrightarrow{\divmap} \Div(X_p),
\]
with $\cA$ being the ring of rigid analytic functions on $X_p$ that are defined over $\Q_p$.
\begin{proposition}
\label{prop:divexact} 
The map $\divmap$ is surjective, i.e., the sequence
\begin{align}
0 \longrightarrow \cA^\times \longrightarrow \rM^\times \xlongrightarrow{\divmap} \Div(X_p) \longrightarrow 0 
\end{align}
is exact.
\end{proposition}
\begin{proof}
Let 
$$\sD\  := \ \sum_{v\in V} a_v \cdot \Delta_{v,p} \ \in  \ \Div(X_p)$$
be a locally finite, rational quadratic divisor.
Fix a self-dual $\Z_p$-lattice $\Lambda\subseteq V_{\Q_p}$.
By Lemma \ref{padiclemmaconverse} the locally finite divisor $\sD$ can be written as  a sum  
$$ \sD = \sum_{d=0}^\infty \sD_d, \qquad
\sD_{d}   := \sum_{\iso_{\Lambda}(v)=d} a_v \Delta_{v,p} $$
  of finite divisors $\sD_m$.
    Let $v$ be a vector that contributes to $\sD_{d}$.
		One may suppose that $v$ is primitive, 
and choose an isotropic vector
  $\tilde v\in \Lambda'$ satisfying 
  \begin{equation} 
  \label{congruencemodlevel}
\tilde v \equiv v \bmod{p^d}.
\end{equation}
 Fix a primitive representative $v_\xi\in \Lambda_{\cO_{\C_p}}'$ for every $\xi\in X_p$. 
  By definition of $X_p$, the expression $\langle \tilde v, v_\xi\rangle$ is
  non-zero for all $\xi\in X_p$.
	The function
 \begin{equation}
 \label{eqn:def-fv}
  f_v(\xi) := \frac{\langle v,v_\xi\rangle}{\langle \tilde v, v_\xi\rangle} 
  \end{equation} 
	is independent of the choice of $v_\xi$ and defines a rigid meromorphic function 
   of $\xi\in X_p$ with divisor  $\Delta_v$.
 If $k< d$, then  the restriction of $f_v$ to
   $X_{p,L}^{\leq k}$ is analytic and \eqref{congruencemodlevel} implies that
  $$ | f_v(\xi) -1| \leq p^{k-d}  \quad \mbox{ for all }   \xi\in X_{p,L}^{\leq k}.$$
  It follows that the rational functions
  $$ f_{\sD_d}(\xi) := 
  \prod_{\iso_\Lambda(v)=d} f_v(\xi)^{a_v} $$
  with divisor $\sD_d$ converge uniformly  to $1$  on any affinoid. 
  The
  infinite product
  $$ f_{\sD}(\xi) := \prod_{d = -\infty}^\infty f_{\sD_d}(\xi) $$
therefore   converges to a rigid meromorphic function on $X_p$ with divisor $\sD$. The lemma follows.
 \end{proof}

 \begin{remark}
 Note that  the rigid meromorphic function $f_\sD$ constructed in the proof
 is far from unique since it depends on
 the choice of a system of isotropic vectors $\{\tilde v\}_{v\in V}$.
 A different choice would multiply $f_\sD$ by an element
 of $\cA^\times$.
 \end{remark}

The short exact sequence of Proposition \ref{prop:divexact} induces a long exact sequence in cohomology. Let
\[
\left[\quad \right]\colon H^s(\Gamma,\Div(X_p))\rightarrow H^{s+1}(\Gamma, \cA^\times)
\]
be the induced boundary homomorphism.
It measures the obstruction of lifting a divisor-valued cohomology class to a cohomology class with values in meromorphic functions.
Thus, its image should be viewed as an analogue of the divisor class group.

The Gross--Kohnen--Zagier theorem of Borcherds (cf.~\cite{BorcherdsGKZ}) states that generating series of Heegner divisors in the divisor class group on an orthogonal Shimura variety is a modular form.
This leads naturally to the following question:
Assume that $\Gamma$ acts trivially on $\D_\latt$.
Let $\mathbf{e}_\beta$, $\beta\in \D_\latt$, denote the standard basis of $\Z[\D_\latt]$ corresponding to $\beta$.
What modularity properties does the formal Fourier series
\begin{align}\label{generatingseries}
Z_\latt(\tau):=\sum_{\beta\in \D_\latt}\sum_{m \in \Z_{(p)}^{>0}} \left[\sD_{m,\beta}\right] \cdot e^{2\pi i \cdot m \tau} \cdot \mathbf{e}_\beta
\end{align}
have?
Note that Proposition \ref{ordinarityprop} suggests that it is enough to consider only those $m\in \Q^{>0}$ that lie in $\Z_{(p)}$, that is, those that satisfy $\ord_p(m)\geq 0$.
We will partially answer this question in Section \ref{maintheorems}.

\subsubsection{The cycle class map}
For higher dimensional orthogonal Shimura varieties the kernel of the cycle class map from the divisor class group to the second singular cohomology group is frequently torsion.
This is used in the proof of \cite[Theorem 1.3]{YZZ}  to deduce the modularity of the generating series of Heegner cycles in the Chow group from Kudla and Millson's modularity theorem for topological cycles.
The analogue of the cycle class map in the current setup is the homomorphism
\[
\mathrm{cyc} \colon H^{s+1}(\Gamma,\cA^\times)\rightarrow H^{s+1}(\Gamma,\cA^\times/\Q_p^\times).
\]
\begin{remark}
If $n=3$ and $s=0$ one can identify $H^1(\Gamma,\cA^\times)$ with the Picard group of the Mumford curve $\Gamma\backslash X_p$, while $H^{1}(\Gamma,\cA^\times/\Q_p^\times)$ is canonically isomorphic to $\Z$, and the homomorphism $\mathrm{cyc}$ agrees with the usual degree map.
See
 \cite[Appendix A]{DVBorcherds} for a discussion of this case and its parallel with rigid meromorphic cocycles in signature $(2,1)$.
\end{remark}
Similarly, as in the setting of orthogonal Shimura varieties the kernel of $\mathrm{cyc}$ tends to be torsion frequently:
\begin{proposition}\label{cyclezero}
Let $n\geq 4$ and $r\geq s$.
If $n=4$, assume that $V_{\Q_p}$ is split and that $G$ is almost $\Q$-simple.
Then the kernel of $\mathrm{cyc}$ is finite in the following cases:
\begin{itemize}
\item $s$ is even
\item $s=1$
\item $s\equiv 1\bmod 4$ and $r$ is odd
\end{itemize}
\end{proposition}
Moreover, in the case $s=0$ the exponent of $\ker(\mathrm{cyc})$ divides $p-1$.
\begin{proof}
It is enough to show that the group $H^{s+1}(\Gamma,\Q_p^\times)$ is finite.
Since $\Gamma$ is a $p$-arithmetic group, it is of type $\textit{(VFL)}$ by \cite[no 2.4, Theorem 4]{serre-cohomologie}.
The remark on page 101 of \textit{loc.cit.~}implies that for every finitely generated abelian group $A$ the cohomology groups $H^{i}(\Gamma,A)$, $i\geq 0$, are finitely generated and that
\[
H^{i}(\Gamma,A)=H^{i}(\Gamma,\Z) \otimes_\Z A\quad
\]
for every flat $\Z$-module $A$ and every $i\geq 0$.
Thus, it is enough to prove that
$H^{s+1}(\Gamma,\C)=0$.
For $n\neq 4$ the group $G$ is almost $\Q$-simple.
Thus, in all cases, one may apply the main theorem of \cite{BFG} that describes $H^{i}(\Gamma,\C)$ in terms of automorphic forms.
To be more precise let $r_p$ denote the $\Q_p$-rank of $G_{\Q_p}$.
One may decompose
\begin{equation}\label{BFGdecomp}
H^{i}(\Gamma,\C)\cong H^{i}_{\mathrm{Stbg}}(\Gamma,\C) \oplus H^{i}_{\mathrm{const}}(\Gamma,\C).
\end{equation}
The first summand is generated by cuspidal automorphic representations, whose local factor at $p$ is a twist of the Steinberg representation of $G(\Q_p)$, and which contribute to the cohomology in degree $i-r_p$ of some locally symmetric space attached to $G$.
By a theorem of Vogan--Zuckerman (see \cite[Theorem 8.1]{VoganZuckerman}), there is no cuspidal cohomology in degrees below $s$.
The existence of a self-dual $\Z_p$-lattice in $V_{\Q_p}$ forces the Weil index of $V_{\Q_p}$ to be greater or equal to $2$ if $n\geq 5$.
Hence  $r_p\geq 2$ in all cases, which implies that
\[
H^{s+1}_{\mathrm{Stbg}}(\Gamma,\C)=0.
\]
The second summand in \eqref{BFGdecomp} can be described via $G(\R)$-invariant differential $i$-forms on $X_\infty$.
These are in one-to-one correspondence with classes in $H^{i}(X_\infty^{\vee},\C)$, where $X_\infty^{\vee}$ denotes the compact dual of $X_\infty$.
More precisely, $X_\infty^{\vee}$ is the Grassmannian of oriented $s$-planes in $\R^{r+s}$.
If $s$ is even, the rational cohomology of $X_\infty^{\vee}$ is concentrated in even degrees (see \cite{Borelcohomology}).
In case $s=1$, the oriented Grassmannian $X_\infty^{\vee}$ is just the $r$-sphere and, thus, its cohomology is concentrated in degree $0$ and $r$.
Finally, if $r$ and $s$ are both odd, then the cohomology of $X_\infty^{\vee}$ is generated by Pontryagin classes of the tautological bundles over $X_\infty^{\vee}$, whose degrees are by definition divisible by $4$, and a class in degree $r+s-1$ (cf.~\cite{Takeuchi}).
Thus, we deduce that $H^{s+1}_{\mathrm{const}}(\Gamma,\C)=0$ in all cases.
\end{proof}

\begin{remark}
In the definite case, that is, $s=0$, the vanishing of the first cohomology group of $\Gamma$ also follows from Margulis' normal subgroup theorem (see \cite[Chapter VIII, Theorem 2.6]{Margulis}).
\end{remark}

\begin{remark}
In case of signature $(3,1)$ the group $G_\R$ is almost $\R$-simple and, therefore, $G$ is almost $\Q$-simple.
\end{remark}

\subsection{Vector-valued modular forms}\label{kudlamillson}
We give a reminder on the Weil representation attached to finite quadratic modules and vector-valued modular forms and introduce a $U_{p^2}$-operator on vector-valued modular forms.
We state the modularity of theta series of definite quadratic forms as well as Funke and Millson's theorem on the modularity of intersection numbers of special cycles with modular symbols.

\subsubsection{The Weil representation attached to a finite quadratic module}\label{finitequadraticmodules}
Let $\cH\subseteq \C$ denote the complex upper half plane and $\GL_2(\R)^+\subseteq \GL_2(\R)$ the subgroup of matrices with positive determinant.
The group $\GL_2(\R)^+$ acts on $\cH$ via Möbius transformation, that is,
\[
\begin{pmatrix}a & b \\ c & d\end{pmatrix}.\tau = \frac{a\tau+b}{c\tau+d}.
\]
For $g=\smallmat{a}{b}{c}{d}\in\GL_2(\R)^+$ we let
\[
j(g,\cdot)\colon \cH \longrightarrow \C,\quad \tau\longmapsto c\tau+d,
\]
be the usual automorphy factor.
The metaplectic two-fold cover $\widetilde{\GL}_2(\R)^+$ of $\GL_2(\R)^+$ is the group of pairs $(g,\phi(\tau))$ with $g\in\GL_2(\R)^+$ and $\phi\colon \cH\rightarrow \C$ a holomorphic function satisfying $\phi(\tau)^2=j(g,\tau)$.
The multiplication is given by the law
\[(g,\phi(\tau))(g',\phi'(\tau))=(gg', \phi(g'.\tau)\phi'(\tau)).\]
Let $\Mp_2(\Z)$ be the preimage of $\SL_2(\Z)$ in $\widetilde{\GL}_2(\R)^+$.
It is generated by the two elements
\[
T=\left(\mat{1}{1}{0}{1},1\right)\quad \mbox{and}\quad S=\left(\mat{0}{-1}{1}{0}, \sqrt{\tau}\right).
\]

A \emph{finite quadratic module} $(D,q_D)$ is  a finite group $D$ together with a function
\[
q_D\colon D \longrightarrow \Q/\Z
\]
such that
\begin{itemize}
\item $q_D(n x)= n^2 q_D(x)$ for all $x\in D$, $n\in \Z$ and
\item the function $B_D(x,y)=q_D(x+y)-q_D(x)-q_D(y)$ is a perfect symmetric pairing on $D$.
\end{itemize}
The \emph{level} $N_D$ of $D$ is the smallest positive integer such that $N_D\cdot q_D(x)=0$ for all $x\in D$.
It is clearly bounded by the exponent of $D$.
Its \emph{signature} $\sign(D)\in \Z/8\Z$ is defined by the formula
\[
e^{2\pi i \cdot \sign(D)/8}=\frac{1}{\sqrt{|D|}} \sum_{y\in D} e^{2\pi i \cdot q_D(x)}. 
\]

Let us recall the \emph{Weil representation}
\[\rho_D\colon \Mp_2(\Z)\rightarrow \End_\C(\C[D])\]
attached to $D$.
We denote by $\mathbf{e}_x$, $x\in D$, the standard basis of $\C[D]$.
The Weil representation is determined by the action of the two generators:
\begin{align*}
\rho_D(T)(\mathbf{e}_x) &= e^{2\pi i \cdot q_D(x)}\cdot \mathbf{e}_x,\\
\rho_D(S)(\mathbf{e}_x) &= \frac{\sigma_w(D)}{\sqrt{|D|}}\cdot \sum_{y\in D} e(-B_D(x,y))\cdot \mathbf{e}_y,
\end{align*}
where $\sigma_w(D)$ is the Witt invariant of $D$, that is,
\[\sigma_w(D)=\frac{1}{\sqrt{|D|}}\cdot \sum_{y\in D} e^{-2\pi i \cdot q_D(x)}=e^{-2\pi i \cdot \sign(D)/8}.\]

\subsubsection{Explicit actions}
In case $\sign(D)\equiv 0 \bmod 2$, the action of $\Mp_2(\Z)$ on $\rho_D$ factors through $\SL_2(\Z)$.
Moreover, it is trivial on the principal congruence subgroup $\Gamma(N_D)$.
If $\sign(D)\equiv 1 \bmod 2$, then $4$ divides $N_D$.
In particular, $\sign(D)$ is even if $2\nmid |D|$.
The explicit formula for the action of $T$ 
together with \cite[Theorem 5.4]{BorcherdsReflection}, 
 implies the following well-known description of the action of $\Gamma_0(N_D)$ on $\rho_D$.
\begin{lemma}\label{explicitWeil}
Suppose that $2\nmid |D|$ and let $x$ be an element of $D$.
Then  
\[
\rho_{D}(\gamma)(\mathbf{e}_x) =e^{2\pi i \cdot bd\cdot q_D(x)}\legendre{d}{|D|}\cdot \mathbf{e}_{dx}\quad\forall\gamma=\begin{pmatrix}a& b\\ c & d\end{pmatrix} \in\Gamma_0(N_D).
\]
\end{lemma}

\subsubsection{Decomposition of Weil representation}
The finite quadratic module $D$ decomposes canonically as the direct sum
\[
D=\bigoplus_{\ell} D_\ell,
\]
over all primes $\ell$, where $D_\ell:=D[\ell^\infty]$ denotes the submodule of $\ell^\infty$-torsion elements.
Each $D_\ell$ is also a finite quadratic module.
The Weil representation of $D$ decomposes accordingly:
\begin{proposition}[\cite{Zemel}, Proposition 3.2]
Let $\{\ell_{1},\ldots,\ell_k\}$ be a finite set of distinct primes containing all prime divisors of $|D|$.
The canonical map
\[
\bigotimes_{i=1}^{k} \C[D_{\ell_i}]\longrightarrow \C[D],\quad (e_{x_{i}})_{1\leq i \leq k}\longmapsto e_{x_1+\cdots+x_k}
\] 
is an isomorphism of $\Mp_2(\Z)$-representations.
\end{proposition}

In particular, one may decompose $D=D_p\oplus D^{(p)}$, where $D^{(p)}=\oplus_{\ell\neq p} D_\ell$ and get the following decomposition of finite Weil representations:
\begin{equation}\label{Weildecomp}
\rho_D =\rho_{D_p}\otimes_\C \rho_{D^{(p)}}.
\end{equation}

\subsubsection{Finite quadratic modules from lattices}\label{examplefromlattice}
Let $L_\circ\subseteq V$ be an even $\Z$-lattice, that is, $q(w)\in \Z$ for all $w\in L_\circ$. (Typically it will be assumed that
$L_\circ \otimes \Z[1/p]= \latt$ but this is not essential.)
The function
\[\tilde{q}\colon \D_{L_\circ}\longrightarrow \Q/\Z,\quad \beta \longmapsto q(\beta) \bmod \Z\]
makes the discriminant module $\D_{L_\circ}$ into a finite quadratic module.
Moreover, the equality
\[\sign(\D_{L_\circ})=r-s\bmod 8\]
holds by Milgram's formula.
We are mostly interested in lattices of the form $L_\circ=\latt\cap p^t\Lambda$ with $\Lambda\subseteq V_{\Q_p}$ a self-dual $\Z_p$lattice and $t\geq 0$.
By identifying $\Q/\Z[1/p]$ with the prime to $p$-part of $\Q/\Z$
the quadratic form $q$ induces a map
\[
\D_\latt \rightarrow \Q/\Z[1/p] \hookrightarrow \Q/\Z,
\]
which makes $\D_\latt$ into a finite quadratic module.
Moreover, there is a canonical isomorphism
\[
\D_{\latt\cap p^t\Lambda}^{(p)}\cong \D_{\latt}
\]
of finite quadratic modules.
Similarly, the canonical homomorphism $\Q/\Z\rightarrow \Q_p/\Z_p$ induces an isomorphism between the $p^\infty$-torsion submodule of $\Q/\Z$ and $\Q_p/\Z_p$.
Thus, the quadratic form $q$ makes $\D_{p^t\Lambda}=p^{-t}\Lambda/p^t\Lambda$ into a finite quadratic module and
there is a canonical isomorphism
\[
\D_{\latt\cap p^t\Lambda, p}\cong \D_{p^t\Lambda}
\]
of finite quadratic modules.
Thus, \eqref{Weildecomp} gives the decomposition
\[
\rho_{\D_{\latt\cap p^t\Lambda}}=\rho_{\D_{\latt}}\otimes_\C \rho_{\D_{p^t\Lambda}}.
\]
Moreover, self-duality of $\Lambda$ implies that 
\begin{align*}
\ord_p(N_{\D_{L\cap p^t\Lambda}}) &=\ord_p(N_{\D_{p^t\Lambda}})= 2t\\
\intertext{and}
\ord_p(|\D_{\latt\cap p^t\Lambda}|)&=\ord_p(|\D_{p^t\Lambda}|)=2nt.
\end{align*}

\subsubsection{Extending the Weil representation}
Fix a finite quadratic module $D$.
Following Bruinier and Stein (cf.\cite{Bruinierstein}) we extend the action of $\Mp_2(\Z)$ on $\rho_D$ to a larger group in order to define Hecke operators on $\rho_D$-valued modular forms.
To that end let $\Z_{(N_D)}$ be the localization of $\Z$ away from $N_D$, that is, we invert all primes not dividing $N_D$.
We define
\begin{align*}
\mathcal{H}(N_D)&:=\{g\in \GL_2(\Z_{(N_D)}) \vert\ \det(g)>0,\ \det(g)\ \mbox{is a square} \bmod{N_d} \}.
\intertext{and}
\mathcal{Q}(N_D)&:=\{(g,r)\in \mathcal{H}(N_D))\times (\Z/N_D\Z)^\times\ \vert\ \det(g)\equiv r^2 \bmod{N_d} \}.
\end{align*}
In \cite[\S 4,5]{Bruinierstein},  Bruinier and Stein construct a projective, unitary representation of $\mathcal{Q}(N_D)$ on $\C[D]$.
In case that $\sign(D)$ is even, this is a honest representation.
In general, there exists a central extension $\mathcal{Q}_1(N_D)$ of $\mathcal{Q}(N_D)$ by $\{\pm 1\}$ and a unitary representation of $\mathcal{Q}_1(N_D)$ on $\C[D]$ that will also denote by $\rho_D$.
The following explicit description of the action will be an important ingredient in the study of the $U_{p^2}$-operator.
\begin{lemma}[\cite{Bruinierstein}, Lemma 3.6]
For every integer $m$ with $(m,N_D)=1$ the formula
\[
\rho_D\left(\begin{pmatrix}1 & 0 \\ 0 & m^2\end{pmatrix}, m, \pm 1\right) (\mathbf{e}_x)=\pm\mathbf{e}_{mx}
\]
holds.
\end{lemma}
Let $\widetilde{\mathcal{H}}(N_D)$ be the preimage of $\mathcal{H}(N_D)$ in $\widetilde{\GL}_2(\R)^+$.
The group $\mathcal{Q}_2(N_D)$ is defined to be the group of all tuples $(g,\phi,r,t)$ such that $(g,\phi)\in \widetilde{\mathcal{H}}(N_D)$ and $(g,r,t)\in\mathcal{Q}_1(N_D)$ with multiplication induced from the canonical embedding
\[
\mathcal{Q}_2(N_D)\lhook\joinrel\longrightarrow \widetilde{\mathcal{H}}(N_D) \times\mathcal{Q}_1(N_D).
\]
It acts on $\C[D]$ via its projection onto $\mathcal{Q}_1(N_D)$.

Let $\sqrt{\cdot}$ denote the principal branch of the holomorphic square root.
The injection
\[
L_D\colon \Mp_2(\Z)\longrightarrow \mathcal{Q}_2(N_D),\quad (g,\pm\sqrt{j(g,\cdot)})\longmapsto (g,\pm\sqrt{j(g,\cdot)},1,\pm 1),
\]
defines a group homomorphism that satisfies
\[
\rho_D(L_D(\gamma))=\rho_D(\gamma)\quad \forall \gamma\in \Mp_2(\Z).
\]
We often view $\Mp_2(\Z)$ as a subgroup of $\mathcal{Q}_2(N_D)$ via the embedding $L_D$.

\subsubsection{Slash operator and Fourier expansions}
Let $D$ be a finite quadratic module and $k\in\frac{1}{2}\Z$ a half-integer.
The metaplectic group $\widetilde{\GL}_2(\R)^+$ acts on functions $f\colon \cH\rightarrow \C$ from the right via
\[
f\vert_{k,(g,\phi)}(\tau):=\det(g)^{k/2}\phi(\tau)^{-2k}f(g.\tau)\quad \forall (g,\phi)\in \widetilde{\GL}_2(\R)^+.
\]
This induces an action of the group $\mathcal{Q}_2(N_D)$ via the projection $\mathcal{Q}_2(N_D) \rightarrow \widetilde{\mathcal{H}}(N_D)$.
Any function $f\colon \cH\rightarrow \C[D]$ can be uniquely written as a sum $f=\sum_{x\in D} f_x\cdot \mathbf{e}_x$ for some functions $f_x\colon \cH \rightarrow \C$.
Given such a function and $\gamma\in\mathcal{Q}_2(N_D)$, we define
\[
f\vert_{k,\gamma}=\sum_{x\in D} f_x\vert_{k,\gamma}\cdot \rho_D(\gamma)^{-1}(\mathbf{e}_x).
\]
Let $w\geq 1$ be an integer and suppose that $f$ satisfies $f\vert_{k,T^{w}}=f$ or, in other words, $f_x(\tau+w) \cdot e^{-2 \pi i\cdot w q(x)}=f_x(\tau)$ for all $x\in D$.
Then the function $\tau\mapsto f_x(\tau) e^{-2 \pi i \cdot w q(x)\tau}$ is periodic with period $w$ and, thus,
$f$ has a Fourier expansion of the form
\[
f(\tau)=\sum_{x\in D} \sum_{m\in \Q} a_f(m,x) \cdot e^{2\pi i \cdot m\tau}\cdot \mathbf{e}_x.
\]
Moreover, in case $w=1$ the Fourier coefficients $a_f(m,x)$ vanish unless $m\in q(x)+\Z$.

\subsubsection{Vector-valued modular forms}
Let $\cgroup\subseteq \mathcal{Q}_2(N_D)$ be a subgroup that is commensurable with $L(\Mp_2(\Z))$.
A \emph{vector-valued modular form of weight $k$, type $\rho_D$ and level $\cgroup$} is a function
\[f\colon \cH\longrightarrow \C[D]\]
such that
\begin{enumerate}[(i)]
\item $f\vert_{k,\gamma}= f$ for all $\gamma\in \cgroup$, $\tau\in\cH$,
\item $f$ is holomorphic and
\item $f$ is holomorphic at the cusps, that is, for every $\gamma\in \Mp_2(\Z)$ the Fourier coefficients of $f\vert_{k,\gamma}$ satisfy:
\[
a_{f\vert_{k,\gamma}}(m,x)=0\quad \forall m< 0.
\]
\end{enumerate}
A modular form $f$ is uniquely determined by its Fourier expansion.
We denote by $M_{k,D}(\cgroup)$ the vector space of $\C[D]$-valued modular form of weight $k$ and level $\cgroup$.
Let $N$ be any positive integer.
We denote by $\cgroup_0(N)\subseteq \Mp_2(\Z)$ the preimage of the congruence subgroup $\Gamma_0(N)$ under the quotient map $\Mp_2(\Z)\rightarrow \SL_2(\Z)$.
Note that if $\sign(D)\neq 2k \bmod 2$, then $M_{k,D}(\cgroup_0(N))=\{0\}$.

More generally, let $A$ be an abelian group.
A formal power series
\[
f=\sum_{x\in D} \sum_{m\in \Q}  a(m,x) \cdot e^{2\pi i \cdot m\tau}\cdot \mathbf{e}_x,\quad a(m,x)\in A,
\]
is a \emph{modular form of weight $k$, type $\rho_D$ and level $\cgroup$} if for every homomorphism $\psi\colon A\to\C$ the formal power series 
\[
\psi(f)(\tau)=\sum_{x\in D} \sum_{m\in \Q}  \psi(a(m,x))\cdot e^{2\pi i \cdot m\tau}\cdot \mathbf{e}_x
\] is the Fourier expansion of a modular form of weight $k$, type $\rho_D$ and level $\cgroup$.

\subsubsection{Evaluation maps}
Let $f\colon \cH\rightarrow \C[D]$ be a function and $\mu=\sum_{x_p\in D_p} a_{x_p} \cdot \mathbf{e}_{x_p}$ an  element of $\C[D_p]$.
The function $\Ev_\mu(f)\colon \cH\rightarrow \C[D^{(p)}]$ is defined via
\[
\Ev_\mu(f)_{x^{(p)}}:=\sum_{x\in D_p} a_x\cdot  f_{(x^{(p)},x_p)}\quad  \forall x^{(p)}\in D^{(p)}.
\]
The following lemma is an immediate consequence of \eqref{Weildecomp}.
\begin{lemma}\label{trivializing}
Let $\cgroup\subseteq \Mp_2(\Z)$ be a finite index subgroup, $f\in M_{k,D}(\cgroup)$ and $\mu\in \C[D_p]$.
Then $\Ev_\mu(f)\in M_{k,D^{(p)}}(\cgroup_\mu)$, where $\cgroup_\mu$ denotes the stabilizer of $\mu$ in $\cgroup$.
\end{lemma}

Let $\Lambda\subseteq V_{\Q_p}$ be a self-dual $\Z_p$-lattice, $t\geq 1$ an integer, and $\ell\in Q_\Lambda(\Z/p^t\Z)$ an isotropic line, that is, $\ell$ is a free $\Z/p^t\Z$-module on which $q(v)\equiv 0 \bmod p^t$.
Consider the set
\begin{align*}
P(\Lambda,\ell)&:=\left\{x_p\in \D_{p^t\Lambda}\ \vert\ x_p\ \mbox{generates}\ \ell\right\}\\
\intertext{and put}
\mu(\Lambda,\ell)&:=\sum_{x_p\in P(\Lambda,\ell)} \mathbf{e}_{x_p} \in \C[\D_{p^t\Lambda}]
\end{align*}
\begin{corollary}\label{vertextrivializing}
Let $\Lambda\subseteq V_{\Q_p}$ be a self-dual and $\Z_p$-lattice and $\ell\in Q_\Lambda(\Z/p^t\Z)$ an isotropic line.
Then:
\[
\Ev_{\mu(\Lambda,\ell)}(f) \in M_{k,\D_{\latt}}(\cgroup_0(p^{2t}))\quad \forall f \in M_{k,\D_{\latt\cap p^{t}\Lambda}}.
\]
\end{corollary}
\begin{proof}
The action of $\Mp_2(\Z)$ on $\D_{p^t\Lambda}$ factors through $\SL_2(\Z)$.
Let ${x_p}\in \D_\Lambda$ be any class that is represented by a vector in $\Lambda$.
Then, Lemma \ref{explicitWeil} together with the discussion in Section \ref{examplefromlattice} implies that
\[
\rho_{\D_\Lambda}(\gamma) (\mathbf{e}_{x_p}) = \mathbf{e}_{d\cdot x_p}\quad  \forall \gamma=\begin{pmatrix} a& b \\ c & d \end{pmatrix}\in \Gamma_0(p^{2t}).
\]
Thus, the operator $\rho_{\D_\Lambda}(\gamma)$ permutes the elements in $\{\mathbf{e}_{x_p}\ \vert\ x_p\in P(\Lambda,\ell)\}$, which implies that it stabilizes the element $\mu(\Lambda,\ell)$.
The assertion follows from Lemma \ref{trivializing}.
\end{proof}

\subsubsection{Hecke operators}
For any subgroup $\cgroup\subseteq \mathcal{Q}_2(N_D)$ as above and any $\alpha=(g,\phi,r,t) \in \mathcal{Q}_2(N_D)$
the double coset $\cgroup \alpha\cgroup$ decomposes into a finite union of left cosets:
\[
\cgroup \alpha\cgroup =\bigcup_i \cgroup \xi_i
\]
The Hecke operator
\[
T_{\cgroup\alpha\cgroup}\colon M_{k,D}(\cgroup)\longrightarrow M_{k,D}(\cgroup),\quad f\longmapsto \det(g)^{k/2-1}\sum_i f\vert_{k,\alpha_i}
\]
does not depend on the choice of coset representatives $\xi_i$.
We are only interested in the case $p\nmid N_D$ and the special element
\[
\alpha_p:=\left(\begin{pmatrix}1 & 0 \\ 0 & p^2\end{pmatrix}, p , p , 1\right)\in \mathcal{Q}_2(N_D).
\]

\noindent Let $A$ be an abelian group.
Given a formal Fourier series
\begin{align*}
f(\tau)=&\sum_{x\in D} \sum_{m\in \Q_{\geq 0}}  a_f(m,x)\cdot e^{2\pi i \cdot m\tau}\cdot \mathbf{e}_x,\quad a_f(m,x)\in A,
\intertext{we define}
U_{p^{2}}(f)(\tau):=&\sum_{x\in D} \sum_{m\in \Q_{\geq 0}}  a_f(p^{2}m,p x)\cdot e^{2\pi i \cdot m\tau}\cdot \mathbf{e}_x.
\end{align*}
Proposition \ref{ordinarityprop} implies that the formal Fourier series defined in \eqref{generatingseries} satisfies
\begin{equation}\label{ordinarity}
U_{p^{2}}(Z_\latt)(\tau)=Z_\latt(\tau).
\end{equation}

This $U_{p^2}$-operator satisfies the following analogue of the classical level lowering property of the $U_p$-operator on scalar-valued modular forms of integral weight (cf.~\cite[Lemmas 6,7]{AtkinLehner}, and \cite[Lemma 1]{Li}):
\begin{proposition}\label{levellowering}
Let $D$ be a finite quadratic module with $p\nmid N_D$ and $t\geq 1$ an integer.
Then:
\[
U_{p^2}(f)=T_{\cgroup_0(p^t)\alpha_p\cgroup_0(p^t)}(f)\quad \forall f\in M_{k,D}(\cgroup_0(p^t))
\]
In particular, one has $U_{p^2}(f)\in M_{k,D}(\cgroup_0(p^t))$ for all $f\in M_{k,D}(\cgroup_0(p^t))$.
Moreover, if $t\geq 3$, then
\[
U_{p^2}(f)\in M_{k,D}(\cgroup_0(p^{t-2})).
\]
\end{proposition}
\begin{proof}
Consider the element
\[
\beta_p:=\left(\begin{pmatrix}1 & 0 \\ 0 & p^2\end{pmatrix}, p \right) \in \widetilde{\mathcal{H}}(N_D).
\]
Lemma 4.7 and Lemma 4.8 of \cite{Bruinierstein} imply that if
\[
\cgroup_0(p^t)\beta_p\cgroup_0(p^t)= \bigcup_i \cgroup_0(p^t)\beta_p \gamma_i \subseteq \widetilde{\mathcal{H}}(N_D)
\]
with $\gamma_i \in \cgroup_0(p^t)$ is a decomposition into disjoint left cosets, then
\[
\cgroup_0(p^t)\alpha_p\cgroup_0(p^t) =\bigcup_i \cgroup_0(p^t) \alpha_p \gamma_i \subseteq \mathcal{Q}_2(N_D)
\]
is also a disjoint decomposition into left cosets.
In particular, we may choose
\[
\gamma_i=\left(\begin{pmatrix}1 & 0 \\ i & 1\end{pmatrix}, p \right),\quad 0\leq i < p^2.
\]
The desired equality of operators then directly follows from the last computation in the proof of \cite[Theorem 4.10]{Bruinierstein}.

Now let $t\geq 3$.
Given integers $N,M\geq 1$ we let $\cgroup_0(N,M)$ be the preimage in $\Mp_2(\Z)$ of the congruence subgroup
\[
\Gamma_0(N,M)=\left\{\begin{pmatrix}a & b \\ c & d\end{pmatrix}\in \SL_2(\Z)\ \middle\vert\ c\equiv 0 \bmod N,\ b\equiv 0 \bmod M\right\}
\]
Lemma 4.8 of \textit{loc.cit} implies that
\[
f\vert_{k,\alpha_p}\in M_{k,D}(\Gamma_0(p^{t-2},p^2)) \quad \forall f\in M_{k,D}(\Gamma_0(p^{t})).
\]
Averaging a modular form in $M_{k,D}(\Gamma_0(p^{t-2},p^2))$ over the elements $\gamma_i$ gives a form of level $\Gamma_0(p^{t-2})$.
\end{proof}

\begin{remark}
Proposition \ref{levellowering} combined with equation \eqref{ordinarity} suggests that the formal power series $Z_\latt$ is a modular form of level $\cgroup_0(p)$.
\end{remark}

\begin{corollary}\label{maincombinatoriallemma}
Let $\Lambda\subseteq V_{\Q_p}$ be a self-dual $\Z_p$-lattice, $\ell\in Q(\Z/p^t\Z)$ an isotropic line, and $f\in M_{k,\D_{\latt\cap p^t\Lambda}}$.
There exists a modular form $g\in M_{k,\D_\latt}(\cgroup_0(p))$ such that
\[
a_g(m,\beta)=\sum_{x_p\in P(\Lambda,\ell)} a_f(p^{2t} m, (p^t \beta, x_p)) \quad \forall m> 0,\ \beta\in \D_{\latt}.
\]
\end{corollary}
\begin{proof}
Corollary \ref{vertextrivializing} together with Proposition \ref{levellowering} implies that the modular form $g=(U_{p^2}^{t}\circ\Ev_{\mu(\Lambda,\ell)})(f)$ has level $\cgroup_0(p)$.
By construction it has the desired Fourier coefficients.
\end{proof}

\subsubsection{Modularity of theta series}
Let us suppose for the moment that $s=0$, that is, the quadratic form $q$ is positive definite.
Let $L_\circ\subseteq V$ be any $\Z$-lattice such that $q(L_\circ)\subseteq \Z$.
Since $q$ is positive definite the set $\cO_{L_\circ}(m,x):=\{v\in x\ \vert\ q(v)=m\}$ is finite for any $x\in \D_{L_\circ}$ and any $m\in \Q_{\geq 0}$.
We put $r_{L_\circ}(m,x):=|\cO_{L_\circ}(m,x)|$.
It is well known that the theta series
\[
\vartheta_{L_\circ}(\tau):=\sum_{x\in \D_{L_\circ}} \sum_{m\in \Q_{\geq 0}} r_{L_\circ}(m,x) \cdot e^{2\pi i \cdot m \tau} \cdot \mathbf{e}_x
\]
is an element of $M_{n/2,\D_{L_\circ}}$.
The following hyperbolic analogue when $s=1$
 is due to Funke and Millson (see \cite[Theorem 1.7]{FunkeMillson})
\begin{theorem}[Funke--Millson]
\label{FMhyperbolic}
Assume that $s=1$ and that $Q(\Q)\neq \emptyset$.
For every $\Z$-lattice $L_\circ\subseteq V$ such 
that $q(L_\circ)\subseteq \Z$ and every pair of 
rational isotropic lines $\ell_{-},\ell_{+}\in Q(\Q)$ 
there exists a modular form
\[
\vartheta_{L_\circ,(\ell_{-},\ell_{+})}\in M_{n/2,\D_{L_\circ}}
\]
such that
\[
a_{\vartheta_{L_\circ,(\ell_{-0},\ell_{+})}}(m,\beta)=\sum_{v\in \cO_{L_0}(m,\beta)} [\ell_{-},\ell_{+}]\cap \Delta_{v,\infty}
\]
holds for all $(m,\beta)$ with $m$ compact with respect to $(V,q)$ (see Definition \ref{def-compactm}).
\end{theorem}

\subsection{Modularity theorems}\label{maintheorems}
For the remainder of this chapter we assume that $\Gamma$ acts trivially on the discriminant module $\D_\latt$.
We state the main theorems on the existence of lifts of linear combinations of Kudla--Millson divisors to rigid meromorphic cocycles.
To this end, for any ring $R$ consider the $R$-module
\[
\mathcal{C}_{\D_\latt}\hspace{-0.2em}(R):=\bigoplus_{m\in \Z_{(p)}^{>0}} \bigoplus_{\beta\in \D_\latt} R.
\]
We attach to an element $\underline{c}=(c_{m,\beta})\in \mathcal{C}_{\D_\latt}\hspace{-0.2em}(\Z)$ the Kudla--Millson divisor
\[
\sD_{\underline{c}}:=\sum_{m\in \Z_{(p)}^{>0}} \sum _{\beta \in \D_\latt} c_{m,\beta}\cdot \sD_{m,\beta}.
\]
Moreover, we define $\mathcal{F}_{\D_\latt}\hspace{-0.2em}(R)$ to be the module of formal $q$-series of the form
\[
f(q)=\sum_{\beta\in\D_\latt}\sum_{m\in \Z_{(p)}^{>0}} a_f(m,\beta)\cdot e^{2\pi i \cdot m\tau}\cdot \mathbf{e}_\beta,\quad a_f(m,\beta)\in R.
\]
For every ring $R$ the canonical pairing
\[
\mathcal{F}_{\D_\latt}\hspace{-0.2em}(R) \times \mathcal{C}_{\D_\latt}\hspace{-0.2em}(R) \longrightarrow R,
\quad (f(q),\underline{c}) \longmapsto \underline{c}(f(q)):=\sum_{\beta\in\D_\latt}\sum_{ m\in \Z_{p}^{>0}} c_{m,\beta}\cdot a_f(m,\beta)
\]
is non-degnerate.
The homomorphism
\[
M_{k,\D_\latt}(\cgroup_0(N))\longrightarrow \mathcal{F}_{\D_\latt}\hspace{-0.2em}(\C)
\]
that maps a modular form to the positive part of its Fourier expansion is injective.
We consider $M_{k,\D_\latt}(\cgroup_0(N))$ as a submodule of $\mathcal{F}_{\D_\latt}\hspace{-0.2em}(\C)$ via this embedding.
Since $M_{k,\D_\latt}(\cgroup_0(N))$ is finite-dimensional it is equal to its double orthogonal complement with respect to the pairing above.
By \cite{McGraw} $M_{k,\D_\latt}$ has a basis of modular forms whose Fourier expansion have rational coefficients.
The same arguments show that $M_{k,\D_\latt}(\cgroup_0(N))$ has such a basis for every integer $N\geq 1$, which implies that
\begin{align}\label{doubleperp}
M_{k,\D_\latt}(\cgroup_0(N))=(M_{k,\D_\latt}(\cgroup_0(N))^\perp\cap \mathcal{C}_{\D_\latt}(\Z))^\perp.
\end{align}

\subsubsection{Main theorems in the definite case}
\begin{theorem}\label{maindefinite1}
Assume that $s=0$. 
Let $\underline{c}\in\mathcal{C}_{\D_\latt}\hspace{-0.2em}(\Z)$ such that
\[
\underline{c}(f)=0
\]
for every $f\in M_{n/2,\D_\latt}(\cgroup_0(p))$.
Then there exists $\hat{J}\in H^{0}(\Gamma,\rM^\times/\Z_p^\times)$ such that
\[
\divmap_\ast(\hat{J})=\sD_{\underline{c}}.
\]
\end{theorem}
\begin{proof}
The proof of this theorem will be given in Section \ref{Borcherdsdefinite}.
\end{proof}

\begin{theorem}
\label{maindefinite2}
Let $n\geq 4$ and $s=0$.
If $n=4$, assume that $V_{\Q_p}$ is split and that $G$ is almost $\Q$-simple.
Let $\underline{c}\in\mathcal{C}_{\D_\latt}\hspace{-0.2em}(\Z)$ such that
\[
\underline{c}(f)=0
\]
for every $f\in M_{n/2,\D_\latt}(\cgroup_0(p))$.
There exists a rigid meromorphic cocycle $J\in \mathcal{RMC}(\Gamma)$ such that
\[
\divmap_\ast(J)=(p-1)\cdot \sD_{\underline{c}}.
\]
\end{theorem}
\begin{proof}
Theorem \ref{maindefinite1} implies that $\mathrm{cyc}([\sD_{\underline{c}}])=0$.
As the kernel of the cycle class map is finite of exponent dividing $p-1$ by Proposition \ref{cyclezero}, it follows that $[(p-1)\cdot \sD_{\underline{c}}]=0$.
Thus, the claim follows.
\end{proof}

\begin{theorem}
\label{modularitydefinite}
Let $n\geq 4$ and $s=0$.
If $n=4$, assume that $V_{\Q_p}$ is split and that $G$ is almost $\Q$-simple.
There exist $a_{0,\beta}\in H^{1}(\Gamma,\cA^\times)\otimes \Q$, $\beta\in \D_{\latt}$, such that the formal power series
\[
\sum_{\beta\in \D_\latt} a_{0,\beta} \cdot \mathbf{e}_\beta + Z_\latt(\tau)=\sum_{\beta\in \D_\latt} \big(a_{0,\beta}  + \sum_{m\in \Z_{(p)}^{>0}} \left[\sD_{m,\beta}\right] \cdot e^{2\pi i \cdot m \tau} \big) \cdot \mathbf{e}_\beta
\]
is a modular form of weight $n/2$, type $\rho_{\D_\latt}$ and level $\cgroup_0(p)$.
\end{theorem}
\begin{proof}
Let $\psi\colon H^{1}(\Gamma,\cA^\times) \to \C$ a homomorphism of abelian groups.
Theorem \ref{maindefinite2} implies that
\[\psi(Z_\latt)(\tau)\in (M_{k,\D_\latt}(\cgroup_0(N))^\perp\cap \mathcal{C}_{\D_\latt}(\Z))^\perp.
\]
Thus, the claim follows from \eqref{doubleperp}.
\end{proof}

\begin{remark}
Crucial in this argument  is the finiteness of the kernel of the cycle class map.
This fails in dimension three.
Indeed, the kernel of the cycle class map is the Jacobian of the Mumford curve $\Gamma\backslash X_p$.
In the forthcoming work \cite{BDGR} Theorem \ref{modularitydefinite} is proven in signature $(3,0)$ via a study of $p$-adic deformations of theta series attached to the definite quadratic space $V$.
\end{remark}
\begin{remark}
If the signature is $(4,0)$ and $V_{\Q_p}$ is split, one can identify the quotient $\Gamma\backslash X_p$ with a quaternionic Shimura surface. Under this identification Kudla--Millson divisors and Heegner divisors match up.
Thus, Theorem \ref{modularitydefinite} gives a $p$-adic analytic proof of the Gross--Kohnen--Zagier theorem for these Shimura surfaces. (For more details see the upcoming article \cite{GeBa}).
\end{remark}

\subsubsection{Main theorems when $s=1$}
There are analogous results in some hyperbolic cases where $s=1$.
Assume  that $Q(\Q)\neq\emptyset$.
We call an element $\underline{c}=(c_{m,\beta})\in \mathcal{C}_{\D_\latt}\hspace{-0.2em}(R)$ {\em compact} with respect to $(V,q)$ if $c_{m,\beta}=0$ for all $m$ that are not compact with respect to $(V,q)$.
Let $\underline{c}=(c_{m,\beta})\in \mathcal{C}_{\D_\latt}\hspace{-0.2em}(\Z)$ be compact with respect to $(V,q)$.
By Corollary \ref{modsymbcor} we may lift the Kudla--Millson divisor $\sD_{\underline{c}}$ to a $\Gamma$-invariant $\Div(X_p)$-valued modular symbol $\widetilde{\sD}_{\underline{c}}$.

\begin{theorem}
\label{mainhyperbolic1}
Let $s=1$, $Q(\Q)\neq \emptyset$ and $\underline{c}\in \mathcal{C}_{\D_\latt}\hspace{-0.2em}(\Z)$ compact with respect to $(V,q)$.
Under the condition that
\[
\underline{c}(f)=0
\]
for every $f\in M_{n/2,\D_\latt}(\cgroup_0(p))$
there exists $\hat{J}\in \mathrm{MS}(\rM^\times/\Z_p^\times)^\Gamma$ such that
\[
\divmap_\ast(\hat{J})=\widetilde{\sD}_{\underline{c}}.
\]
\end{theorem}
\begin{proof}
The proof of this theorem is the content of Section \ref{Borcherdshyperbolic}.
\end{proof}

\begin{theorem}
\label{mainhyperbolic2}
Let $n\geq 4$, $s=1$, $Q(\Q)\neq \emptyset$ and and $\underline{c}\in \mathcal{C}_{\D_\latt}\hspace{-0.2em}(Z)$ compact with respect to $(V,q)$.
If $n=4$, assume that $V_{\Q_p}$ is split.
Suppose that
\[
\underline{c}(f)=0
\]
for every $f\in M_{n/2,\D_\latt}(\cgroup_0(p))$.
There exists a non-zero integer $c\in \Z$ and a rigid meromorphic cocycle $J\in \mathcal{RMC}(\Gamma)$ such that
\[
\divmap_\ast(J)=c\cdot \sD_{\underline{c}}.
\]
\end{theorem}
\begin{proof}
Since $\delta(\hat{J})\in H^{1}(\Gamma,\rM^\times/\Z_p^\times)$ is a lift of $\sD_{\underline{c}}$ it follows that
$\mathrm{cyc}([\sD_{\underline{c}}])=0$.
As in the proof of Theorem \ref{maindefinite2} we conclude by using Proposition \ref{cyclezero}.
\end{proof}

\begin{remark}
As mentioned before, the existence of positive rational numbers that are compact with respect to $(V,q)$ forces the dimension of $V$ to be less or equal than $5$.
In particular, Theorem \ref{mainhyperbolic2} is only applicable in signature $(3,1)$ and $(4,1)$.
Theorem \ref{mainhyperbolic2} for the split three-dimensional quadratic space is proven in \cite{DVBorcherds}. 
\end{remark}

\subsection{Proof of modularity in the definite case}\label{Borcherdsdefinite}
The aim of this section is to prove Theorem \ref{maindefinite1} and Theorem \ref{mainhyperbolic1} by constructing explicit infinite $p$-adic products that are reminiscent of Borcherds products.
\subsubsection{Convergence of infinite $p$-adic products}
Let $\Lambda\subseteq V_{\Q_p}$ be a self-dual lattice.
The $p$-adic Borcherds products will be given by products of functions of the form
\[
\xi \mapsto \frac{\langle \xi, v_1 \rangle}{\langle \xi, v_2 \rangle}:=\frac{\langle v_\xi, v_1 \rangle}{\langle v_\xi, v_2 \rangle},\]
where $v_\xi\in V_{\C_p}$ is any representative of $\xi$.
The next lemma gives a criterion for the convergence of such products.
Its proof is a slight modification of the proof of Proposition \ref{prop:divexact}.

\begin{lemma}\label{convergencelemma}
Let $\Lambda\subseteq V_{\Q_p}$ be a self-dual $\Z_p$-lattice.
Suppose $v_1,v_2 \in \Lambda' $ satisfy
\begin{itemize}
\item $v_1 \equiv \alpha \cdot v_2 \bmod{p^k \Lambda}$ for some $\alpha \in \mathbb{\Z}_p^\times$ and
\item $q(v_i)=0 \bmod{p^k \Z_p}$.
\end{itemize}
Then for every $l\leq k$ and every $\xi$ lying in the standard affinoid $X_{p,\Lambda}^{\leq l}$,  
\[
\left|  \frac{1}{\alpha} \cdot \frac{\langle \xi,v_1 \rangle}{\langle \xi, v_2 \rangle} - 1 \right|_p \leq p^{l-k}.
\]
In particular, if a sequence of pairs $v_1^{i}, v_2^{i}\in \Lambda'$ satisfies
\begin{itemize}
\item $v_1^{i} \equiv \alpha_i \cdot v_2^{i} \bmod{p^{k_i} \Lambda}$ for some $\alpha_i \in \mathbb{Z}_p^\times$,
\item $q(v_i)=0 \bmod{p^{k_i} \Z_p}$.
\item $k_i \to +\infty$,
\end{itemize}
then the product
$$\prod_{i=1}^\infty \frac{\langle \xi, v_1^{i} \rangle}{\langle \xi, v_2^{i}  \rangle}$$
converges to an element of $\rM^\times/\Z_p^\times$. 
\end{lemma}

\begin{proof}
Suppose $\xi$ is represented by a primitive vector $v_\xi\in \Lambda'_{O_{\C_p}}.$
The inequality
\[
\left| \langle v_{\xi}, v_2 \rangle \right|_p\geq p^{-l}
\]
holds for all $\xi\in X_{p,\Lambda}^{\leq n}$ since $k\geq l$. 
Thus, we can compute
\begin{align*}
\left|  \frac{1}{\alpha} \cdot \frac{\langle \xi,v_1 \rangle}{\langle \xi, v_2 \rangle} - 1 \right|_p
= \left| \frac{\langle v_\xi,v_1 - \alpha \cdot v_2 \rangle}{\langle v_\xi, v_2 \rangle} \right|_p 
\leq p^{l-k}
\end{align*}
for all such $\xi$, which proves the first claim.
The second claim is a direct consequence of this.
\end{proof}

\subsubsection{Convergence}
Assume for the remainder of this section that $s=0$.
Fix elements $\beta_1,\ldots \beta_l \in \D_{\latt}$, positive rational numbers $m_1,\ldots,m_l\in \Z_{(p)}^{>0}$, and integers $c_{m_1,\beta_1},\ldots,c_{m_l,\beta_l}\in \Z$
such that
$$\sum_{i = 1}^l c_{m_i,\beta_i}\cdot a_f(m_i,\beta_i) = 0$$ 
for all modular forms $f \in M_{n/2,\D_{\latt}}(\cgroup_0(p))$.
Consider the Kudla--Millson divisor
\[
\sD_{\underline{c}}=\sum _{i=1}^{l}c_{m_i,\beta_i}\cdot \sD_{m_i,\beta_i}.
\]

Note that since the quadratic form $q$ is definite the set $\cO_\latt(m_i,\beta_i)\cap \Lambda$ is finite for every $\Z_p$-lattice $\Lambda\subseteq V_{\Q_p}$.
The following lemma states the convergence of $p$-adic Borcherds products.
\begin{proposition}\label{convergencedefinite}
For every self-dual $\Z_p$-lattice $\Lambda\subseteq V_{\Q_p}$
the infinite product
\begin{equation}\label{productexpansion}
\hat{J}_{\underline{c},\Lambda}(\xi):=
\lim_{k\to\infty} \prod_{i=1}^{l} \prod_{\substack{v\in \cO_\latt(m_i,\beta_i)\\ \ord_{\Lambda}(v)\geq -k}} \langle \xi, v \rangle^{c_{m_i,\beta_i} }
\end{equation}
converges in $\rM^\times / \Z_p^\times$.
Moreover, the following equality holds:
\[
\divmap(\hat{J}_{\underline{c},\Lambda})=\sD_{\underline{c}}.
\]
\end{proposition}

\begin{proof}
After scaling by appropriate powers of $p$ we get the equality
\[
\hat{J}_{\underline{c},\Lambda}(\xi)= \prod_{t=0}^\infty \hat{J}_{\underline{c},\Lambda,t}(\xi)
\]
in $\rM^\times / \Z_p^\times$,
where the factors are given by
\[
\hat{J}_{\Lambda,0}(\xi):=\prod_{i = 1}^l \prod_{\substack{v\in \cO_\latt(m_i,\beta_i)\\ \ord_{\Lambda}(v)\geq 0}} \langle \xi,v \rangle^{c_{m_i,\beta_i}}\\
\]
and
\[
\hat{J}_{\Lambda,t}(\xi):= \prod_{i = 1}^l \prod_{\substack{\cO_\latt(p^{2t}m_i,p^t \beta_i)\\ \ord_{\Lambda}(v)= 0}} (p^{-t}\cdot \langle \xi,v \rangle)^{c_{m_i,\beta_i}}
\]
for $t\geq 1$.
The conditions $q(v) = m_i p^{2t}$ and $v$ is primitive in $\Lambda$  
imply in particular that $v$ generates an isotropic line in $\Lambda / p^t \Lambda.$ 
It is then natural to decompose $J_{t}$ further as a product over isotropic lines:
$$\hat{J}_{\Lambda,t} = \prod_{\ell \subseteq \Lambda / p^n \Lambda \text{ isotropic}}\hat{J}_{\Lambda,t,\ell}$$
where $\hat{J}_{t,\ell}$ is the sub-product of $\hat{J}_t$ determined by insisting that the image of $v \bmod p^t\Lambda$ generates $\ell$.

We claim that the total sum of exponents in the product expression for $\hat{J}_{\Lambda,t,\ell}$ equals 0.
The desired sum of exponents equals
\[
\sum_{i = 1}^l c_{m_i,\beta_i} \cdot a_{m_i,\beta_i},\quad \mbox{where}\ a_{m_i,\beta_i} := \left|\left\{v\in (p^t\beta_i)\cap P(\Lambda,\ell) \ \vert\ q(v)=p^{2t}m_i \right\}\right|.
\]
Corollary \ref{maincombinatoriallemma} applied to the theta series of the definite quadratic lattice $\latt\cap p^t \Lambda$ implies that $a_{m_i,\beta_i}$ is the $(m_i,\beta_i)$-Fourier coefficient of a modular form of weight $n/2$, type $\rho_{\D_\latt}$ and level $\cgroup_0(p)$.
Therefore, the claim follows by the hypothesis made on the integers $c_{m_i,\beta_i}$.

By Lemma \ref{convergencelemma}, it follows that the product $\xi \mapsto \prod_{t = 0}^\infty \hat{J}_{\Lambda,t}(\xi)$ converges.
The last assertion is an immediate consequence of the construction of $\hat{J}_{\underline{c},\Lambda}$.
\end{proof}

\subsubsection{Action of the $p$-arithmetic group}
One easily deduces the following transformation law of $\hat{J}_{\underline{c},\Lambda}\in \rM^\times/\Z_p^{\times}$ under the action of the $p$-arithmetic group $\Gamma$ from the construction.

\begin{lemma}\label{equivariancedefinite}
Let $\Lambda\subseteq V_{\Q_p}$ be a self-dual $\Z_p$-lattice and $\gamma \in \Gamma$.
The equality
\[
\gamma.\hat{J}_{\underline{c},\Lambda}=\hat{J}_{\underline{c},\gamma\Lambda}
\]
holds.
\end{lemma}

\subsubsection{Independence of lattice at $p$}
We show that the function $\hat{J}_{\underline{c},\Lambda}$ is independent of the choice of lattice at $p$.
The proof relies on the properties of $p$-neighbouring lattices, which we are going to recall first.
Remember that two $\Z_p$-lattices in $V_{\Q_p}$ are called \emph{$p$-neighbours} if their intersection is of index $p$ in each of them. 
All $p$-neighbours of a self-dual lattice are also self-dual.
The following alternative description of $p$-neighbours of self-dual lattices is well-known (see for example the discussion after \cite[Lemma 4.3]{Chenevier}):
\begin{lemma}\label{pneighbourdescr}
Let $\Lambda_1,\Lambda_2\subseteq V_{\Q_p}$ be self-dual $\Z_p$-lattices, which are $p$-neighbours.
Then there exists a primitive isotropic vector  $w\in\Lambda'_1$ such that
\[
\Lambda_2 = \frac{1}{p}\Z_p w+ \left\{ u \in \Lambda_1\ \vert\ \langle u,w \rangle \equiv 0 \bmod p \right\}.
\]
\end{lemma}

The following important property of $p$-neighbours goes back to the seminal work of Kneser (cf.~\cite{Kneserneighbours}):
\begin{lemma}\label{pneighbourchain}
Any two self-dual $\Z_p$-lattices in $V_{\Q_p}$ can be connected by a finite chain of self-dual $\Z_p$-lattices in which any two consecutive lattices
 are $p$-neighbours.
\end{lemma}

\begin{lemma}\label{independencedefinite}
Let $\Lambda_1,\Lambda_2 \subseteq V_{\Q_p}$ be self-dual $\Z_p$-lattices.
The equality
\[
\hat{J}_{\underline{c},\Lambda_1}=\hat{J}_{\underline{c},\Lambda_2}
\]
holds in $\rM^\times/\Z_p^\times$.
\end{lemma}

\begin{proof}
By Lemma \ref{pneighbourchain} we may assume that $\Lambda_1$ and $\Lambda_2$ are $p$-neighbours.
Thus, by Lemma \ref{pneighbourdescr} there exists a primitive isotropic vectors $w\in\Lambda'_1$ such that
\[
\Lambda_2 = \frac{1}{p}\Z_p w+ \left\{ u \in \Lambda_1\ \vert\ \langle u,w \rangle \equiv 0 \bmod p \right\}.
\]
From the definition one immediately gets that
\begin{equation}\label{ratioequation}
\frac{\hat{J}_{\underline{c},\Lambda_1}}{\hat{J}_{\underline{c},\Lambda_1}}
=\lim_{k\rightarrow \infty}
\prod_{i=1}^{l}
\prod_{\substack{v\in \cO_\latt(m_i,\beta_i)\\ \ord_{\Lambda_1}(v)\geq -k\\ \ord_{\Lambda_2}(v)< -k}} \langle \xi, v \rangle^{c_{m_i,\beta_i} } \times
\prod_{i=1}^{l}
\prod_{\substack{v\in \cO_\latt(m_i,\beta_i)\\  \ord_{\Lambda_2}(v)\geq -k\\ \ord_{\Lambda_1}(v)< -k}} \langle \xi, v \rangle^{-c_{m_i,\beta_i}}
\end{equation}

Suppose that $\ord_{\Lambda_2}(v)\geq -k$.
Then we can write:
\[
v = \frac{1}{p^k}\left( \frac{a}{p} w + u \right)
\]
where $a\in \Z_p$ and $u \in \Lambda_1$ satisfies $\langle u,w \rangle \equiv 0 \bmod p.$
Then $\ord_{\Lambda_1}(v) < -k$ if and only if
\begin{enumerate}[(i)]
\item\label{cond1} $\ord_{\Lambda_2}(v) = -k$ and
\item\label{cond2} $p$ does not divide $a$.
\end{enumerate}
Let $\ell_{w/p,k}$ be the isotropic line in $\Lambda_2/p^k\Lambda_2$ generated by $1/p\cdot w$.
Conditions \eqref{cond1} and \eqref{cond2} are equivalent to $p^k v$ being primitive in $\Lambda_2$ and generating $\ell_{w/p,k}$.
In particular, the second term in \eqref{ratioequation} is equal to
$\hat{J}_{\Lambda_2,k,\ell_{w/p,k}}^{-1}$.
As in the proof of Proposition \ref{convergencedefinite}, we deduce that $\hat{J}_{\Lambda_2,k,\ell_{w/p,k}}$ converges to $1$ in $\rM^\times/\Z_p^\times$.
By symmetry of the $p$-neighbour relation, the first term in \eqref{ratioequation} also converges to $1$.
\end{proof}

\subsubsection{Proof of Theorem \ref{maindefinite1}}
Let $\Lambda\subseteq V_{\Q_p}$ be any self-dual $\Z_p$-lattice.
By Lemma \ref{independencedefinite} the function
\[
\hat{J}_{\underline{c}}:=\hat{J}_{\underline{c},\Lambda}\in \rM^\times/\Z_p^\times.
\]
is independent of the choice of $\Lambda\subseteq V_{\Q_p}$.
Moreover, it is $\Gamma$-invariant by Lemma \ref{equivariancedefinite} and, by construction, its divisor is equal to $\sD_{\underline{c}}$.
Hence, Theorem \ref{maindefinite1} follows.

\subsection{Proof of modularity in the hyperbolic case}\label{Borcherdshyperbolic}
The proof of the main theorem in the hyperbolic setting is very close to the one in the definite setting.
We will only indicate the main differences.
Throughout this section we assume that $s=1$ and $Q(\Q)\neq \emptyset$.
we fix elements $\beta_1,\ldots \beta_l \in \D_{\latt}$, positive rational numbers $m_1,\ldots,m_l\in \Z_{(p)}^{>0}$, and integers $c_{m_1,\beta_1},\ldots,c_{m_l,\beta_l}\in \Z$.
Moreover, we assume that each $m_i$ is compact with respect to $(V,q)$.

\subsubsection{Convergence in the hyperbolic case}
Let us fix rational isotropic lines $\ell_{-},\ell_{+}\in Q(\Q)$.
The set $\{v\in \cO_\latt(m_i,\beta_i)\cap \Lambda\ \vert\ \Delta_{v,\infty}\cap [\ell_{-},\ell_{+}]\neq 0\}$ is finite by Lemma \ref{modularsymbolintersection}.
Replacing the modularity of theta series of definite quadratic lattices by Theorem \ref{FMhyperbolic} in the proof of Proposition \ref{convergencedefinite} yields the following:
\begin{proposition}\label{convergenhyperbolic}
For every self-dual $\Z_p$-lattice $\Lambda\subseteq V_{\Q_p}$
the infinite product
\[
\hat{J}_{\underline{c},\Lambda}(\ell_{-},\ell_{+})(\xi):=
\lim_{k\to\infty} \prod_{i=1}^{l} \prod_{\substack{v\in \cO_\latt(m_i,\beta_i)\\ \ord_{\Lambda}(v)\geq -k}} \langle \xi, v \rangle^{c_{m_i,\beta_i}\cdot (\Delta_{v,\infty}\cap [\ell_{-},\ell_{+}]) }
\]
converges in $\rM^\times / \Z_p^\times$.
Moreover, the following equality holds:
\[
\divmap(\hat{J}_{\underline{c},\Lambda}(\ell_{-},\ell_{+}))=\widetilde{\sD}_{\underline{c}}(\ell_{-},\ell_{+}).
\]
\end{proposition}

\subsubsection{Properties}
From the construction one immediately deduces the following:
\begin{lemma}
For every self-dual $\Z_p$-lattice the assignment
\[
Q(\Q)\times Q(\Q)\longrightarrow \rM^\times/\Z_p{\times},\quad (\ell_{-},\ell_{+})\longmapsto \hat{J}_{\underline{c},\Lambda}(\ell_{-},\ell_{+})
\]
defines an $\rM^\times/\Z_p{\times}$-valued modular symbol.
Moreover,  
\[
\gamma.\hat{J}_{\underline{c},\Lambda}(\ell_{-},\ell_{+})=\hat{J}_{\underline{c},\gamma\Lambda}(\gamma.\ell_{-},\gamma.\ell_{+})
\]
for every $\gamma\in \Gamma$.
\end{lemma}

Again, by replacing the modularity of theta series of definite lattices by Theorem \ref{FMhyperbolic} the same arguments as in the proof of Lemma \ref{independencedefinite} give the following:
\begin{lemma} \label{independenhyperbolic}
Let $\ell_{-},\ell_{+}\in Q(\Q)$ be rational isotropic vectors and $\Lambda_1,\Lambda_2 \subseteq V_{\Q_p}$ self-dual $\Z_p$-lattices.
The equality
\[
\hat{J}_{\underline{c},\Lambda_1}(\ell_{-},\ell_{+})=\hat{J}_{\underline{c},\Lambda_2}(\ell_{-},\ell_{+})
\]
holds in $\rM^\times/\Z_p^\times$.
\end{lemma}

To summarize, for every every self-dual $\Z_p$-lattice $\Lambda\subseteq V_{\Q_p}$ the assignment
\[
(\ell_{-},\ell_{+})\longmapsto\hat{J}_{\underline{c}}(\ell_{-},\ell_{+}) 
\]
defines a $\Gamma$-invariant $\rM^\times/\Z_p^\times$-valued modular symbol, whose divisor is given by $\widetilde{\sD_{\underline{c}}}$, which proves Theorem \ref{mainhyperbolic1}.


\section{Special points and singular moduli}
\label{sec:special-points}
This chapter introduces the notion of {\em special points} 
on $X_p$ and defines the {\em value} of a rigid meromorphic cocycle at such a  point.
Since a special point is an element of 
$X_p$ whose stabiliser in $G$ contains a suitable
 kind of maximal torus, a general discussion in Section \ref{sec:maximal-tori} 
of maximal  tori in orthogonal groups    precedes
 the definition of special points, which is given  in Section \ref{sec:spdef}.
Section \ref{sec:galoisaction} describes a Galois action on special points akin to the Galois action on CM points of Shimura varieties, which is the basis  of a Shimura reciprocity law for special values of rigid meromorphic cocycles.

As in the previous chapters, $V$ is a quadratic space over $\Q$ of signature $(r,s)$ and 
dimension $n = r+s$. 
It is assumed throughout this chapter that $r\geq s$, and the algebraic closure  $\overline{\Q}$  is considered as a subfield of
$\C_p$ via a fixed chosen embedding.
Given a field $K$ with a fixed algebraic closure $\overline{K}$, the absolute Galois group of $K$  is denoted by $\mathcal{G}_{K}=\Gal(\overline{K}/K)$.

\subsection{Maximal tori in orthogonal groups}
\label{sec:maximal-tori}
A $\Q$-rational  torus $T\subseteq G$ is {\em maximal} in $G$ if 
$$
\dim(T)=\begin{cases} \frac{n}{2}& \mbox{if}\ n\ \mbox{is even,}\\ \frac{n-1}{2} &\mbox{if}\ n\ \mbox{is odd.}\end{cases}
$$
The Lie algebra $ {\rm Lie}(T) \subseteq \End_{\Q}(V)$ of $T$ is a $\dim(T)$-dimensional vector space
 of commuting, skew-adjoint endomorphisms of $V$.
It follows from \cite[Proposition 3.3]{BCKM}  that the subring generated by ${\rm Lie}(T)$ is an \'etale $\Q$-algebra  of dimension $2 \dim(T)$, denoted $E_T$. This algebra is endowed  with an involution $\sigma$ obtained by restricting the involution on ${\rm End}(V)$
 sending $T$ to its adjoint $T^*$. 
The  subalgebra $F_T:= E_T^\sigma$ of $\sigma$-fixed points is  
an \'etale algebra  over $\Q$ of dimension $\dim(T)$.

\subsubsection{Description of $V$ in terms of $T$}
  When $\dim(V)$ is odd,
  $E_T$ acts trivially on a non-degenerate one-di\-men\-sional subspace of 
  $V$ and acts faithfully on its orthogonal complement,
  denoted $V_T$.
	When $n$ is even, $E_T$ acts faithfully on $V_T :=V $.
In both cases, $V_T$ is a free $E_T$-module of rank one  and a free $F_T$-module of rank two.
 \begin{lemma}
 There are unique hermitian
 bilinear pairings (relative to $\sigma$)
  $$ \langle \ , \ \rangle_{E_T}\colon V_T\times V_T\rightarrow E_T,
  \qquad 
    \langle \ , \ \rangle_{F_T}\colon V_T\times V_T\rightarrow F_T,$$
 satisfying
 $$ \langle \lambda v_1, v_2\rangle = {\rm Tr}^{E_T}_\Q(\lambda \langle v_1,v_2\rangle_{E_T}), \quad \mbox{ for all } \lambda\in E_T,$$
 and likewise with $E_T$ replaced by $F_T$.
 \end{lemma}
 \begin{proof}
 The \'etale algebra $E_T$ is canonically identified with its $\Q$-linear dual via the trace form.
 For any  $v_1, v_2\in V_T$, the $\Q$-valued functional $\lambda \mapsto \langle \lambda v_1, v_2\rangle$ on $E_T$  can therefore be represented as
 $\lambda \mapsto {\rm Tr}^{E_T}_{\Q}(\lambda x)$ for a unique 
 $x\in E_T$, depending on $v_1$ and $v_2$. Setting $\langle v_1, v_2\rangle_{E_T} := x$, it is readily checked that the resulting $E_T$-valued function on $V_T \times V_T$
 has the asserted bilinearity properties. The same argument works with $E_T$ replaced by $F_T$.
 \end{proof}
 For any non-zero vector $v$ in the one-dimensional $E_T$-vector space $V_T$, the
    quantity 
    $$\nu:= \langle v,v\rangle_{F_T} = 2\langle v,v\rangle_{E_T} \in F_T^\times$$   is well-defined up to norms from $E_T^\times$, hence its image in $F_T^\times/{\rm N}^{E_T}_{F_T}(E_T^\times)$ is a well-defined invariant of the torus $T$.
 
 \begin{lemma}
 \label{lemma:eva}
 The quadratic space  $V_T$ is isomorphic to $E_T$ equipped with the quadratic form
 $$ q_\nu(\lambda) = {\rm Tr}^{F_T}_{\Q}(\nu\cdot \mathrm{N}^{E_T}_{F_T}(\lambda)).$$
 \end{lemma}
 \begin{proof}
 The choice of a $E_T$-module generator $v\in V_T$  determines a $\Q$ vector space isomorphism
 $$ \iota_v\colon E_T \rightarrow V_T, \qquad  \lambda \mapsto \lambda v.$$
 Under this identification, the original pairing on $V_T$ is transported to the pairing $\langle\!\langle \ , \ \rangle\!\rangle$ 
 on $E_T$ given by
 $$\langle\!\langle \lambda_1, \lambda_2\rangle\!\rangle  = 
 \langle  \lambda_1 v, \lambda_2 v\rangle = 
 {\rm Tr}^{E_T}_{\Q}(\lambda_1 \lambda_2^\sigma \langle v,v\rangle_{E_T}) 
 = {\rm Tr}^{F_T}_{\Q}((\lambda_1 \lambda_2^\sigma + 
 \lambda_1^\sigma \lambda_2)  \langle v,v\rangle_{E_T}),$$
 so that in particular
  $$\langle\!\langle \lambda, \lambda\rangle\!\rangle  
 = {\rm Tr}^{F_T}_{\Q}(\lambda \lambda^\sigma  \langle v,v\rangle_F)
 = {\rm Tr}^{F_T}_{\Q}(\nu \lambda \lambda^\sigma).$$
The lemma follows.
(cf.~also \cite[Proposition 3.9]{BCKM}).
\end{proof}
Lemma  \ref{lemma:eva} leads to a useful description of  the quadratic space $V$.
 If  $n$ is odd, 
the orthogonal complement of $V_T$ in $V$ is one-dimensional and non-degenerate, with quadratic form   given by $q_\alpha(x)=\alpha x^2$ for some $\alpha\in \Q^\times$ which is uniquely determined up to squares.
As quadratic spaces over $\Q$, 
one has
\begin{equation}
\label{eqn:VqT}
(V,q)\cong \left\{
\begin{array}{cl}
 (E_T, q_\nu), & \mbox{ if $n$ is even, } \\
 (E_T\oplus \Q, q_\nu \perp q_\alpha), & \mbox{ if $n$ is odd,}
 \end{array}\right.
 \end{equation}
The torus $T\subseteq G$ is identified with the unitary group $$ \mathrm{U}(E_T,\sigma) = \{ a\in E_T^\times \mbox{ for which } a a^\sigma = 1 \},$$  acting on $E_T$ via multiplication  under the identification \eqref{eqn:VqT}. (Cf.~\cite[Proposition 3.3]{BCKM}, for instance.)

\subsubsection{Aside: spinor norms}
This section recalls a well-known formula for the spinor norm of elements  of $\mathrm{U}(E_T,\sigma)$.
A typical such element   is of the form $a/a^{\sigma}$ for some $a\in E_T^\times$, by Hilbert's Theorem 90.

Fix a field $k$ of characteristic different from $2$.
\begin{lemma}
Let $l/k$ be a quadratic \'etale algebra and denote by $\sigma$ the non-trivial $k$-automorphism of $l$.
Fix an element $\nu\in k^\times$ and consider the $k$-vector space $l$ with the quadratic form given $q_k:=\nu\cdot \mathrm{N}^{l}_{k}$.
Then for every $a\in l^\times$ the spinor norm of multiplication with $a/a^{\sigma}$ is equal to $\mathrm{N}^{l}_k(a) \bmod (k^\times)^2$.
\end{lemma}
\begin{proof}
Without loss of generality we may assume that $\nu=1$ since the spinor norm of elements in $\SO(q_k)$ is independent of scaling the quadratic form.
The reflection $\tau_a$ at $a$ is given by
\[
\tau_a(\lambda)= \lambda- \frac{a\lambda^{\sigma}+a^{\sigma}\lambda}{a a^\sigma}\cdot a= - a/a^\sigma\cdot \lambda^{\sigma}
\quad \forall \lambda\in l.
\]
In particular, the equality $\tau_a \circ \tau_1 = a/a^\sigma$ holds.
By definition the spinor norm of $\tau_a$ is equal to $q_l(a)=\mathrm{N}^{l}_{k}(a)$ and, thus, the claim follows. 
\end{proof}

The following lemma describes the behaviour of spinor norms under restriction of scalars.
\begin{lemma}
Let $l/k$ be a finite \'etale algebra.
Let $(W,q_\ell)$ be a non-degenerate quadratic space over $l$.
Then $W$ equipped with $q_k:=\mathrm{Tr}^{l}_k \circ q_l$ defines a non-degnerate quadratic space over $k$.
There is a canonical embedding $\SO(q_l)\subseteq \SO(q_k)$ and the following diagram involving the spinor norms $\mathrm{sn}_{q_l}$ and $\mathrm{sn}_{q_k}$ associated to $q_l$ resp.~$q_k$ is commutative:
\begin{center}
 \begin{tikzpicture}
    \path
		(0,0) node[name=A]{$\SO(q_l)$}
		(3,0) node[name=B]{$\SO(q_k)$}
		(0,-1.5) node[name=C]{$l^\times /(l^\times)^2$}
		(3,-1.5) node[name=D]{$k^\times/(k^\times)^2$};
		\draw[->] (A) -- (B) ;
		\draw[->] (A) -- (C) node[midway, left]{$\mathrm{sn}_{q_l}$};
		\draw[->] (B) -- (D) node[midway, right]{$\mathrm{sn}_{q_k}$};
		\draw[->] (C) -- (D)  node[midway, above]{$\mathrm{N}^{l}_k$};
  \end{tikzpicture} 
\end{center}
\end{lemma}
\begin{proof}
In case $l$ is a field, this is \cite[Chapter 2, Lemma 5.5]{scharlau-book} and \cite[Chapter 9, Example 3.7]{scharlau-book}.
The proof carries over verbatim to the more general case.
\end{proof}

Combining the two lemmas above one deduces:
\begin{corollary}
\label{cor:spinornorm}
Let $k_1/k$ be a finite \'etale algebra, $\nu$ an element of $k_1^\times$ and $l/k_1$ a quadratic \'etale algebra with  involution $\sigma$.
Equip the $k$-vector space $l$ with the non-degnerate quadratic form $q_k(\lambda)=\mathrm{Tr}^{l}_{k_1}(\nu\cdot \mathrm{N}^{k_1}_{k}(\lambda))$.
For every $a\in l^\times$ the spinor norm of multiplication with $a/a^\sigma$ is equal to $\mathrm{N}^{l}_{k}(a) \bmod (k^\times)^2$. 
\end{corollary}

 \subsubsection{Maximally $\R$-split tori}
  Under the identification $T\cong U(E_T,\sigma)$, every arithmetic subgroup of $T(\Q)$ is a finite index subgroup of the group ${\cO}_{E_T/F_T}^\times$ of relative units of $E_T/F_T$, which fits into the exact sequence
 $$ 1 \rightarrow {\cO}_{E_T/F_T}^\times \lra {\cO}_{E_T}^\times \xlongrightarrow{\rm{N}} 
 \cO_{F_T}^\times,
 $$
 in which the last map has finite cokernel.

The  connected component of the commutative real Lie group
 $T(\R)$ is isomorphic to 
 $$ T(\R)_0 \simeq \R^t \times ({\mathbb S}_1)^{\dim(T)-t},$$
where ${\mathbb S}_1 = \C_1^\times$ is the unit circle. 
The integer $t$  is the {\em real rank} of $T$.
The $\Z$-rank of every arithmetic subgroup of $T(\Q)$ is less or equal than $t$ by Dirichlet's unit theorem for algebraic tori. (See for example \cite{Shyr}.)
It is equal to $t$ if and only if $T$ does not have a $\Q$-split subtorus.

  \begin{proposition}
  \label{realrank}
	For every maximal torus $T\subseteq G$ of real rank $t$ one has $t\le s$.
Moreover, equality holds if and only if $F_T$ is totally real and admits exactly $s$ real places that split into   pairs of real places of  $E_T$.
\end{proposition}
 \begin{proof}
 Let $a$ be the number of  real places of $F_T$ that lie below a complex place of $E_T$, let $b$ be the number of 
 real places of $F_T$ that lie below two real places of $E_T$, and let $c$ be the number of complex places of $F_T$, so that
 $a + b + 2c = \dim(T)$.
 
 Let $\nu \in F_T$ be as in Lemma \ref{lemma:eva}.
 The signature of the quadratic form $q_\nu$  can be read off directly from the invariants $a,b$, and $c$ above. More precisely, write $a = a_+ + a_-$, where $a_+$ is the number of real places $v$ of $F_T$ that lie below a  complex  place of $E_T$ and for which $v(\nu)>0$, and $a_-$ is the number of those places for which $v(\nu)<0$. A direct calculation shows that
 $$ (r,s) = a_+(2,0) + a_-(0,2) + b(1,1) + c (2,2) = (2a_+ + b + 2c, 2a_- + b + 2c),$$
while the real rank of $T(\R)$ is given by
$$ t = b + c.$$
Comparing the  equations for $s$ and $t$
 shows that $t\le s$, with equality if and only if $a_- = c = 0$.  The proposition follows.
 \end{proof}
  Proposition \ref{realrank}  motivates the following:
\begin{definition} 
  A maximal torus $T\subseteq G$ for which the equality $t=s$ is satisfied is called a 
  {\em maximally $\R$-split   torus} in $G$.
	\end{definition}

\subsubsection{Weight spaces and fixed points}
Let $T\subseteq G$ be a maximal torus and let $X^{\ast}(T)$ be its character group.
It is a free $\Z$-module of rank $\dim(T)$ endowed with a   
canonical action of $\mathcal{G}_{\Q}$.
The $\overline{\Q}$-vector space $V_{\overline{\Q}}$ admits a {\em weight space decomposition}
\[\label{wsd}
V_{\overline{\Q}}=\bigoplus_{\chi\in X^{\ast}(T)} V_{\chi},
\]
in which each $V_{\chi}$ is at most one-dimensional and 
 $V_{\chi}$ is non-zero if and only if 
 $V_{\chi^{-1}}$ is. 
Moreover, $V_{\chi}$ and $V_{\chi'}$ are orthogonal unless $\chi'=\chi^{-1}$.
If $\chi\neq 1$ and $V_\chi\neq \{0\}$,
 the weight $\chi$ is said to be {\em non-trivial}. The set $\mathcal{W}_T\subset X^\ast(T)$ of
   non-trivial weights  is a finite $\mathcal{G}_{\Q}$-set of cardinality $2 \dim(T)$,
 on which the involution $\sigma $ acts by sending $\chi$ to $\chi^{-1}$.  

For all  $\chi\in\mathcal{W}_T$, the weight space
 $x_\chi:=V_{\chi}$ is isotropic and, thus, defines a point in $Q(\overline{\Q})$.
Moreover, the direct sum $H_\chi:= x_{\chi}\oplus x_{\chi^{-1}}$ is a hyperbolic plane, 
 and  
\begin{align}
\label{eqn:hyperbolicdecomp}
 V_{T,\overline{\Q}}= \bigoplus_{\chi\in \mathcal{W}_T/\sigma} H_\chi
\end{align}
expresses $V_{T,\overline{\Q}}$ as an orthogonal direct sum of hyperbolic planes over $\overline{\Q}$.

Conversely, let  $x$ be a point of $Q(\C_p)$ whose stabiliser contains $T(\C_p)$. 
Then $x$ is the weight space of a unique non-trivial weight $\chi_x$ and thus $x$ is algebraic, that is, $x$ lies in the subset $Q(\overline{\Q})\subseteq Q(\C_p)$.
Thus, the map
\[
\mathcal{N}_T\xlongrightarrow{\sim} \mathcal{W}_T,\quad x\longmapsto \chi_x,
\]
sets up a $\mathcal{G}_\Q$-equivariant bijection from the set $\mathcal{N}_T\subseteq Q(\overline{\Q})$ of isotropic eigenlines of $T$ to $\mathcal{W}_T$.
This bijection induces an action of $\sigma$ on $\mathcal{N}_T$ satisfying
\begin{equation}
\label{sigmainversion}
\chi_{\sigma(x)}=\chi_{x}^{-1}.
\end{equation}

\begin{lemma}
\label{galoischi}
Let $T\subseteq G$ be a maximal torus.
Then
\[
\chi_{\tau(x)}=\tau\circ \chi_x \circ \tau^{-1}
\]
holds for every point $x\in \mathcal{N}_T$ and every $\tau\in\mathcal{G}_\Q$.
\end{lemma}
\begin{proof}
Given $v\in \ell_{\chi_x}$ and $t\in T(\overline{\Q})$ we compute
\begin{align*}
\chi_{\tau(v)}(t)\cdot \tau(v)
=t.\tau(v)
=\tau(\tau^{-1}(t).v)
=\tau(\chi_x(\tau^{-1}(t))\cdot v)
=\tau(\chi_x(\tau^{-1}(t)))\cdot v,
\end{align*}
which proves the claim.
\end{proof}

Lemma \ref{galoischi} and equation \eqref{sigmainversion} immediately imply:
\begin{corollary}
\label{galoissigma}
Let $T\subseteq G$ be a maximal torus.
Then
\[
\sigma(\tau(x))=\tau(\sigma(x))
\]
holds for all $\tau\in\mathcal{G}_\Q$ and all $x\in\mathcal{N}_T$.
\end{corollary}


The weight space decomposition of a maximal torus $T\subseteq G$ can be described
 in terms of the \'etale algebra $E_T\subset {\rm End}(V)$ with involution associated to $T$. 
Let $x\in \mathcal{N_T}$ be a fixed point of $T(\C_p)$.
The action of $\mathrm{Lie}(T)$ on the isotropic line spanned by $x$ 
determines a ring homomorphism
\[
\varphi_x\colon E_T\longrightarrow \overline{\Q}\subseteq \C_p \]
whose restriction to $\mathrm{U}(E_T,\sigma)\subseteq E_T^\times$ agrees with $\chi_x$.
These algebra homomorphisms are distinct and thus they exhaust the set of all $\Q_p$-algebra homomorphisms from $E_T$ to $\overline{\Q}$.
The bijection
\[
\mathcal{N}_T\xlongrightarrow{\sim}\Hom(E_T,\overline{\Q}),\qquad x\longmapsto \varphi_x,
\]
is $\mathcal{G}_{\overline{\Q}}$-equivariant.
In particular, two points $x,x'\in\mathcal{N}_T$ lie in the same $\mathcal{G}_{\Q}$-orbit if and only if $\ker(\varphi_x)=\ker(\varphi_{x'})$.
Moreover, the relation
\[
\varphi_{\sigma(x)}=\varphi_{x}\circ \sigma
\]
holds.

\subsubsection{Fixed points in $X_p$}
\label{sectionnonsplit}

The ring homomorphism $\varphi_x$ induces a $\Q_p$-algebra homomorphism
\[
\varphi_{x,p}\colon E_T\otimes_\Q \Q_p \longrightarrow \C_p.
\]
Two points $x,x'\in\mathcal{N}_T$ lie in the same $\mathcal{G}_{\Q_p}$-orbit if and only if $\ker(\varphi_{x,p})=\ker(\varphi_{x',p})$.
In case $n$ is odd, we extend $\varphi_{x,p}$ by zero to a homomorphism $\varphi_{x,p}(E_T\oplus\Q)\otimes_\Q \Q_p \rightarrow \C_p$.
The following lemma clarifies when a fixed point of $T(\C_p)$ belongs to $X_p$.
\begin{lemma}
Let $T\subseteq G$ a maximal torus and $x\in \mathcal{N}_T$ a fixed point of $T(\C_p)$.
\begin{enumerate}[(a)]
\item Assume that $n$ is even.
Then $x$ belongs to $X_p$ if and only if
the subspace $\ker(\varphi_{\sigma(x),p})\subseteq E_T\otimes \Q_p$ contains no non-zero isotropic vectors (relative to the quadratic form $q_\nu$).
\item Assume that $n$ is odd.
Then $x$ belongs to $X_p$ if and only if
the subspace $\ker(\varphi_{\sigma(x),p})\subseteq (E_T\oplus\Q)\otimes \Q_p$ contains no non-zero isotropic vectors (relative to the quadratic form $q_\nu\perp q_\alpha$.)
\end{enumerate}
 \end{lemma}
 \begin{proof}
Under the identification $V\cong E_T$ respectively $V\cong E_T\oplus \Q$ the orthogonal complement of the isotropic line $x$ is given by the kernel of $\varphi_{\sigma(x)}$.
Thus, the lemma follows from the definition of $X_p$.
  \end{proof}

\begin{corollary}
Let $T\subseteq G$ be a maximal torus.
Assume that $E_T\otimes\Q_p$ is a field.
Then every fixed point $x\in\mathcal{N}_T$ belongs to $X_p$.
\end{corollary}
\begin{proof}
The claim follows from the preceding lemma since since every $\Q_p$-algebra homomorphism  $\varphi \colon E_T\otimes\Q_p \rightarrow \C_p$ is injective.
\end{proof}


\begin{definition}
A point $x \in X_p$ is called a  {\em toric fixed point}
if it is stabilized by  $T(\C_p)$ where  
  $T\subseteq G$   is a maximal torus.
In that case $(x,T)$ is called a {\em toric pair}.
\end{definition}

\begin{remark}
\label{uniquetorus}
Given a toric fixed point $x\in X_p$ there are in general several maximal tori $T\subseteq G$ that stabilise $x$.
But under the assumption that the \'etale algebra $E_T$ is a field, $T$ is uniquely determined by $x$.
Indeed, all the fixed points of $T(\C_p)$ in $Q(\C_p)$ are $\mathcal{G}_\Q$-conjugated in this case.
In particular, if $g\in G(\Q)$ stabilizes one fixed point, it stabilizes all of them.
The common stabilizer of all fixed points in $G(\C_p)$ is given by $T(\C_p)$, which proves the claim.
\end{remark}

In general, the algebra $E_T$ associated to a maximal torus $T\subseteq G$ does not have to be a field.
But the existence of a fixed point in $X_p$ has strong implications on the splitting behaviour of $p$-adic places in the extension $E_T/F_T$:
\begin{lemma}
\label{nonsplit}
Let $(x,T)$ be a toric pair.
Then every $p$-adic place of the \'etale $\Q$-algebra $F_T$ is non-split in $E_T$.
Moreover, $y$ and $\sigma(y)$ lie in the same $\mathcal{G}_{\Q_p}$-orbit for all $y\in \mathcal{N}_T$.
\end{lemma}
\begin{proof}
Let $S_p$ be the set of $p$-adic places of $F_T$.
The \'etale $\Q_p$-algebra $F_T \otimes \Q_p$ decomposes into the product over the completions $F_{\mathfrak{p}}$, $\mathfrak{p}\in S_p$.
For $\mathfrak{p} \in S_p$ we put $E_{\mathfrak{p}} = E_T \otimes_{F_T} F_{\mathfrak{p}}$.
The involution $\sigma$ induces an $F_{\mathfrak{p}}$-linear involution on $E_{\mathfrak{p}}$.
We may decompose $T_{\Q_p}$ as follows:
\[
T_{\Q_p}\cong \mathrm{U}(E_T\otimes \Q_p ,\sigma)\cong \prod_{\mathfrak{p}\in S}\mathrm{U}(E_{\mathfrak{p}} ,\sigma) .
\]
Assume that there exists a place $\mathfrak{p}\in S_p$ that is split in $E_T$.
Then $E_{\mathfrak{p}}\cong F_{\mathfrak{p}} \times F_{\mathfrak{p}}$ and thus, the unitary group $\mathrm{U}(E_{\mathfrak{p}} ,\sigma)$ is isomorphic to the Weil restriction of $\mathbb{G}_{m,F_{\mathfrak{p}}}$ from $\F_p$ to $\Q_p$.
In particular, it contains a $\Q_p$-split subtorus.
Proposition \ref{GITXp} now implies that no fixed point of $T(\C_p)$ lies in $X_p$.
This proves the first claim.
The second claim is a direct consequence of the first one.
\end{proof}

Lemma \ref{nonsplit} implies that for every toric pair $(x,T)$ the subspace of $V_{\overline{\Q}}$ generated by the $\mathcal{G}_\Q$-conjugates of $x$ is a direct sum of hyperbolic planes and, in particular, it is non-degenerate.
By construction, it is the base change of a (non-degenerate) rational subspace $V_x\subseteq V$ that is stable under the action of $T$.
Let $T\rightarrow \SO(V_x)$ be the induced homomorphism, whose image is a maximal torus in $\SO(V_x)$ that we denote by $T_x$.
The inclusion $\SO(V_x)\hookrightarrow \SO(V)$ restricts to a homomorphism $T_x \hookrightarrow T$.
It follows that $T_x$ is a factor of $T$.
By construction, the \'etale $\Q$-algebra $E_{T_x}$ (with involution $\sigma$) associated to the torus $T_x$ is a field.
Applying Remark \ref{uniquetorus} to the quadratic space $V_x$ and the torus $T_x$ we see that $T_x$ does not depend on the choice of $T$.
In particular, the assignment $x\mapsto \sigma(x)$ is independent of the choice of the torus $T$ stabilizing $x$.
There is an equality $\varphi_x(E_T)=\varphi_x(E_{T_x})$ of subfields of $\C_p$.

\subsection{Special points and values of rigid meromorphic cocycles}
\label{sec:spdef}
After defining \emph{oriented special points} and the value of a rigid meromorphic cocycle at such a point a crude form of the algebraicity conjecture on these special values is formulated.
\subsubsection{Special points}

\begin{definition}
A point $x \in X_p$ is said to be {\em special} if there exists a maximal $\R$-split torus $T\subseteq G$ such that $T(\C_p)$ stabilizes $x$.
In that case $(x,T)$ is called a special pair.
\end{definition}

Let $(x,T)$ be a special pair.
When $V$ is three-dimensional,  the \'etale algebra $E_T$ with involution 
is a quadratic field in which $p$ is non-split.
It is imaginary when $s=0$, and real when $s=1$.

When $V$ is four-dimensional, the algebra $E_T$ can be of one of the following types:
\begin{enumerate}
\item  
A quadratic extension of a real quadratic field with exactly $2s$ real embeddings.
The \'etale $\Q_p$-algebra $E_T\otimes \Q_p$ is either a field, or a direct sum of two quadratic extensions of $\Q_p$.
\item A direct sum $E_1 \oplus E_2$ of quadratic fields in which the prime $p$ does not split, and where exactly $s$ of the fields are real quadratic.
\end{enumerate}
In all of these cases, the torus $T$ is equal to the stabilizer of $x$ in $G$.
 
For general $V$ of rank $n$, there is a decomposition
$$ E_T \otimes \Q_p =  E_1 \oplus E_2,$$
where $E_{1}$ is a field and $E_2$ an \'etale algebra over 
$\Q_p$ of rank $\le 4$ when $n$ is even, and rank $\le 2$ when $n$ is odd.
In particular,
\[
\dim E_{1} \geq \begin{cases}
n-3 & \mbox{ if $n$ is odd,}\\
n-4 & \mbox{ if $n$ is even.}
\end{cases}
\]
Only when equality is attained in these inequalities does the stabiliser of $x$ contain more than one maximal torus.

\subsubsection{Oriented special points}
The following proposition controls the rank of $p$-arithmetic subgroups of the stabilizer of special points of $X_p$.
\begin{proposition}\label{torusrank}
Let $(x,T)$ be a special pair.
The rank of every $p$-arithmetic subgroup of $T(\Q)$ is equal to $s$.
\end{proposition}
\begin{proof}
By Proposition \ref{GITXp} the base change of $T$ to $\Q_p$ does not contain any $\Q_p$-split subtorus.
In particular, $T$ does not contain any $\Q$-split subtorus. 
Hence, the $\Z$-rank of every $p$-arithmetic subgroup of $T(\Q)$ is equal to the real rank of $T$ by Dirichlet's $S$-unit theorem for tori (see \cite{Shyr}).
The claim now follows from the definition of special points. 
\end{proof}

In particular, the subgroup $\Gamma_T=\Gamma\cap T(\Q)$ is a finitely generated abelian group of rank $s$.
Thus,  
\[
H_s(\Gamma_T,\Z) \simeq \Z \times \mbox{finite abelian group}.
\]
This motivates the following:
\begin{definition} 
\label{specialpair}
An \emph{oriented special point (of level $\Gamma$)} is a triple $\vec{x}=(x,T_{\vec{x}},o_{\vec{x}})$ consisting of a special pair $(x,T_{\vec{x}})$ together with a homology class $o_{\vec{x}} \in H_s(\Gamma_{T_{\vec{x}}}, \mathbb{Z})$.  
\end{definition}
In the present context, oriented special points play the role of special points in the classical theory of Shimura varieties.

\begin{remark}
\label{rem:boundedtorsion}
Since the order of torsion elements in $G(\Q)$ is bounded from above, it follows that the torsion in $H_s(\Gamma_T,\Z)$ is bounded from above, where $(x,T)$ ranges over all special pairs.
\end{remark}

\subsubsection{Regularity}
A meromorphic function $f\in \rM^\times$ is said to be \emph{regular at a point $\xi\in X_p$} if $\xi$ does not lie in the support of the divisor of $f$.
Let $\vec{x}$ be an oriented special point and $J\in H^{s}(\Gamma,\rM^\times)$.
Restriction to $\Gamma_{T_{\vec{x}}}$ followed by taking cap product with the orientation $o_{\vec{x}}$ yields a homomorphism
$$H^s(\Gamma, \rM^\times)\xlongrightarrow{\mathrm{res}} H^s(\Gamma_{T_{\vec{x}}}, \rM^\times)\xlongrightarrow{\cap o_{\vec{x}}} (\rM^\times)_{\Gamma_{T_{\vec{x}}}}$$
and we write $J_{\vec{x}}$ for the image of $J$ under this map.

\begin{definition}
\label{defregularity}
A class $J\in H^{s}(\Gamma,\rM^\times)$ is said to be {\em regular at the oriented special point $\vec{x}$} if the homology class $J_{\vec{x}}\in(\rM^\times)_{\Gamma_{T_{\vec{x}}}}$ admits a representative $\widetilde{J}_{\vec{x}}\in \rM^{\times}$ that is regular at $x$.
\end{definition}

\begin{remark}
In the case where $E_T$ is a field and $n$ is even, any function in $\rM^\times$ is automatically regular at every special point $x$, since no non-zero vector in $V$ is orthogonal to $x$.
\end{remark}

For a $\Gamma$-invariant subset $S \subseteq X_p,$ let $(\rM^\times)_S$ denote those rigid meromorphic functions on $X_p$ which are regular at all $\xi \in S$.
Suppose that a class $J$ has a lift to $H^s(\Gamma, (\rM^\times)_{\Gamma x})$,
then it is clearly regular at $(\vec{x},T)$.
Note that the $p$-adic Borcherds products constructed in Theorem \ref{maindefinite2} and \ref{mainhyperbolic2} naturally define classes in $H^s(\Gamma, (\rM^\times)_{S})$ where $S$ are certain unions of rational quadratic divisors.
 
\subsubsection{Values of cohomology classes at oriented special points}
\begin{lemma}\label{rmcspecialpairevaluation}
Let $J$ be a class in $H^{s}(\Gamma,\rM^\times)$ that is regular at $\vec{x}$ and let $\widetilde{J}_{\vec{x}}, \widetilde{J}_{\vec{x}}^{\prime}\in \rM^\times$ be two representatives of $J_{\vec{x}}$ that are regular at $x$.
Then  
\[ \widetilde{J}_{\vec{x}}(x)=\pm \widetilde{J}^{\prime}_{\vec{x}}(x).\]
\end{lemma}
\begin{proof}
We start with the following simple observation:
let $g$ be an element of $G(\Q)$ and $v\in V$, $q(v)\neq 0$, an eigenvector of $g$ with eigenvalue $\lambda$.
Then $\lambda$ is either equal to $1$ or $-1$ since
\[
q(v)=q(gv)=q(\lambda v)=\lambda^{2} q(v).
\]
Thus there exists a set $\mathcal{R}\subseteq V_{+}$ of vectors of positive length with the following properties:
\begin{itemize}
\item any two distinct elements in $\mathcal{R}$ are linearly independent,
\item any vector $v\in V_{+}$ with $x\in \Delta_{v,p}$ is a rational multiple of an element in $\mathcal{R}$ and
\item if $v$ is an element of $\mathcal{R}$ and $t\in T(\Q)$, then either $tv$ or $-(tv)$ is an element of $\mathcal{R}$.
\end{itemize}
We define an action $T(\Q)$ of $\mathcal{R}$ by setting $t\ast v:=\pm tv$, which in turn induces a linear action on the module $\bigoplus_{\mathcal{R}}\Z$.
Fix an auxiliary element $v_0\in V_{+}$ of positive length such that $x\notin \Delta_{v_0,p}$ and consider the map
\[
\bigoplus_{v\in \mathcal{R}}\Z \longrightarrow \rM^\times
,\quad a=(a_v)\longmapsto f_a^{v_0}(\xi):=\prod_{v\in \mathcal{R}}\left(\frac{\langle v, \xi \rangle}{\langle v_0, \xi\rangle}\right)^{a_v}.
\]
By construction,
\[
\divmap(f_a^{v_0})=\sum_{v\in\mathcal{R}} a_v  \cdot \Delta_{v,p}.
\]
Moreover, every meromorphic function $f\in\rM^\times$ can uniquely written as a product
\begin{equation}\label{divdecomp}f=f_a^{v_0} \cdot h\end{equation}
with $a\in \bigoplus_{\mathcal{R}}\Z$ and $h\in \rM^\times$ regular at $x$.
If $\widetilde{v}_0$ is a different choice of auxiliary element, the quotient
\[f_a^{\widetilde{v}_0/v_0}:=f_a^{\widetilde{v}_0}/f_a^{v_0}
=\prod_{v\in\mathcal{R}}\left(\frac{\langle v_0, x \rangle}{\langle \widetilde{v}_0, x\rangle}\right)^{a_v}\]
is clearly regular at $x$.
A straightforward calculation shows that for every $t\in T(\Q)$,  
\begin{equation}\label{equalone}
f_a^{tv_0/v_0}(x)=1.
\end{equation}
Moreover, it follows directly from the defintions that the equality
\begin{equation}\label{almostequivariance}
t.f_a^{v_0}=\pm f^{t v_0}_{t\ast a}
\end{equation}
holds for every $t\in T(\Q).$
We may write the quotient $\widetilde{J}_{\vec{x}}/\widetilde{J}^{\prime}_{\vec{x}}$ in the form
\[
\widetilde{J}_{\vec{x}}/\widetilde{J}^{\prime}_{\vec{x}}
= \prod_{i=1}^{m}\frac{t_i.(f_{a_i}^{v_0}\cdot h_i)}{f_{a_i}^{v_0}\cdot h_i}
= \prod_{i=1}^{m}\frac{t_i.f_{a_i}^{v_0}}{f_{a_i}^{v_0}}\cdot \prod_{i=1}^{m}\frac{t_i. h_i}{h_i}
\]
with $t_i\in \Gamma_x$, $h_i\in \rM^\times$ meromorphic functions that are regular at $x$ and $a_1,\ldots,a_m$ tuples as above.
As the functions $h_i$ are regular at $x$, it follows that
$$\frac{t_i.h_i(x)}{h_i(x)}=\frac{h_i(t_i^{-1}x)}{h_i(x)}=\frac{h_i(x)}{h_i(x)}=1$$
for all $i=1,\ldots, m$.
Using \eqref{almostequivariance} we may rewrite the remaining term as follows:
\[
\prod_{i=1}^{m}\frac{t_i . f_{a_i}^{v_0}}{f_{a_i}^{v_0}} 
= \pm \prod_{i=1}^{m} \frac{f_{t_i\ast a_i}^{t_i v_0}}{f_{a_i}^{v_0}} 
= \pm \prod_{i=1}^{m} \frac{f_{t_i\ast a_i}^{v_0}}{f_{a_i}^{v_0}} \cdot \prod_{i=1}^{m} f_{t_i\ast a_i}^{t_i v_0/v_0}.
\]
As this function is regular at $x$ and the second factor is regular at $x$, the uniqueness of the decomposition \eqref{divdecomp} implies that the first factor is equal to one.
The value of the second factor at $x$ is equal to one by \eqref{equalone}.
This proves the claim.
\end{proof}

\begin{definition}
\label{defvalue}
Let $J$ be a class in $H^{s}(\Gamma,\rM^\times)$ that is regular at the oriented special point $\vec{x}$.
The \emph{value of $J$ at $\vec{x}$} is defined by
\[ J[\vec{x}]:=\widetilde{J}_{\vec{x}}(x) \in \C_p^\times/\{\pm 1\}\]
where $\widetilde{J}_{\vec{x}}\in \rM^\times$ is any representative of $J_{\vec{x}}$ that is regular at $x$.
\end{definition}
Note that the value $J[\vec{x}]$ lies in the field of definition of $x$ over $\Q_p$.
In case $s=0$, one has $H^0(\Gamma_x,\Z)=\Z$.
If the orientation of $\vec{x}$ corresponds to the integer $n$ under this identification, then $J[\vec{x}]$ is simply the $n$-th power of the value of the $\Gamma$-invariant function $J$ at $x$.
For general $s$ scaling the orientation defines an action of $\Z$ on the set of oriented special points.
One clearly has
\[
J[n\cdot\vec{x}]=J[\vec{x}]^{n}.
\]
In particular, if the orientation of $\vec{x}$ is torsion, $J[\vec{x}]$ is a root of unity whose order is bounded from above by Remark \ref{rem:boundedtorsion}.

\subsubsection{$\Gamma$-invariance of special values}
Let $(x,T)$ be a special pair and $g$ an element of $G(\Q)$.
Then $(g x, g T g^{-1})$ is also a special pair and  
$\Gamma_{g T g^{-1}}=g\Gamma_x g^{-1}$.
Write
\[
c_g\colon H_s(\Gamma_T,\Z)\rightarrow H_s(\Gamma_{g T g^{-1}},\Z)
\]
for the homomorphism induced by conjugation with $g$.
If $\vec{x}=(x,T,o)$ is an oriented special point, then so is $g\vec{x}:=(gx, g T g^{-1}, c_g(o))$.
Since conjugation by $\gamma\in\Gamma$ induces the identity on $H^{s}(\Gamma, \rM^{\times})$ we immediately get the following:
\begin{proposition} \label{invariantevaluation}
Suppose that a class $J\in H^{s}(\Gamma,\rM^\times)$ is regular at the oriented special point $\vec{x}$.
Then $J$ is regular at $\gamma.\vec{x}$ for all $\gamma\in \Gamma$.
Moreover, the equality
\[ J[\vec{x}] = J[\gamma.\vec{x}]\]
holds.
\end{proposition}

\subsubsection{Algebraicity of special values}
In general we do not expect the values of rigid meromorphic cocycles at oriented special points to be algebraic.
Indeed, in signature $(n,0)$ every constant function with values in $\Q_p^\times$ is a rigid meromorphic cocycle.
More generally, since $p$-arithmetic groups are of type (VFL), the homomorphism
\[
H^s(\Gamma,\Z)\otimes_\Z \Q_p^\times \longrightarrow H^s(\Gamma,\Q_p^\times)
\]
has finite kernel and cokernel and the group $H^s(\Gamma,\Z)$ is finitely generated.
It follows that for every $J\in H^s(\Gamma,\Q_p^\times)$ there exists a finitely generated submodule $\Pi_J\subseteq \Q_p^\times$ such that
\[
J[\vec{x}]\in \Pi_J
\]
for all oriented special points $\vec{x}$.
Moreover, the rank of $\Pi_J$ is bounded by the rank of $H^s(\Gamma,\Z)$.
We now state a crude form of the algebraicity conjecture:
\begin{conjecture}\label{crudeconjecture}
Let $J\in\mathcal{RMC}(\Gamma)$ be a rigid meromorphic cocycle.
There exists a finitely generated submodule $\Pi_J\subseteq \Q_p^\times$
such that
\[
J[\vec{x}]\in \overline{\Q}^\times \Pi_J 
\]
for every oriented special point $\vec{x}$ at which $J$ is regular.
Moreover, the rank of $\Pi_J$ is less or equal to the rank of $\rk_\Z H^s(\Gamma,\Z)$.
\end{conjecture}
By the discussion following Definition \ref{defvalue} the conjecture is vacuous for all points whose orientation is torsion.

\begin{remark}
With similar methods as in the proof of Proposition \ref{cyclezero} one can determine the rank of $H^s(\Gamma,\Z)$.
For example, in signature $(r,1)$ one deduces that $H^1(\Gamma,\Z)$ is always torsion and, hence, Conjecture \ref{crudeconjecture} predicts that the special values of rigid meromorphic cocycles are in fact algebraic.
\end{remark}

\subsection{Class group action on special points}
\label{sec:galoisaction}
In this section we associate to a toric fixed  point $x$ a cocharacter $\mu_x$, a reflex field $E_x$, and a reflex norm $r_x$.
This leads to an action of the idele class group of $E_x$ on the set of oriented special points modulo $\Gamma$ imitating the Galois action on CM points of Shimura varieties.

\subsubsection{The cocharacter $\mu_x$ and its reflex field}
\label{cocharacter}
Let $T\subseteq G$ be a maximal torus and $x\in Q(\C_p)$ a point stabilized by $T(\C_p)$.
Remember that $x$ corresponds to a non-trivial weight $\chi$ of $T$ and, in particular, $x$ is algebraic.
We consider $x$ as a line in $V_{\overline{\Q}}$.
As before, write $\sigma(x)$ for the $T$-weight space corresponding to the weight $\chi^{-1}$ and consider the hyperbolic summand $H_{x}=x\oplus \sigma(x)$ of $V_{\overline{\Q}}$ attached to $x$ by \eqref{eqn:hyperbolicdecomp}.
The assignment
\begin{align*}
\mu_{x,T}(t).v=\begin{cases}
t\cdot v & \mbox{if}\ v\in x,\\
t^{-1}\cdot v & \mbox{if}\ v\in \sigma(x),\\
v & \mbox{if}\ v\in H_{x}^\perp.
\end{cases}
\end{align*} 
defines a cocharacter $\mu_{x,T}\colon \mathbb{G}_{m,\overline{\Q}} \rightarrow T_{\overline{\Q}}$.

Assume for the moment that $x$ is an element of $X_p$.
Then, by the discussion following Lemma \ref{nonsplit} the point $\sigma(x)$ does only depend on $x$ and not $T$.
Hence, the cocharacter $\mu_{x,T}$ is independent of $T$ and defines a homomorphism
\[
\mu_{x}\colon \mathbb{G}_{m,\overline{\Q}} \longrightarrow T_{x,\overline{\Q}},
\]
where $T_x$ is the factor of $T$ associated to $x$ in Section \ref{sectionnonsplit}.

\begin{definition}
Let $x$ be a toric fixed point.
The \emph{reflex field $E_x \subseteq \C_p$ of $x$} is the field of definition of $\mu_x$.  
\end{definition}
The fields of definition of the points $x$ and $\sigma(x)$ agree and, thus, the reflex field $E_x$ is just the field of definition of the point $x\in X_p$.
It is a finite extension of $\Q$.
More precisely, one has the following equality of fields:
\[\varphi_x(E_T)=\varphi_x(E_{T_x})=E_x\subseteq \C_p.\]
The isomorphism $\varphi_x\colon E_{T_x}\rightarrow E_x$ induces an isomorphism
\[
\varphi_x\colon T_x\xlongrightarrow{\cong} \mathrm{U}(E_x,\sigma)
\]
of algebraic groups over $\Q$.

\subsubsection{Abstract reflex norm associated to a cocharacter}
\label{sec:abstractreflexnorm}
Let $T$ be a torus over $\mathbb{Q}$ with cocharacter lattice $X_\ast(T)$.
Let $\mu\colon \mathbb{G}_{m,E} \rightarrow T_E$ be a cocharacter defined over a finite extension $E$ of $\Q$.
Write $\underline{E}^\times:=\mathrm{Res}_{E/\Q}(\mathbb{G}_{m,E})$ for the Weil restriction of the multiplicative group $\mathbb{G}_{m,E}$ from $E$ to $\Q$.
As explained in the proof of \cite[Chapter II, Theorem 2.4]{Oesterle} there exists a canonical isomorphism
\begin{align}
\label{eq:inducedcochar}
X_\ast(\underline{E}^\times)\cong 
\mathrm{Ind}_{\mathcal{G}_{E}}^{\mathcal{G}_{\Q}} \mathbb{Z}
:= \left\{ f\colon \mathcal{G}_{E} \backslash \mathcal{G}_{\Q} \rightarrow \mathbb{Z} \right\}.
\end{align}
By Frobenius reciprocity, the $\mathcal{G}_{E}$-equivariant map of cocharacter lattices
\[
\overline{\mu}\colon \mathbb{Z} \cong X_\ast(\mathbb{G}_{m}) \longrightarrow X_\ast(T)
\]
induces the $\mathcal{G}_{\Q}$-equivariant homomorphism
\[
\overline{r_\mu}\colon \mathrm{Ind}_{\mathcal{G}_{E}}^{\mathcal{G}_{\Q}}\mathbb{Z} \longrightarrow X_\ast(T),
\quad f \longmapsto \sum_{\tau \in \mathcal{G}_{E} \backslash \mathcal{G}_{\Q}} f(\tau)\cdot \tau^{-1}.\mu
\]
We are led to the following definition:
\begin{definition}
Let $T / \mathbb{Q}$ be a torus and $\mu\colon \mathbb{G}_{m,E} \rightarrow T_E \in X(T)$ be a cocharacter defined over a finite extension $E$ of $\Q$.
The \emph{reflex norm attached to the pair $(\mu,E)$} is the unique $\Q$-rational homomorphism
$$r_\mu\colon \underline{E}^\times \longrightarrow T$$
that induces the homomorphism $\overline{r_\mu}$ on cocharacter lattices.
\end{definition}

\subsubsection{Reflex norm $r_x$ attached to fixed points in $X_p$} \label{reflexnormsection} 
\label{reflexnorm}
The above construction can be applied to the cocharacter $\mu_x$ of Section \ref{cocharacter}.
\begin{definition}
Let $x\in X_p$ be a toric fixed point with corresponding cocharacter $\mu_x$ and reflex field $E_x$. 
The \emph{associated reflex norm}
\[
r_x\colon \underline{E}_x^\times\rightarrow T_x
\]
is the \emph{abstract reflex norm} associated to the pair $(\mu_x, E_x)$.
\end{definition}

\begin{lemma}
\label{stabilizercommuteswithreflexnorm}
Let $x \in X_p$ be a toric fixed point and let $g \in G(\mathbb{Q})$.
Then $g$ stabilizes $x$ if and only if $g$ commutes with $\mu_x$.
In particular, if $g$ stabilizes $x$, then $g$ commutes with the image of the reflex norm $r_x$.   
\end{lemma}

\begin{proof}
The cocharacter $\mu_x$ defines an action of $\mathbb{G}_{m,E}$ on $V_E$ with distinct one-dimensional weight spaces spanned by $x$ and $\sigma(x)$.
So if $g$ commutes with $\mu_x$, then $g$ must preserve these weight spaces and so must stabilize $x$ and $\sigma(x)$.  

Conversely, by Galois conjugacy of $x$ and $\sigma(x)$ (see Lemma \ref{nonsplit}), $g \in G(\mathbb{Q})$ stabilizes $x$ if and only if it stabilizes $\sigma(x)$.
Since $g$ scales the lines spanned by $x$ and $\sigma(x)$ and preserves the orthogonal decomposition $H \oplus H^{\perp},$ where $H$ is the hyperbolic plane spanned by $x$ and $\sigma(x)$ it follows that $g$ commutes with $\mu_x$.  
If $g$ stabilizes $x$, and hence commutes with $\mu_x$ by the above, it follows directly that $g$ commutes with $r_x$.
\end{proof}

\subsubsection{Cocharacter lattices of maximal tori in orthogonal groups}
In the following we will give an explicit description of the reflex norm.
We first give a characterization of the cocharacter lattice $X_\ast(T)$ of a maximal $\Q$-rational torus $T\subseteq G$ in terms of the set $\mathcal{N}_T$ of isotropic eigenlines of $T$.
To this end, put
$$X_\ast^{\geom}(T) := \{ f\colon  \mathcal{N}_T \rightarrow \mathbb{Z} \ \vert\  f(x) + f(\sigma(x)) = 0 \ \forall x\in \mathcal{N}_T \}.$$
The absolute Galois group $\mathcal{G}_\Q$ of $\Q$ acts naturally on $\mathcal{N}_T$ and so it also does on $X_\ast^{\geom}(T)$.
Given $f\in X_\ast^{\geom}(T)$ and $t\in \C_p^\times$ we define the orthogonal operator
\[
\mu_f\colon V_{\C_p}\longrightarrow V_{\C_p}
\]
via $\mu_f(t)v=t^{f(x)}v$ if $v\in x$, $x\in\mathcal{N}_T$ and $\mu_f(t)v=v$ if $v$ is orthogonal to every $x\in \mathcal{N}_T$.
One easily checks that $\mu_f(t)\in T(\C_p)$ and thus $\mu_{f}$ defines a cocharacter of $T$.

\begin{lemma}
\label{cochargeom}
Let $T\subseteq G$ be a $\Q$-rational maximal torus.
The mapping
\begin{align*}
X_\ast^{\geom}(T) &\longrightarrow X_\ast(T) \\
f &\longmapsto \mu_f
\end{align*}
is a $\mathcal{G}_\Q$-equivariant isomorphism.
\end{lemma}

\begin{proof}
Via this homomorphism the function $f_x= \mathbf{1}_x - \mathbf{1}_{\sigma(x)}$ is mapped to $\mu_{x,T}$..
The set of miniscule cocharacters $\{\mu_x\ \vert\ x\in \mathcal{N}_T \}$ generates $X_\ast(T)$.
This proves surjectivity.
Bijectivity follows immediately since both modules are finitely generated free $\Z$-modules of the same rank.
Corollary \ref{galoissigma} implies that $\tau(f_x)=f_{\tau(x)}$.
Thus, for proving Galois equivariance it suffices to show that
\[
\tau \circ \mu_{x,T} \circ \tau^{-1} = \mu_{\tau(x),T}
\]
for all $\tau\in\mathcal{G}_\Q$, which follows along the same lines as Lemma \ref{galoischi}.
\end{proof}

Given a toric pair $(x,T)$ define $\mathcal{N}_x\subseteq \mathcal{N}_T$ to be the $\mathcal{G}_\Q$-orbit of $x$ and write $X_{\ast}^{\geom}(T_x)\subseteq X_\ast^{\geom}(T)$ for the subset of those functions, which are supported on $\mathcal{N}_x$.
By Lemma \ref{cochargeom} there is a canonical identification
\begin{align}
\label{geomident}
X_\ast(T_x)\xlongrightarrow{\cong} X_\ast^{\geom}(T_x).
\end{align}
Since $E_x$ is the field of definition of $x$, the map
\[
\mathcal{N}_x\longrightarrow \mathcal{G}_{\Q}/\mathcal{G}_{E_x},\quad \tau(x)\longmapsto \tau,
\]
is a well-defined isomorphism of $\mathcal{G}_{\Q}$-sets,
which induces the isomorphism of $\mathcal{G}_{\Q}$-modules
\begin{align}
\begin{split}
\label{eq:indisom}
\{f\colon \mathcal{N}_x \to \Z\} &\xlongrightarrow{\cong} \mathrm{Ind}_{\mathcal{G}_{E}}^{\mathcal{G}_{\Q}}\mathbb{Z} \\
f &\longmapsto [\tau \mapsto f(\tau(x))\cdot \tau^{-1}].
\end{split}
\end{align}
By \eqref{eq:inducedcochar} the right hand side of this homomorphism can be identified with the cocharacter lattice of $\underline{E}_x^\times$.
Thus, restricting \eqref{eq:indisom} to $X_\ast^{\geom}(T_x)$ yields the map
\[
X_\ast^{\geom}(T_x) \longrightarrow X_\ast(\underline{E}_x^\times),
\]
which is induced by the homomorphism of tori
\[
T_x  \xlongrightarrow{\varphi_x} \mathrm{U}(E_x,\sigma) \subseteq \underline{E}_x^\times.
\]

\subsubsection{Concrete identification of the reflex norm}
The \emph{concrete reflex norm} attached to a toric fixed point $x\in X_p$ is the homomorphism
\[
r_\sigma\colon \underline{E}_x^\times \longrightarrow \mathrm{U}(E_x,\sigma),\quad a \longmapsto a/a^{\sigma}
\]
of algebraic groups over $\Q$.
\begin{proposition}\label{pro:reflexexplicit}
For every toric fixed point $x\in X_p$
the equality
\[
r_\sigma=\varphi_x \circ r_x
\]
holds.
In particular, the kernel of $r_x$ is equal to the subgroup $\underline{F}_x^\times\subseteq \underline{E}_x^\times$.
\end{proposition}
\begin{proof}
For the first claim it is enough to prove that both homomorphism induce the same map on cocharacter groups.
Under the identification \eqref{geomident} the reflex norm induces the following map on cocharacter groups:
\[
\overline{r_x}\colon \mathrm{Ind}_{\mathcal{G}_{E_x}}^{\mathcal{G}_{\Q}}\mathbb{Z} \longrightarrow X_\ast^{\geom}(T),
\quad f \longmapsto \sum_{\tau \in \mathcal{G}_{E} \backslash \mathcal{G}_{\Q}} f(\tau)\cdot \tau^{-1}.(1_x-1_\sigma(x)).
\]
Thus, the composition $\varphi_x \circ r_x$ induces the map
\[
\mathrm{Ind}_{\mathcal{G}_{E_x}}^{\mathcal{G}_{\Q}}\mathbb{Z} \longrightarrow \mathrm{Ind}_{\mathcal{G}_{E_x}}^{\mathcal{G}_{\Q}}\mathbb{Z},
\quad f \longmapsto \sum_{\tau \in \mathcal{G}_{E} \backslash \mathcal{G}_{\Q}} f(\tau)\cdot (\tau - \sigma\tau),
\]
which is readily identified as the one induced by $r_\sigma$.

The second claim follows since $\ker(r_\sigma)=\underline{F}_x^\times$ and $\varphi_x\colon T_x \to \mathrm{U}(E_x,\sigma)$ is an isomorphism.
\end{proof}

Proposition \ref{pro:reflexexplicit} together with Corollary \ref{cor:spinornorm} immediately implies the following:
\begin{corollary}
Let $x \in X_p$ be a toric fixed point.
The formula
\[
\mathrm{sn}_\Q(r_x(e))= \mathrm{N}^{E_x}_{\Q}(e) \bmod (\Q^\times)^2
\]
holds for all $e\in E_x^\times$.
\end{corollary}

\subsubsection{Class group action on toric fixed points}
\label{sec:classgroup}
Denote by $\A$ (resp.~$\A^{p}_{f}$) the ring of adeles (resp.~the ring of adeles away from $p$ and $\infty$).
If $E$ is a finite extension of $E$, put $\A_E:=\A\otimes_\Q E$ and $\I_E:=\A_E^\times$.
For an open compact subgroup $U^{p}\subseteq G(\A^{p}_{f})$, an open subgroup $U_p\subseteq G(\Q_p)$ containing $G(\Q_p)^+$, and
put $U:=U^{p}U_p G(\R)^+$.
For a coset $h\in G(\A)/U$ consider the stabilizer $\Gamma_h:=G(\Q)\cap hUh^{-1}$ of $h$ in $G(\Q)$.
This is a $p$-arithmetic subgroup that lies in the kernel of the real spinor norm.
The set
\[
\Sigma_U :=\{\mbox{toric fixed points on}\ X_p\} \times G(\A) / U.
\]
of \emph{adelic toric fixed points (of level $U$)} is equipped with the diagonal $G(\Q)$-action.
For a toric fixed point $x$ let $x_U\subseteq \Sigma_U$ be the set of pairs with first coordinate equal to $x$ and write $\overline{x_U}:=\G(\Q)\backslash G(\Q)x_U$ for its image in $\G(\Q)\backslash \Sigma_U$.
Given a point $(x,h)\in \Sigma_U$ the reflex norm defines an action of the idele group $\I_{E_x}=\underline{E}_x^\times(\A)$ on $x_U$:
\begin{align*}  
t \star (x,h) &:= (r_x(t_p)x, r_x(t) \cdot h) \\
&\hphantom{:}= (x, r_x(t) \cdot h). 
\end{align*}

\begin{lemma}
\label{descendedclassgroupaction}
Let $x\in X_p$ be a toric fixed point.
The action of $\I_{E_x}$ on $x_{U}$ descends to a well-defined action on $\overline{x_U}$.
\end{lemma}

\begin{proof}
Suppose $ (x,h)$ and $ (x,h')$ are $G(\mathbb{Q})$-equivalent points of $x_U$, i.e., $ g.(x,h) = (x,h')$ for some $g \in G(\mathbb{Q})$.
In particular, $gx = x$ holds, which by Lemma \ref{stabilizercommuteswithreflexnorm} implies that $g$ commutes with the image of the reflex norm.
It follows that
\begin{equation*}
t \star (x,h')=t \star \left(g.(x,h)  \right) = g.\left( t \star (x,h)  \right),
\end{equation*}
which proves the claim.
\end{proof}

To an adelic toric fixed point $(x,h)\in\Sigma_U$ we attach the relative class group 
\begin{align}\label{def:classgroup}
C_{x,h} := E_x^\times \I_{F_x}\backslash \I_{E_x}/ r_x^{-1}(h U h^{-1}).
\end{align}
\begin{corollary}\label{cor:classgroup}
Let $(x,h)$ be an adelic toric fixed point.
The action of $\I_{E_x}$ on the orbit of $(x,h)$ in $\overline{x}_U$ factors through $C_{x,h}$.
\end{corollary}
\begin{proof}
By Proposition \ref{pro:reflexexplicit} the subgroup $\I_{F_x}$ is identified with $\ker(r_x)(\A)$, which acts trivially on $x_U$ and, hence, also on $\overline{x}_U$.
Since $r_x(E_x^\times)$ is a subgroup of $G(\mathbb{Q})$, the subgroup $E_x^\times$ acts trivially on $\overline{x}_U$.
Finally $r_x^{-1}(h U h^{-1})$ acts trivially on $h$ by definition.
The result follows.
\end{proof}

\subsubsection{Class group action on oriented special points}
\begin{definition}
An \emph{adelic oriented special point $(\vec{x},h)$ of level $U$} a coset $h\in G(\A)/U$ and an oriented special point $\vec{x}$ of level $\Gamma_h$.
The set of oriented adelic special points of level $U$ will be denoted by $\Sigma_U^{\mathrm{or}}$.
\end{definition}

The set $\Sigma_U^{\mathrm{or}}$ carries a natural $G(\Q)$-action:
let $(\vec{x},h)$ be an adelic oriented special point given by the tuple $(x,T,h,o)$ and $g\in G(\Q)$.
The adelic oriented special pair $g(\vec{x},h))$ is given by the tuple $(gx,g T g^{-1}, gh, c_g(o))$, where as before
\[
c_g\colon H_s(\Gamma_{h,T})\rightarrow H_s(\Gamma_{gh,gTg^{-1}})
\]
denotes the homomorphism induced by conjugation by $g$.
Consider the subset $x_U^{\mathrm{or}}\subseteq \Sigma_U^{\mathrm{or}}$ of all adelic oriented special points with first coordinate equal to $x$.
As in Section \ref{sec:classgroup}, the reflex norm induces an action of $\I_{E_x}$ on $x_U^{\mathrm{or}}$:
given $t\in \I_{E_x}$ and $(x,T,h,o)\in x_U^{\mathrm{or}}$ then $t\star(\vec{x},h)$ is associated to the special pair $(x,T)$, the coset $r_x(t)h$ and the orientation $c_t(o)$ where
\[
c_t\colon H_s(\Gamma_{h,T})\rightarrow H_s(\Gamma_{r_x(t)h,T})
\]
is the homomorphism induced by conjugation with $r_x(t)$.
With similar arguments as in Lemma \ref{descendedclassgroupaction} and Corollary \ref{cor:classgroup} one proves the following:
\begin{proposition}
Let $(\vec{x},h)$ be an adelic oriented special point.
The action of $\I_{E_x}$ on $x_U$ descend to an action on $\overline{x}_U^{\mathrm{or}}:=G(\Q)\backslash G(\Q)x_U^{\mathrm{or}}$.
Moreover, the action of $\I_{E_x}$ on the orbit of the image of $(\vec{x},h)$ in $\overline{x}_U^{\mathrm{or}}$ factors through the relative class group $C_{x,h}$.
\end{proposition}

\subsubsection{Connected components}
\begin{definition}
Two cosets $h,h'\in G(\A)/U$ lie in the same \emph{compononent} if their images in $G(\Q)\backslash G(\A)/U$ agree.
\end{definition}
Assume that $h,h'\in G(\A)/U$ lie in the same component and chose $g\in G(\Q)$ with $h =g h'$.
The element $g$ is unique up to left multiplication by elements of $\Gamma_h$.
Let $J$ be an element of $H^{s}(\Gamma_h,\rM^\times)$ and $(\vec{x},h')$ an adelic oriented special point.
Then $g.\vec{x}$ is a special point of level $\Gamma_h$.
We say that $J$ is regular at $(\vec{x},h')$ if it is regular at $g.\vec{x}$.
This regularity condition is independent of the choice of $g$.
Moreover, Proposition \ref{invariantevaluation} implies that the \emph{special value}
\[
J[(\vec{x},h')]:=J(g.\vec{x})
\] 
is independent of the choice of $g$ as well.

The product of the local spinor norm maps defined a group homomorphism
\begin{align}\label{adelicspinornorm}
\mathrm{sn}_\A\colon G(\A)\longrightarrow \I_\Q^2\backslash \I_\Q.
\end{align}
Since the spin group of a non-degnerate quadratic space is semi-simple and $G(\Q_p)$ is not compact, strong approximation implies that \eqref{adelicspinornorm} induces a bijection
\[
G(\Q)\backslash G(\A)/U=\Q^\times \I_\Q^2\backslash \I_\Q/\mathrm{sn}_\A(U)=:C_{U}.
\]
Let $E_U/\Q$ be the abelian extension associated to $C_{x,h}$ by class field theory.
The description of the Galois action on connected components of orthogonal Shimura varieties in \cite[Section 1]{kudla-cycles} suggests that the ``connected components'' of the space $G(\Q)\backslash (X_\infty \times X_p \times G(\A)/U)$ are ``defined over $E_U$''.

Let $x\in X_p$ be a toric fixed point. The diagram
\begin{center}
 \begin{tikzpicture}
    \path
		(0,0) node[name=A]{$\I_{E_x}$}
		(2,0) node[name=B]{$G(\A)$}
		(0,-1.5) node[name=C]{$\I_\Q$}
		(2,-1.5) node[name=D]{$\I_\Q^2\backslash \I_\Q$};
		\draw[->] (A) -- (B) node[midway, above]{$r_x$};
		\draw[->] (A) -- (C) node[midway, left]{$\mathrm{N}^{E_x}_{\Q}$};
		\draw[->] (B) -- (D) node[midway, right]{$\mathrm{sn}_{\A}$};
		\draw[->] (C) -- (D);
  \end{tikzpicture} 
\end{center}
is commutative by Corollary \ref{cor:spinornorm} and Proposition \ref{pro:reflexexplicit}.
Thus, for every $h\in G(\A)/U$ the norm map $\mathrm{N}^{E_x}_{\Q}\colon \I_E \rightarrow \I_\Q$ descends to a homomorphism $C_{x,h}\rightarrow C_{U}$. We denote its kernel by $C_{x,h}^{U}\subseteq C_{x,h}$.
The commutativity of the diagram above implies that $h$ and $r_x(t)h$ lie in the same component for all $t\in C_{x,h}^{U}$.
In particular, if $(\vec{x},h)$ is an adelic oriented special point and $t\in C_{x,h}^{U}$, the special value $J[t\star(\vec{x},h)]$ is defined for all $J\in H^{s}(\Gamma_h,\rM^\times)$.

\subsection{The conjectural reciprocity law}
\label{sec:conjecture}
We formulate a reciprocity law regarding the algebraicity and field of definition of the values of rigid meromorphic cocycles at oriented special pairs, akin to the classical Shimura reciprocity law for singular moduli on Shimura varieties.
Throughout this section we fix an open subgroup $U^p \subseteq G(\A^{p,\infty})$ and an open subgroup $U_p\subseteq G(\Q_p)$ containing $G(\Q_p)^+$ and put $U=U^p U_p G(\R)^+$.
As before, for a coset $h\in G(\A)/U$ put $\Gamma_h:=G(\Q)\cap h U h^{-1}$.
Let $(\vec{x},h)$ be an oriented adelic special point and $C_{x,h}$ the by \eqref{def:classgroup} associated relative class group.
Denote by $H_{x,h} / E_x$ the abelian extension associated to $C_{x,h}$ by class field theory, that, is the Artin reciprocity map furnishes a canonical isomorphism
\[
C_{x,h} \xlongrightarrow{\rec} \Gal(H_{x,h} / E_x).
\]
By functoriality of the Artin reciprocity map the subgroup $C_{x,h}^{U}$ corresponds to the group of those field automorphisms that act trivially on $H_{x,h}\cap E_U$.
Regarding rationality properties of adelic rigid meromorphic cocycles and their values at adelic special pairs, we are guided by the following two principles:
\begin{itemize}
\item
The set $G(\Q)\backslash\Sigma_U^{\mathrm{or}}$ behaves like the collection of special points on a Shimura variety with its attendant canonical model.
The Shimura reciprocity law for classical Shimura varieties then makes the following plausible: 
The image of the adelic special pair $(\vec{x},h)$ in $G(\Q)\backslash\Sigma_U^{\mathrm{or}}$ is defined over the abelian extension $H_{x,h} / E_x$ and $\tau \in \Gal(H_{x,h} / E_x)$ acts via $\rec^{-1}(\tau)$ on it.
\item
Rigid meromorphic cocycles $J$ behave as though they are meromorphic functions with divisors defined over $E_U$ on a ``Shimura variety" $Y_U$ whose collection of special points coincides with $G(\Q)\backslash\Sigma_U^{\mathrm{or}}$.     
\end{itemize}
These principles lead to the following analogue of the Shimura reciprocity law:
\begin{conjecture}
\label{conj:main}
Let $h\in G(\A)/U$ be a coset.
For every rigid meromorphic cocycle $J\in \mathcal{RMC}(\Gamma_h)$ of level $\Gamma_h$
there exists a finitely generated submodule $\Pi_J\subseteq \Q_p^\times$ of rank less or equal to the rank of $H^s(\Gamma_h,\Z)$
such that
\[
J[(\vec{x},h)]\in H_{x,h} E_U \Pi_J  
\]
for every oriented special pair $(\vec{x},h)$ at which $J$ is regular.
Moreover, for every $\alpha\in \Pi_J$ such that $J[(\vec{x},h)]/\alpha \in H_{x,h} E_U$ there exists $\alpha^{\prime}\in \Pi_J$ such that
\[
\rec(t)\left(\frac{J[(\vec{x},h)]}{\alpha}\right)=\frac{J[t\star (\vec{x},h)]}{\alpha^{\prime}}\quad \forall t\in C_{x,h}^{U}.
\]
\end{conjecture}


\section{Examples and numerical experiments}
\label{sec:examples}

This closing chapter reports on a few numerical experiments illustrating Conjecture \ref{conj:main}.
These experiments only scratch the surface and it would be worthwhile to extend their scope by developing
systematic, practical approaches for computing rigid meromorphic cocycles and their special values.

Note that the case $s=0$ of definite quadratic spaces represents the most tractable  setting for  Conjecture \ref{conj:main}. Because $G(\R)$  is compact, the $p$-arithmetic group $\Gamma$ acts discretely on $X_p$, and a 
rigid meromorphic cocycle is just a $\Gamma$-invariant rigid meromorphic function on $X_p$ whose divisor is supported on a finite union of $\Gamma$-orbits of rational quadratic divisors.
In ranks $3$ and $4$, the theory of $p$-adic uniformization of
Shimura curves and quaternionic Shimura surfaces places Conjecture \ref{conj:main}
within the purview of the classical theory of complex multiplication by reducing it to the study of CM points on certain orthogonal Shimura varieties. (See the forthcoming work \cite{GeBa} for a detailed discussion.)
This setting has already been explored from a practical,  computational angle in the 
 literature, notably in
\cite{greenberg-matt}, \cite{negrini1}, and \cite{giampietro}
for the special case of signature $(3,0)$, and the upcoming work \cite{negrini2} of Negrini describes similar calculations in signature $(4,0)$.
The remainder of this chapter will therefore focus on the hyperbolic case where $s=1$ and rigid meromorphic cocycles are one-cocycles on $p$-arithmetic groups.

\subsection{Signature $(2,1)$ and class fields of real quadratic fields}
\label{sec:example-21}
A prototypical example of a 
three-dimensional quadratic space  can be 
 obtained by letting $B$ be a quaternion algebra over $\Q$,  and setting
\begin{equation}
\label{excisom1-Q}
V :=  B^{{\rm Tr}=0},  \ \ \ q(v) =  \pm {\rm Norm}(v),
\end{equation}
where ${\rm Norm}$ denotes the reduced norm.
For example, when $V=\mathrm{M}_2(\Q)$, equipped with the negative of the trace form, 
this quadratic space is the direct sum of a hyperbolic plane and the form $ x^2$, and is of
signature $(2,1)$.
More generally, $V$ is of signature $(2,1)$ whenever $\mathrm{M}_2(\Q)$ is replaced by an indefinite quaternion algebra
over $\Q$ in this construction.

The exact sequence \eqref{eqn:spin-gen}
with $k=\Q$ then becomes
\begin{equation}
\label{eqn:spin-3}
1 \rightarrow \{\pm 1\} \rightarrow B_1^\times \rightarrow
  B^\times/\Q^\times   \rightarrow  \Q^\times/ (\Q^\times)^2,
\end{equation}
where the penultimate map is induced from the natural inclusion of $B_1^\times$ into $B^\times$, and the rightmost map is induced from the norm modulo squares.
 In particular,  one then has
 $$ G(\R) = \PSL_2(\R).$$
Rigid meromorphic cocycles attached to quadratic spaces of
signature $(2,1)$ have already been considered in the literature, notably in 
\cite{DV1}  where $V $ is the space 
of trace zero elements in $\mathrm{M}_2(\Z[1/p])$ and 
$\Gamma = \SL_2(\Z[1/p])$, and in 
\cite{GMX} and \cite{Ge-quaternionic}, where more general quadratic spaces arising from the trace zero elements in certain indefinite quaternion algebras 
are considered. The constructions of \textit{loc.cit.~}are in one sense finer than what will be described in this section
 because they exploit the circumstance, peculiar to rank $3$, that the rational quadratic divisor  $\Delta_{v,p}$ on $\cH_p$  attached to a vector $v\in B$  is of the form
$$\Delta_{v,p} = (\tau_v) + (\tau_v'),$$
where $\tau_v$ and $\tau_v'$ are defined over a quadratic 
extension of $\Q_p$ and interchanged by the Galois automorphism.
With some care, it is possible to construct divisor valued cohomology classes whose divisor involves $\tau_v$ {\em but not $\tau_v'$}.
Furthermore,  the collection of $\Gamma$-orbits of  vectors of length $d$ is   more rich and subtle in the rank $3$ settings, since it is a principal homogeneous space for the  narrow class group  of
(an order in) $\Q(\sqrt{d})$.
A rigid meromorphic cocycle with divisor concentrated  on a single orbit is in some sense ``defined over" the associated narrow ring class field and does not arise in general as a Borcherds lift.
 A similar phenomenon occurs in classical Borcherds theory, where not every  meromorphic function  on a modular or Shimura curve  with CM divisor  can be realised as a Borcherds lift, although such converse results have been proved in certain cases
  for quadratic spaces of higher rank
 (cf.~\cite{bruinier-converse} for example). 

 Let $V$ be the space of matrices of the form 
 $$ V = \left\{ \left(\begin{array}{rr} -b & -c \\ a & b \end{array}\right) \mbox{ with } a,b,c \in \Q \right\},$$
 endowed with the negative of the trace form, for which  $q(v) = -\det(v)$. The matrix 
\begin{equation}
\label{eqn:matrix-in-V}
v =  \left(\begin{array}{rr} -b & -c \\ a & b \end{array}\right) \in V 
\end{equation}  
corresponds to the quadratic form $[a,2b,c]$ of discriminant $4(b^2-ac)$,whose roots,
 $$ \frac{-b-\sqrt{b^2-ac}}{a}, \qquad  \frac{-b+\sqrt{b^2-ac}}{a}, $$
  have opposite signs precisely
 when $ac<0$. 
 Let $\Phi$ be the Schwartz function on the finite ad\`elic 
 space $V_{{\mathbb A}_f}$ 
 whose component  $\Phi_\ell$ at odd primes $\ell$ 
 is the characteristic function of $M_2(\Z_\ell)_0$, 
 and whose component at $2$ is defined in terms of the
  odd Dirichlet character $\chi_4$ of conductor $4$ by setting
 $$ \Phi_2 \left(\begin{array}{rr} -b & -c \\ a & b \end{array}\right) = \left\{ \begin{array}{cl} 
 \chi_4(a)  & \mbox{ if } a,b,   c  \in \Z_2, \\
 0 & \mbox{ otherwise.}  \end{array}\right.$$
 Let $\Phi^{(p)}$ be the Schwartz function on the ad\`elic space
 $V_{{\mathbb A}_f^{(p)}}$
  with the component at $p$ ignored.
 The stabiliser of this Schwarz function
   is the $p$-arithmetic analogue of the Hecke congruence group $\Gamma_0(2)$:
 $$ \Gamma = \left\{  \left(\begin{array}{rr} a & b \\ c & d \end{array}\right)  \in \SL_2(\Z[1/p])  \mbox{ with } 2|c
 \right\}.
$$
  
For each $d>0$, let 
$$ \Delta_d(0,\infty) := \sum_{ \substack{\langle v,v\rangle = d \\ ac<0} } \delta_{v} \Phi^{(p)}(v)   \Delta_v,
$$
be the  locally finite divisor on $\cH_p$ attached to $\Phi^{(p)}$ and $d$, 
where the sum is taken over the matrices
$v$ as in \eqref{eqn:matrix-in-V}
of length $d = b^2-ac$ satisfying $ac<0$, and $\delta_v = {\rm sign}(a)$.
 The degree of   $\Delta_d(0,\infty) \cap \cH_p^{\le n}$  is equal 
  to the $d p^{2n}$-th Fourier coefficient  $b(dp^{2n})$ of the modular form
 of  weight 
  $3/2$  and level $4$ given by
  $$ g(q) = \theta(q)  E_1(1,\chi_4)(q) = \sum  b(n) q^n, $$
  where $\theta(q) = \sum_{n\in \Z} q^{n^2}$ is the usual unary theta series and $E_1(1,\chi_4)$ is the weight one Eisenstein series attached to $\chi_4$. Some of the coefficients of this generating series are listed in the table below:
  $$
  \begin{array}{r|rrrrrr}
  d & 2 & 5 & 8 & 11 & 14 & 17 \\ \hline
  b(d) & 3 & 6 & 3 & 6 & 12 & 12 
  \end{array} $$
When $p=3$, the relevant space of modular forms of  weight $3/2$  and level $12$ is spanned by $g(q)$ and  
$g(q^3)$. It follows that the locally finite divisors
\begin{eqnarray*}
 \Delta_5(0,\infty) - 2 \Delta_2(0,\infty), &&  \qquad
\Delta_8(0,\infty) - \Delta_2(0,\infty), \\
 \Delta_{11}(0,\infty) - \Delta_5(0,\infty), && \qquad \Delta_{14}(0,\infty) - \Delta_{17} (0,\infty)
 \end{eqnarray*}
are  the divisors of a  rigid meromorphic period function  in the sense of \cite{DV1} (but on 
$\Gamma$ rather than
on $\SL_2(\Z[1/p])$).
Let 
$$J_1, J_2, J_3, J_4  \in H^1(\Gamma,\rM^\times)$$
 denote the rigid meromorphic cocycles with these  divisors.
They were calculated with $200$ digits of $3$-adic accuracy and evaluated at the RM point 
$\tau_{11} = \frac{1+\sqrt{11}}{2}$ of discriminant $44$.
Since the narrow Hilbert class field of $\Q(\sqrt{11})$ is
the biquadratic field $\Q(i, \sqrt{11})$, Conjecture
\ref{conj:main}
asserts that $J_t[\tau_{11}]$ is defined over this field, for $t=1,2,3,4$. Indeed, one finds that, at least to within this accuracy,
{\small
\begin{eqnarray*}
 J_1[\tau_{11}]  &\stackrel{?}{=} & \frac{992481+ 322880\cdot i}{17 \cdot 29^2 \cdot 73}, \\
 J_2[\tau_{11}]  &\stackrel{?}{=} &\frac{ -19245079  -2983200 \cdot i}{29\cdot 61\cdot 101\cdot109}, \\
 J_3[\tau_{11}]  &\stackrel{?}{=} & \frac{19229239383465  + 7867810272448 \cdot i}{13^4\cdot 17^2 \cdot 29^2\cdot  41\cdot  73}, \\
 J_4[\tau_{11}]  &\stackrel{?}{=}& \frac{
 4967915642907602238905 -13926798659822783142912\cdot i}{17^ 3\cdot 29\cdot 41\cdot 73\cdot 109\cdot 149\cdot  193\cdot 197 \cdot 233 \cdot 241}.
\end{eqnarray*}
}
These Gaussian integers are all of norm one and the primes 
that divide their denominators are all inert  in $\Q(\sqrt{11})$. 
This experiment extends the framework of 
\cite{DV1} ever so slightly by working with a lattice that differs 
from the space of all integral binary quadratic forms.
It also illustrates a feature of  Conjecture  \ref{conj:main},
 namely that the field of definition of $J[\tau]$ 
is a class field that depends only on $\tau$ and not on
$J$, insofar as the rigid meromorphic cocycles constructed in the orthogonal group framework are all ``defined over $\Q$".

\subsection{Signature $(3,1)$ and Bianchi cocycles}
\label{sec:example-31}

Let $V  :=  \Q^4$ be the standard  rational
 Minkowski space,  endowed with the quadratic 
form 
$$ q(x,y,z,t) := x^2+y^2+z^2 -t^2.$$
It   admits a convenient description as a    subspace of $\mathrm{M}_2(K)$, with $K:= \Q(i)$. The matrix  ring $\mathrm{M}_2(K)$
 is endowed with the standard 
anti-involution $M\mapsto M^*$ as well as an 
involution $M \mapsto \overline{M}$ coming from the Galois 
automorphism of $K$, defined by
$$ \mat{a}{b}{c}{d}^{\dagger} := \mat{d}{-b}{-c}{a}, \qquad \overline{\mat{a}{b}{c}{d}} := 
\mat{\overline a}{\overline  b}{\overline c}{\overline d},$$
and we define
\begin{eqnarray*}
  V &=&  \{M\in \mathrm{M}_2(K) \mbox{ satisfying } M^{\dagger} =-\overline{M} \}
   \\
  &=&  \left\{ \left(\begin{array}{rr} 
\overline{\alpha}  & -b \\ c & -\alpha  \end{array}\right) \mbox{ with } \alpha\in K, \quad b,c\in \Q \right\}, 
\end{eqnarray*}
equipped with the quadratic form 
$$q(M) = -\det(M).$$
We will adopt the convenient notational shorthand
$$ [\alpha; b,c] := \left(\begin{array}{cc} 
\alpha  & -b \\ c & -\overline{\alpha} \end{array}\right), \qquad q([\alpha;b,c]) = \alpha\overline{\alpha} -bc.
$$ 
The group $\SL_2(K)$ operates as isometries of $V$ by twisted conjugation,
$$ \gamma \dagger M := \gamma \cdot M \cdot \overline{\gamma}^{-1} =
\gamma \cdot M \cdot \overline{\gamma}^{\dagger},$$
leading to a homomorphism 
$$\SL_2(K) \longrightarrow  \SO_{V}(\Q).$$
The real symmetric space $X_\infty$ is a real $3$-dimensional manifold, and is  identified with $\C \times \R^{>0}$ by the rule sending the line spanned by the vector $[\beta;u,v]$
of negative norm to 
$$ [\beta; u,v] \mapsto (z,t) = (\beta/v, (uv-\beta\overline{\beta})/v^2).$$
The $p$-adic symmetric space 
 $X_p$ is identified with $\cH_p \times \cH_p$  via the association
\[
\begin{pmatrix}
 (\tau_1,\tau_2) & -1 \\
 \tau_1 \tau_2 & -(\tau_2,\tau_1) \end{pmatrix}\longleftrightarrow(\tau_1, \tau_2).
\]
For any $M=[\alpha; b,c]\in V$ of positive norm, a function $F_M\in\rM^\times$ having $\Delta_{M,p}$ as divisor can be defined by setting
$$
F_M(\tau_1,\tau_2) = (-1,\tau_1) \left(\begin{array}{cc}
\overline{\alpha} & -b \\ c & -\alpha \end{array}\right) \left(\begin{array}{c} \tau_2 \\ 1\end{array}\right) = c\tau_1 \tau_2 -\alpha \tau_1 -\overline{\alpha}\tau_2 + b.
$$

Consider the standard lattice
\begin{eqnarray*}
 \latt &=&  \{M\in \mathrm{M}_2(\Z[i][1/p]) \mbox{ satisfying } M^{\dagger} =-\overline{M} \},
\end{eqnarray*}
on which the group $\Gamma = \SL_2(\Z[i][1/p])$ acts naturally by isometries, as above.
Concrete instances of rigid meromorphic cocycles for the Bianchi group $\Gamma$
can be described using modular symbols, as we now explain.
The group $\Gamma$ acts naturally on the boundary
$\C \cup \{\infty\} = \mathbb{P}_1(\C)$ of $X_\infty$ by M\"obius transformations, and preserves the subset 
$\mathbb{P}_1(K)$ of $K$-rational points, whose stabilisers in $\Gamma$ are parabolic subgroups. Given $r,s\in \mathbb{P}_1(K)$,
denote by $(r,s)$ the open hyperbolic geodesic path joining 
$r$ and $s$ in $X_\infty$. 

Given an integer $d>0$, one can {\em formally}
define
 a modular symbol with values in  $\rM^\times$
 by setting, for all $r,s\in \mathbb{P}_1(K)$,
 $$  J_d(r,s)(\tau_1,\tau_2) := 
 \prod_{\substack{ M \in \latt \\ 
 \langle M,M\rangle = d }}  F_M(\tau_1,\tau_2)^{(\Delta_{M,\infty}\cdot (r,s))_\infty},$$
 where $F_M\in \rM^\times$ is the rational function on 
 $X_p$ having $\Delta_{M,p}$ as divisor,
 and $\Delta_{M,\infty}$ is the real two-dimensional cycle in the three-manifold $X_\infty$  attached to 
 the vector $M$.
If these expressions converge for all $r,s$, then they clearly define a modular symbol with values in 
$\rM^\times$. One also hopes that they satisfy
some type of $\Gamma$-equivariance property, since 
$$ F_{\gamma M}(\gamma\tau_1, \gamma \tau_2) = 
j(\gamma;\tau_1,\tau_2) F_M(\tau_1, \tau_2),$$
where 
$$j(\gamma; \tau_1,\tau_2) = (c_1 \tau_1+d)(c_2\tau_2+d_2), \qquad \gamma = \mat{a}{b}{c}{d},$$
is a multiplicative factor of automorphy  that   depends on $\gamma$ and $\tau$ but not on $M$.

Assuming that the assignment $(r,s) \mapsto J_d(r,s)$ is
$\Gamma$-equivariant, the theory of modular symbols
allows us to reduce the calculation of $J_d$ to that of  $J_d(0,\infty)$.
Following a terminology of Zagier that is taken up in 
\cite{DV1}, the  function
$$ \Phi_d  := J_d(0,\infty)$$
is called the 
{\em rigid meromorphic period function} on $X_p$
associated to $J_d$ (or to the integer $d$).
The formula defining this  rigid meromorphic period function is given by
 $$ \Phi_{d}(\tau_1,\tau_2) = \prod_{j=0}^\infty \Phi_{d,j}(\tau_1,\tau_2), $$
where 
\begin{equation}
\label{eqn:Phidj}
\Phi_{d,j}(\tau_1,\tau_2) = \prod_{\substack{ M\in V_\Z, \\\langle M, M\rangle = dp^{2j} }} F_M(\tau_1,\tau_2)^{(\Delta_{M,\infty} \cdot (0,\infty))}. 
\end{equation}
To make it more concrete, we invoke the following lemma

\begin{lemma} 
Let 
$M = [\alpha; b,c]\in V$ be a 
vector of positive norm with associated real two-dimensional cycle $\Delta_{M,\infty} \subseteq X_\infty$.
Then
$$
\Delta_{M,\infty} \cdot  
(0,\infty)  = \left\{
\begin{array}{cl}
{\rm sign}(c) & \mbox{ if } bc<0; \\
 0 & \mbox{ if } bc>0.
 \end{array}
 \right.
$$
\end{lemma}
 \begin{proof}
 The two-dimensional region attached to $M = [\alpha,b,c]$ is the set 
 of $(z,t) = (\beta/v, (uv-\beta\overline{\beta})/v^2) \in \C\times \R^{>0}$ satisfying 
 the equality
 $$ {\rm Tr}\left(\begin{pmatrix}\alpha & -b \\ c & -\overline{\alpha} \end{pmatrix}
 \cdot
 \begin{pmatrix} \beta & -u \\ v & -\overline{\beta} \end{pmatrix} \right) =0,$$
 that is,
 $$ \alpha\beta + \overline {\alpha} \overline{\beta} - bv -cu = 0.$$
 This equation can be re-written in terms of the coordinates $(z,t)$, as
 $$ \alpha z + \overline{\alpha} \overline{z} - b - c (t+z\overline{z}) =0,$$
 or equivalently, after dividing by $-c$,
 $$ (z\overline{z}+t) - \frac{\alpha}{c} z - \frac{\overline{\alpha}}{c} \overline{z} + \frac{b}{c} = 0.$$
 This equation can be further re-written as
 $$ \left(z-\frac{\overline{\alpha}}{c}\right)\left(\overline{z}-\frac{\alpha}{c}\right)    = \frac{\alpha \overline{\alpha} -bc}{c^2}  - t.$$
  This two-dimensional region intersects the boundary of $ \C \times \R^{>0}$ (with equation $t=0$  in the  circle on the complex $z$-plane  centred at $\overline{\alpha}/c$ and
  of square-radius $(\alpha\overline{\alpha}-bc)/c^2$.  
Since this circle contains $z=0$ in its interior precisely when $bc<0$, the lemma follows.
    \end{proof}

Thanks to this lemma, equation 
\eqref{eqn:Phidj}
 can be  written as
 $$  \Phi_{d,j}(\tau_1,\tau_2) = \prod_{\substack{\alpha\overline{\alpha}-bc = d p^{2j}\\ bc<0, }} (c\tau_1 \tau_2 - \alpha\tau_1 -\overline{\alpha}\tau_2 +b)^{{\rm sgn}(c)}. $$
There are two difficulties that arise  in this definition:
\begin{enumerate}
\item The term ${\rm sgn}(c)$ is problematic to make sense of  when $c=0$, 
which occurs when the two-cycle in $X_\infty$ attached to $[\alpha;b,c]$
intersects $(0,\infty)$ improperly.
\item Even if such improper intersections do not occur,
one has $\Phi_{d,j}(\tau_1,\tau_2) = \pm 1$
for all $j$, because  the symmetry $[\alpha; b,c]\mapsto -[\alpha;b,c]$ forces trivial cancellation in the product.
\end{enumerate}
To remedy the first problem, it can simply be  assumed  that $d$ is not a norm from $\Z[i]$, i.e., is not a sum of two squares,
to ensure that there are no  length $d$ vectors $[\alpha;b,c]$ in the lattice with $bc=0$.
This holds, for instance
if $d$  is  exactly divisible by a  prime which is congruent to $3$ modulo $4$.

To address the second problem,  
the sum over a full lattice needs to be replaced by a sum over a coset of one lattice in an other (or a linear combination of such sums) that is not preserved by 
multiplication by $-1$. For instance, 
let  $\chi_4$ be the odd Dirichlet character of conductor $4$ and weight the factor  attached to $[\alpha;b,c]$ in the product defining $\Phi_{d,j}$  by a further 
exponent of $\chi_4(c)$,  by setting 
 \begin{equation}
 \label{eqn:Phidj2}
\Phi_{d,j}(\tau_1,\tau_2) := \prod_{\substack{\alpha\overline{\alpha}-bc = d p^{2j}\\ bc<0, }} (c\tau_1 \tau_2 - \alpha\tau_1 -\overline{\alpha}\tau_2 +b)^{\chi_4(c){\rm sgn}(c)}. 
\end{equation}
This further weighting by an odd character causes the factors attached to $[\alpha; b,c]$ and $[-\alpha; -b, -c]$ to reinforce each other rather than cancelling out.  The {\em weight} of
$\Phi_{d,j}$ is defined to be the sum of the exponents appearing in the product expansion   \eqref{eqn:Phidj2}.
The following suggestive   formula  relates this weight
 to the Fourier coefficients  of modular forms. 
\begin{lemma}
\label{lemma:degphidj}
For all $j\geq 1$ and all  $d$ which are not a sum of two squares,  the rational function $\Phi_{d,j}$ has weight equal to
$$ \deg \Phi_{d,j} =  16 \sigma(d p^{2j}), \mbox{ where } \sigma(n) = \sum_{4\nmid d |n} d.$$
\end{lemma}
\begin{proof}
By \eqref{eqn:Phidj2}, 
 \begin{align*}
{\rm wt} \Phi_{d,j} &= 
 \sum_{\substack{\alpha\overline{\alpha}-bc = d p^{2j}\\ bc<0, }} \!\!\! \chi_4(c)\cdot {\rm sgn}(c)  
   \ \ =  \ \ \ \ 2\!\!\!\!\! \!\!\!\!\!\sum_{\substack{r^2+s^2 + tu = d p^{2j}\\ t,u>0  }}  \!\!\!\chi_4(t)  \\
  &= \ \  2 \!\!\! \sum_{\substack{r,s \in \Z\\
  r^2+s^2 < dp^{2j} }}  \left(
  \sum_{t | dp^{2j}- (r^2+s^2)}    \chi_4(t) \right).
  \end{align*}
  The inner sum is readily recognised as the 
  $dp^{2j}-(r^2+s^2)$-th coefficient of the weight one Eisenstein series attached to $\chi_4$:
  \begin{align*}
  E_1(1,\chi_4)(q) &=  \  \frac{1}{4} + \sum_{n=1}^\infty \left(\sum_{t|n}\chi_4(t) \right) q^n \\
  &= \sum_{n=0}^\infty a(n) q^n, 
  \end{align*}
  and hence we can write
  \begin{align*}
{\rm wt} \Phi_{d,j} &= 2 \sum_{\substack{r,s \in \Z \\ r^2+s^2 < dp^{2j} }} a(dp^{2j}-(r^2+s^2))  \\
 &= 2 \sum_{m\leq dp^{2j}} r_2(m) a(dp^{2j}-m),
 \end{align*}
 where $r_2(m)$ is the number of ways of representing $m$ as a sum of two squares,
 which is also the $m$-th Fourier coefficient
 of the binary theta series attached to the quadratic form
 $x^2+y^2$. This theta series is also 
 equal to a multiple of $E_1(1,\chi_4)$, a very simple instance of the Siegel-Weil formula, and thus one finds
   \begin{align*}
{\rm wt} \Phi_{d,j} &=   8 \sum_{m\leq dp^{2j}} a(m) a(dp^{2j}-m) \\ 
&= \ \ 8  b(dp^{2j}),
\end{align*}
where $b(n)$ is the $n$-th Fourier coefficient  of
$ 2 E_1(1,\chi_4)^2$. This modular form of weight two,
level $4$  and trivial nebentypus character is the quaternary theta series whose Fourier coefficients give the number of ways of representing an integer as a sum of $4$ squares. It is also the Eisenstein series of weight two and level $4$ with 
$q$-expansion given by 
$$ E_2(q) - 4 E_2(q^4)  :=
\frac{1}{8} + \sum_{n=1}^\infty \left(\sum_{4\nmid t|n} t\right) q^n.$$
 Hence, the equation
\begin{equation}
\label{eqn:coeffsEis24}
  \deg\Phi_{d,j} = 8 b(n) = 8  \sum_{4\nmid t|n} t, 
 \end{equation}
 holds, as was to be shown.
  \end{proof}
 
 Since $\sigma(3) =4 $ and $\sigma(7) = 8$, 
 we may let $\sD := 2 \Delta_3 -\Delta_7$ and conclude that
 the rational functions  
 $$ \Phi_{\sD,j}(\tau_1,\tau_2) :=  \frac{\Phi_{3,j}^2}{\Phi_{7,j}}$$
  are of weight $0$ for all $j$.
  This weight zero condition is not sufficient in and of itself to ensure the convergence of the infinite product
  $\prod_j \Phi_{\sD,j}$  as a rigid meromorphic function. 
  Numerical calculations indeed reveal that this product does not converge, even for 
 $p=5$, the smallest prime that is split in $\Q(i)$.
  The Borcherds theory developed in the body of this paper suggests that the obstruction to this product converging as a rigid meromorphic function on the $5$-adic analytic space
 $X_5$ lies in the space 
 $M_2(20)$ of modular forms of weight $2$ and level $20$.
In addition to the weight two Eisenstein series, this space contains a unique normalised cusp form
 \begin{align*}
g(z) &=  \eta(2z)^{2}\eta(10z)^{2} \\
&= q\prod_{n=1}^\infty(1 - q^{2n})^{2}(1 - q^{10n})^{2}\\
&= q - 2q^{3} - q^{5} + 2q^{7} + q^{9} + 2q^{13} + 2q^{15} - 6q^{17} - 4q^{19} - 4q^{21} + \cdots
\end{align*}
The fact that $2a_3(g) - a_7(g) =-6 \ne 0$ explains why 
the product of the $\Phi_{\sD,j}$ fails to converge for this
choice of $\sD$.
Inspection reveals that
  the functional 
  $$f \mapsto  a_3(f) -a_6(f) +  a_7(f)$$
   vanishes identically on  the cusp form $g$, and on the weight two Eisenstein series, since the coefficients 
  $b_n$ of  \eqref{eqn:coeffsEis24}
  satisfy
  $$ b_3 = 4, \qquad b_6 = 12, \qquad a_7 = 8.
  $$
 This motivates setting
  $$ \sD:=  \Delta_3 - \Delta_6 + \Delta_7,$$
  and studying the rational functions 
  $$\Phi_{\sD,j} := \Phi_{3,j} \div \Phi_{6,j}\times \Phi_{7,j},$$
  as well as 
   the infinite product 
 $$ \Phi_{\sD}(\tau_1,\tau_2) := \prod_{j=0}^\infty \Phi_{\sD,j}(\tau_1,\tau_2),$$
which should - as the the $p$-adic Borcherds theory suggests - converge uniformly to a rigid meromorphic function on each affinoid subset of $\cH_p\times \cH_p$.
 
 The function $ [ \alpha; b,c] \mapsto \chi_4(c)$ is invariant under the action of $\Gamma$, and hence there is  a unique
$\rM^\times$-valued modular symbol $J_\sD$ satisfying
 $$ J_{\sD}(0,\infty) = \Phi_{\sD}, \qquad
 J_\sD(\gamma r, \gamma s)(\gamma \tau_1,\gamma \tau_2)  = J_{\sD}(r,s)(\tau_1,\tau_2), \mbox{ for all } \gamma \in \Gamma.$$
This rigid meromorphic cocycle  (or rather, its associated rigid meromorphic period function $\Phi_{\sD}$)
was calculated on the computer to an accuracy of $160$
significant $5$-adic digits, using the same kind of 
iterative algorithms that are described in 
\cite{DV1} in signature $(2,1)$ and in \cite{negrini2} in signature $(4,0)$. 
The calculation took around a week on a standard machine. The $5$-adic rigid meromorphic period function $\Phi_{\sD}$ is eventually stored as a $6\times 6$ array of power series in two variables of degree $\leq 160$  with coefficients in $\Z/5^{160}\Z$. Once it has been computed and recorded in a file, evaluating the associated rigid meromorphic cocycle $J_\sD$ at a special point typically requires only a few seconds.

\medskip
Special points on $X_p= \cH_p\times \cH_p$ can be divided into three types: 
\begin{enumerate}[(a)]
\item The {\em small RM points} coming from an embedding of  the multiplicative group of an order in a real quadratic field into 
$\SL_2(\Z)$, composed with the natural  inclusion $\SL_2(\Z) \subseteq \SL_2(\cO_K)$. These special points are of the form $(\tau,\tau)$, where $\tau$ is an RM point in $\cH_p$. The reflex field of such a points is the real quadratic field $\Q(\tau)$.
\item The {\em small CM points}, which are of the form
$(\tau,\tau')$ where $\tau$ is an RM point in $\cH_p$ and $\tau'$ is its Galois conjugate.
The reflex field of such a point is the imaginary  quadratic subfield of the biquadratic field $\Q(\tau,i)$ which differs from $\Q(i)$.
\item The {\em big special points} whose reflex fields are ATR extensions of real quadratic fields, that is, quartic extensions of $\Q$ with exactly one complex place.
\end{enumerate}

We discuss the three classes in turn.

\medskip\noindent
{\em Evaluations at small $RM$ points}.
Adopt the shorthand 
$$ J_{\sD}[\tau] := J_{\sD}[\tau,\tau], \qquad \tau\in \cH_p^{\rm}.$$
The following  identities   give a good illustration of Conjecture \ref{conj:main}:
\begin{eqnarray}
\label{eqn:type1-2}
 J_{\sD}\left[\frac{1}{\sqrt{2}} \right] &=& 
\frac{ -289 + 480 i -204\sqrt{2} + 340\sqrt{-2}}{3 \cdot 11}  \bmod{5^{160}}, \\
\label{eqn:type1-3}
  J_{\sD}\left[\frac{1+\sqrt{3}}{2}\right] 
 &=& \frac{ -329 + 96i + 
  188\sqrt{3} - 56\sqrt{-3}}{7^2}  \bmod{5^{160}}, \\
  \nonumber
J_{\sD}\left[  \frac{1+\sqrt{17}}{4}\right] &{=}&
 \frac{ a + b\cdot i + c \cdot \sqrt{17} + d\cdot\sqrt{-17}}
 {3 \cdot 5^3 \cdot 7^2 \cdot 17 \cdot 23}  \bmod{5^{160}},
 \end{eqnarray}
 with 
\begin{equation}
   \label{eqn:type1-17}
   (a,b,c,d) \ =\  (8561065121,  - 13089950772, 
  -2076362976, \ 3174779132).
  \end{equation}
  The  algebraic numbers on the right-hand side of these
  putative evaluations all
  belong to an unramified  abelian extension 
  of the reflex field $\Q(\tau)$,
and    have norm $1$ to this reflex field. More precisely, they are   in the minus part (multiplicatively) for the action of complex conjugation which lies in the center
of  the absolute Galois group of 
$\Q(\tau)$.  In particular, they lie in the unit circle relative to all complex embeddings of $\overline\Q$.
It is also worth noting    that  the  primes 
that arise in the factorisation of $J_{\sD}[\tau]$ lie above rational primes that are 
either inert  or ramified in the reflex field.
These observations evoke patterns that were already observed in  \cite{DV1}, and   this is no coincidence.
  Indeed, the points $(\tau,\tau)$ are the images of the  RM point $\tau\in
\cH_5$ under   the diagonal embedding $\cH_5\subseteq \cH_5\times \cH_5$, and  the restriction of $J_{\sD}$ to the diagonal gives rise to a rigid meromorphic
cocycle for a subgroup of $\SL_2(\Z[1/p])$, namely, the intersection of $\Gamma$ with $\SL_2(\Q)$,  i.e., 
the $\Gamma_0(2)$-type congruence subgroup
of $\SL_2(\Z[1/p])$.
Hence, the evaluations 
\eqref{eqn:type1-2}, 
\eqref{eqn:type1-3},
and
\eqref{eqn:type1-17}  do
 not   represent a truly new verification of Conjecture
 \ref{conj:main}  for a Bianchi cocycle, since they can be reduced to the conjectures of \cite{DV1}.

 \medskip\noindent
{\em Evaluations at small $CM$ points}.
Values at small CM points of the form
 $(\tau,\tau')$ where $\tau$ is an RM point
 in $ \cH_5$
---  for which  the $p$-adic ``CM type"  
  has been flipped,  so that the  associated 
  reflex field   is now an imaginary quadratic field ---
 have no   counterpart in the setting of rigid meromorphic cocycles for $O(2,1)$,  and  are thus significantly  more interesting 
a priori.
The first numerical evaluations reveal that
\begin{eqnarray}
\label{eqn:type2-2}
 J_{\sD}\left[\frac{1}{\sqrt{2}},\frac{-1}{\sqrt{2}}\right]
 &=& 1
 \bmod{5^{160}}, \\
 \label{eqn:type2-3}
 J_{\sD}\left[\frac{1+\sqrt{3}}{2}, \frac{1-\sqrt{3}}{2}\right] 
 &=& 
1 \bmod{5^{160}}, \\ 
\label{eqn:type2-17}
 J_{\sD}\left[ \frac{1+\sqrt{17}}{4}, \frac{1-\sqrt{17}}{4}\right] &=&
  \frac{  2^2   \cdot (4-i)^2 \cdot i }{5^3} \bmod{5^{160}}.\end{eqnarray}
Comparing with \eqref{eqn:type1-2},
\eqref{eqn:type1-3} and \eqref{eqn:type1-17},
 it seems that these more exotic values tend to be of
 significantly smaller  height. This is confirmed by further experiments.  Letting  $\tau_D\in \cH_5$  be an RM point of discriminant $D>0$, it was observed that
 $$ J_{\sD}[\tau_D,\tau_D'] \stackrel{?}{=} 1, \mbox { for } D =\begin{array}{l}
  32, 88, 92, 152, 168, 172, 188, 232, 248,  
  252,\\ 268, 272, 308, 312, 328, 332, 348,  368, 
  $\ldots$
  \end{array}$$
  Another novel phenomenon is the relatively frequent occurrence of instances where $J_{\sD}[\tau_D,\tau_D']$ cannot be evaluated because the point $(\tau_D,\tau_D')$ occurs among the poles and zeroes of the rational functions in the infinite product defining $J_{\sD}\{r,s\}$ with $r,s\in {\mathbb P}_1(\Q(i))$. 
 This happens for
 $$ D = 33, 52, 73, 97, 113, 137, 177, 193, 208, 212, 217, 228, 233, 288, 292, 313, 337,  \ldots,$$
 
 By contrast, the diagonal restriction of $J_{\sD}\{r,s\}$ has zeroes and poles on a collection of   RM points with finitely many discriminants, and it is expected that 
 $J_{\sD}[\tau,\tau]=1$  occurs only in extremely rare instances as well.
In spite of their relative paucity, interesting 
algebraic values at small CM points do   arise.  For instance,  numerical  evaluations  of 
 $$J_{9\cdot 17} := J_{\sD}\left[ \frac{1+3 \sqrt{17}}{4}, \frac{1-3 \sqrt{17}}{4}\right], \qquad
 J_{37} := 
J_\sD \left[\frac{1+\sqrt{37}}{4}, \frac{1-\sqrt{37}}{4}\right],
  $$
 to $160$ digits of $5$-adic precision suggest  that 
 \begin{eqnarray}
 \label{eqn:J917} &&
  J_{9\cdot 17}   \stackrel{?}{=} 
  \frac{ 19^2  \cdot (1+i)^8 \cdot (1-4i)^5  \cdot (5+2i)^3 \cdot (5+4i) \cdot (5+6i)^2 }
  {(1+2i)^5 \cdot (1-2i)^6  \cdot (1-6i)^2 \cdot (9-4i)^2 \cdot (7-8i)^2\cdot (2-13i)^2}, \\  \nonumber & & \\
   \label{eqn:J37} &&
   J_{37} \stackrel{?}{=}\frac{     (2-i)^4 \cdot (3-2i) \cdot (2+5i)^2 \cdot (1+6i) \cdot (5+8i)^2 \cdot (8+7i) \cdot (12+7i) }{3^5 \cdot  7^4 \cdot (2+i)^2 \cdot 
 (5+6i) \cdot (3+10i)}.
  \end{eqnarray}
  Unlike what happens for the small $RM$ values, these
  algebraic invariants have non-trivial norms to 
  $\Q$:
\begin{eqnarray}
\label{eqn:norm917}
 {\rm Norm}(J_{9\cdot 17}) &\stackrel{?}{=} &
   \frac{  2^8 \cdot 17^5  \cdot 19^4  \cdot 29^3 \cdot 41 \cdot 61^2 }
  {5^{11}    \cdot 37^2 \cdot 97^2 \cdot 113^2\cdot 173^2}, \\ \nonumber & & \\
  \label{eqn:norm37}
  {\rm Norm}(J_{37})   &\stackrel{?}{=} &
  \frac{ 5^2 \cdot 13 \cdot 29^2 \cdot 37 \cdot 89^2 \cdot 113 \cdot 193 }{3^{10} \cdot  7^{8} \cdot   
61 \cdot 109}.
  \end{eqnarray}
 Observe that the  primes  that appear in the right hand sides of 
\eqref{eqn:norm917} and
\eqref{eqn:norm37}
 are all ramified
 or inert  in the  respective reflex fields $\Q(\sqrt{-17})$ and $\Q(\sqrt{-37})$.

 The  numerical evaluations in
 \eqref{eqn:type2-17}, 
 \eqref{eqn:J917} and 
 \eqref{eqn:J37}   
  provide the first substantial 
 piece of experimental evidence  for  conjecture
 \ref{conj:main}  in the    case of   Bianchi cocycles,   beyond what directly follows from this conjecture  in the signature $(2,1)$ case.
  
  \medskip
  \medskip\noindent
  {\em  Evaluations at big special points}.
 
 Non-trivial values of $J_{\sD}$ at big special points are expected to be defined over fields of fairly large degree  (abelian extensions of quadratic extensions of real quadratic  fields with a single complex place) and to be of  sizeable height. A straightforward recognition of these values  seems to lie somewhat  beyond the $5$-adic precision
    for $J_{\sD}$ that the authors were able to calculate.
 A more systematic and sophisticated approach to the computer calculation  of rigid meromorphic Bianchi cocycles is clearly required for these more ambitious tests. 
 The second author plans to turn to this question in  an ongoing project with Xavier Guitart and Marc Masdeu. 
  
\bibliographystyle{abbrv}
\bibliography{bibfile}

\end{document}